\documentclass[11pt,a4paper]{amsart}
\usepackage[margin = 1in]{geometry}
\usepackage[utf8]{inputenc}
\usepackage{setspace}
\usepackage{amssymb}
\usepackage{amsmath,amsthm,amssymb,latexsym,amsfonts, mathtools,bbm, mathrsfs}
\usepackage{ stmaryrd }
\usepackage{faktor}
\usepackage{layout}
\usepackage{textcomp}
\usepackage{mathrsfs}
\usepackage{imakeidx}
\makeindex[columns=2, title=Index of Notation, name = notation]
\usepackage{bold-extra}
\usepackage[shortlabels]{enumitem}

\usepackage{hyperref}
\usepackage{cleveref}

\usepackage[dvipsnames]{xcolor}

\hypersetup{
  colorlinks   = true, 
  urlcolor     = blue, 
  linkcolor    = RedViolet, 
  citecolor   =  Emerald 
}

\usepackage{tikz-cd}

\theoremstyle{definition}
\newtheorem{definition}{Definition}[subsection]
\theoremstyle{plain}
\newtheorem{proposition}[definition]{Proposition}
\theoremstyle{plain}
\newtheorem{theorem}[definition]{Theorem}
\theoremstyle{plain}
\newtheorem{lemma}[definition]{Lemma}
\theoremstyle{plain}
\newtheorem{corollary}[definition]{Corollary}
\theoremstyle{plain}

\theoremstyle{definition}
\newtheorem{lemma*}{Lemma}
\theoremstyle{definition}
\newtheorem{remark}[definition]{Remark}
\theoremstyle{definition}

\theoremstyle{definition}

\theoremstyle{definition}

\theoremstyle{definition}
\newtheorem{notation}[definition]{Notation}
\theoremstyle{definition}
\newtheorem{convention}[definition]{Convention}
\theoremstyle{plain}
\newtheorem*{theorem*}{Theorem}

\theoremstyle{plain}
\newtheorem{thm-intro}{Theorem}

\newcommand{\C}{\mathbb{C}}
\newcommand{\T}{\mathbb{T}}
\newcommand{\Cat}{\mathcal{C}}

\newcommand{\afr}{\mathfrak{a}}
\newcommand{\Afr}{\mathfrak{A}}
\newcommand{\bfr}{\mathfrak{b}}
\newcommand{\cfr}{\mathfrak{c}}
\newcommand{\ffr}{\mathfrak{f}}

\newcommand{\nfr}{\mathfrak{n}}

\newcommand{\D}{\mathcal{D}}
\newcommand{\dfr}{\mathfrak{d}}
\newcommand{\A}{\mathcal{A}}

\newcommand{\B}{\mathcal{B}}

\newcommand{\Q}{\mathbb{Q}}
\newcommand{\R}{\mathbb{R}}
\newcommand{\Z}{\mathbb{Z}}
\newcommand{\G}{\mathbb{G}}

\newcommand{\Mcal}{\mathcal{M}}

\newcommand{\Po}{\mathscr{P}}

\newcommand{\Fscr}{\mathscr{F}}
\newcommand{\Gscr}{\mathscr{G}}
\newcommand{\Hscr}{\mathscr{H}}
\newcommand{\Iscr}{\mathscr{I}}

\newcommand{\Ucal}{\mathcal{U}}

\newcommand{\Hcal}{\mathcal{H}}

\newcommand{\Ocal}{\mathcal{O}}
\newcommand{\Ofr}{\mathfrak{o}}
\newcommand{\Hom}{\textnormal{Hom}}
\newcommand{\Homsh}{\underline{\textnormal{Hom}}}

\newcommand{\TSym}{\textnormal{TSym}}

\newcommand{\SL}{\textnormal{SL}}
\newcommand{\ts}{\textnormal{t}}
\newcommand{\st}{\textnormal{s}}
\newcommand{\id}{\textnormal{id}}
\newcommand{\Spec}{\textnormal{Spec}}
\newcommand{\Spf}{\textnormal{Spf}}
\newcommand{\arith}{\textnormal{arith}}

\newcommand{\mom}[1]{\prescript{}{#1}{\textnormal{mom}}}
\newcommand{\coPo}{\widehat{\mathscr{P}}}

\title[Eisenstein-Kronecker classes on Hilbert moduli spaces]{Eisenstein-Kronecker classes on Hilbert moduli spaces and a new construction of Katz's $p$-adic measure}
\author{Guillermo Gamarra Segovia}

\numberwithin{equation}{section}

\begin{document}

\begin{abstract}

	In this work, we extend the construction of the Eisenstein-Kronecker classes of Kings-Sprang to the universal abelian schemes lying over Hilbert moduli spaces parametrizing objects with $(\Gamma(l), \Gamma_{00}(p^\infty))$-level. We then compare the associated ($p$-adic) Hilbert modular forms to the Eisenstein series constructed by Katz. Similarly to the procedure of Kings-Sprang, we produce a $p$-adic measure which interpolates the Eisenstein-Kronecker classes, and use the mentioned comparison results to show that this recovers essentially the $p$-adic Eisenstein measure constructed by Katz.

\end{abstract}

\maketitle

\begin{footnotesize}
\noindent
2020 \textit{Mathematics Subject classification.} Primary 11G18, 14K22; Secondary 11F41, 11F75, 11G15.\\
\textit{Keywords:} Eisenstein-Kronecker classes; Hilbert modular forms; Hilbert modular schemes; $p$-adic measures.
\end{footnotesize}

\tableofcontents
\section*{Introduction}

\subsubsection*{Critical values of Hecke \textit{L}-functions}

Let $\chi$ be a Hecke character associated to a number field $F$. There is an extense volume of work on the study of the special values of its $L$-function, $L(\chi, s)$. In particular, let us discuss the \emph{critical} case, in which the $\Gamma$-factors in the functional equation of this function have no pole at $s = 0$. It is well-known that in this situation, $F$ must either be a totally real field or contain a CM field $L$. Klingen \cite{Klingen19611962} and Siegel \cite{Siegel1970} showed that whenever $\chi$ is the product of the field norm and a generalized Dirichlet character, these critical $L$-values lie in the number field generated by the values of this latter character, and are thus algebraic. 

In the case where $F = L$ is a quadratic imaginary field, the first results in this direction are due to Eisenstein \cite{Weil1999} and Damerell \cite{Damerell1970}. As it turns out, one may express the values $L(\chi, 0)$ as a finite sum of \emph{Eisenstein series}:
$$E_{a,b}(z;\Lambda) := \sideset{}{'}\sum_{\lambda \in \Lambda} \frac{\overline{(\lambda + z)^a}}{(\lambda + z)^b},$$
where $\Lambda \subset \C$ is a lattice and $b > a \geq 0$ are integers. In fact, the Eisenstein series appearing in this expression are evaluated at lattices associated to elliptic curves with \emph{complex multiplication by $L$}, from which one can recover the algebraicity of these $L$-values.

Assuming that $F = L$ is a general CM field, Shimura \cite{Shimura1975} showed that the critical values $L(\chi, 0)$ are algebraic up to taking the quotient by certain periods. Later on, Katz \cite{Katz1978} constructed a $p$-adic measure on the Galois group $\text{Gal}(L(p^\infty)/L)$ which interpolated these values up to some explicit constants (here $L(p^\infty)$ stands for the maximal unramified outside of $p$ extension of $L$). For certain extensions of $L$ quadratic imaginary, algebraicity was shown by Colmez \cite{Colmez1989}, and proven for all finite extensions by Bergeron-Charollois-García \cite{Bergeron2023}.

Algebraicity in the general case in which $F$ is a finite extension of a CM field $L$ was proved by Kings-Sprang \cite{Kings2025}, generalizing the results of Shimura and Katz. Due to their relevance to the present work, let us recall briefly the mentioned work of Katz and Kings-Sprang. 

\subsubsection*{Katz's Eisenstein measure}

Another interesting area in connection with the study of $L$-functions is the construction of $p$-adic analogues of them. This started with the work of Kubota-Leopoldt \cite{Kubota1964}, who constructed $p$-adic $L$-functions associated to Dirichlet characters, and was furthered by Barsky \cite{Barsky19771978} and Cassou-Nougès \cite{CassouNogues1979} to Hecke $L$-functions for totally real fields, based on a formula of Shintani \cite{Shintani1976} for the decomposition of these $L$-functions. 

Ideally, one can construct such $p$-adic $L$-functions via \emph{$p$-adic measures}, which can encode the algebraicity and congruence properties of the $L$-values they interpolate. Such a construction of a $p$-adic measure was done in the work of Deligne-Ribet \cite{Deligne1980}, which was shown by Ribet \cite{Ribet1979} to agree with that of Barsky and Cassou-Nougès. It is this approach that Katz takes in \cite{Katz1978}, where he constructs $p$-adic measures interpolating a generalization of the classical Eisenstein series $E_{a,b}(z;\Lambda)$ as \emph{Hilbert modular forms}.

More concretely, fix a totally real number field $F$, and write $\Ofr$ for its ring of integers. For certain locally constant functions $H \colon (\Ofr \otimes_\Z \Z_p)^2 \to \overline{\Q}$, Katz defines Hilbert modular forms $G_{k, H}$ of parallel weight $k$ such that over $\C$, their transcendental expresion is of the form
$$\sideset{}{'}\sum_{\lambda \in \Ofr^\times \backslash PV_p(\Lambda) \cap \Lambda[p^{-1}]} \frac{PH([\lambda])}{\lambda^k|N(\lambda)|^{2s}}\Bigg|_{s=0},$$
where $PV_p(\Lambda)$ is a certain submodule of $\Lambda \otimes_\Z \Q_p$ and $PH$ is the partial Fourier transform of $H$ (which can be seen as a map on $PV_p(\Lambda)$). Katz then takes the limit of these modular forms to obtain a $p$-adic version $\widehat{G}_{k, H}$.

From these Hilbert modular forms, Katz constructs a measure 
$$\mu_{KE} \in \text{Meas}((\Ofr \otimes_\Z \Z_p)^{\times, 2}/\overline{\Ofr^\times}, V),$$
where $V$ is the ring of $p$-adic Hilbert modular forms over $\Ofr_{\C_p}$, the valuation ring of $\C_p$. One may do this by the formula
$$\int_{(\Ofr \otimes_\Z \Z_p)^{\times, 2}/\overline{\Ofr^\times}} F d\mu_{KE} = \widehat{G}_{1, H},$$
where $F$ is a continuous function on the profinite group and $H(x,y) := N(x)^{-1}F(x^{-1}, y)$.

We note that it is quite straightforward from these definitions to extend the measure $\mu_{KE}$ to measures $\mu_{\Gamma, KE}$ on the quotients $(\Ofr \otimes_\Z \Z_p)^{\times, 2}/\overline{\Gamma}$ for any finite index subgroup $\Gamma \leq \Ofr^\times$. In fact, if $l$ is a positive integer coprime to $p$ and such that $\Gamma$ acts trivially on $\Ofr/l\Ofr$, then the same construction of Katz gives place to a measure on the profinite group
$$((\Ofr \otimes_\Z \Z_p)^{\times, 2}/\overline{\Gamma}) \times (\Ofr/l\Ofr)^2.$$
In particular, for any map of sets $f \colon (\Ofr/l\Ofr)^2 \to \Ofr_{\C_p}$, we obtain a measure 
$$\mu_{\Gamma, KE}(f) \in \text{Meas}((\Ofr \otimes_\Z \Z_p)^{\times, 2}/\overline{\Gamma}, V).$$
In this work, we provide a different construction for these measures through the machinery of Eisenstein-Kronecker classes, which arise from the work of Kings-Sprang.

\subsubsection*{Eisenstein-Kronecker classes}

In \cite{Bannai2010}, Bannai-Kobayashi construct a generating function for the Eisenstein-Kronecker numbers associated to a lattice $\Lambda \in \C$ by using theta functions associated to the Poincaré bundle $\Po$ of the elliptic curve $E$ corresponding to $\Lambda$. Since these numbers are related to critical Hecke $L$-values associated to quadratic imaginary fields, Bannai-Kobayashi are able to show their algebraicity whenever $E$ has CM by such a quadratic imaginary field. This approach was used later on by Sprang \cite{Sprang2019} in order to show that certain algebraic cohomology classes of $\Po$ recover the real analytic Eisenstein series after applying the Hodge decomposition. Building further on this work, Kings-Sprang construct \emph{Eisenstein-Kronecker classes} in the equivariant cohomology of abelian schemes with complex multiplication in order to show the integrality of the values
$$\frac{(\chi(\cfr)N\cfr - 1)L(\chi, 0)}{\Omega^\chi},$$
where $\chi$ is a general critical Hecke character associated to a totally imaginary field $L$, $\cfr$ is an integral ideal of $L$ coprime to the conductor of $\chi$, and $\Omega^\chi$ is an explicit product of powers of $2\pi i$ and of periods of abelian varieties with CM by $L$. 

These Eisenstein-Kronecker classes are certain cohomology classes in the equivariant cohomology of an abelian scheme $\A/S$:
$$EK_{\Gamma, \A}(f) \in H^{g-1}(\A \setminus \D, \Gamma ; \widehat{\Po} \otimes \Omega^g_{\A/S}),$$
where $g$ is the relative dimension of $\A$, $\Gamma$ is a group acting on $\A$, $\D$ is a closed, $\Gamma$-equivariant subscheme of $\A$, $f$ is a $\Gamma$-equivariant function on $\D$, and $\coPo$ is the formal completion of the Poincaré bundle along the inclusion map $\A \to \A \times_S \A^\vee$. These classes can be further specialized by assuming that $S$ is affine and that $\A$ has CM by a totally imaginary field $L$, as well as by choosing a torsion point $x \colon S \to \A \setminus \D$. In this case, one obtains
$$EK^{\beta, \alpha}_{\Gamma, \A}(f,x) \in H^0(S, \TSym^{\alpha + \underline{1}}(\omega_{\A/S}) \otimes \TSym^\beta(H^1_{dR}(\A^\vee/S))).$$
Here $\alpha, \beta \in \prod_{\sigma} \Z$ are certain tuples of integers indexed by the complex embeddings of $L$ such that $\alpha + \underline{1} - \beta$ defines a critical infinity type for $L$. 

In relation to the main results of this work, we place particular emphasis in the following two facts, proven by Kings-Sprang, regarding these Eisenstein-Kronecker classes.  

\begin{enumerate}

	\item If $S = \Spec(\C)$, the $EK^{\beta, \alpha}_{\Gamma, \A}(f,x)$ can be computed to be equal to sums of real analytic Eisenstein series of the form
	$$\sum_{\lambda \in \Gamma \backslash \Lambda} \frac{\overline{(\lambda + x)}^\beta}{(\lambda + x)^{\alpha + \underline{1}}|N(\lambda + x)|^{2s}} \Bigg|_{s=0},$$
	where $\Lambda$ is the complex lattice associated to $\A(\C)$.

	\item If $S = \Spec(\Ofr_{\C_p})$, and if we impose an infinite level structure on $\A$, starting from the classes $EK_{\Gamma, \A}(f)$ one can construct a $p$-adic measure
	$$\mu_{\Gamma, \A}(f,x) \in \text{Meas}((\Ofr \otimes_\Z \Z_p)^2, \Ofr_{\C_p})_\Gamma$$
	whose moments interpolate the $EK^{\beta, \alpha}_{\Gamma, \A}(f,x)$.

\end{enumerate}

\subsubsection*{Main results}

Consider again a totally real field $F$. In this work, we give a generalization of the results of Kings-Sprang to the Hilbert modular case. First, we show that we can construct the Eisenstein-Kronecker classes $EK^{\beta, \alpha}_{\Gamma, \A}(f,x)$ in the case where $S$ is not affine and $\A$ only has real multiplication by $F$, which can be done without changing the approach of \cite{Kings2025} very significantly.

Our particular interest is the case when $S$ is equal to $\Mcal(\cfr, \Gamma(l), \Gamma_{00}(p^\infty))$, the Hilbert moduli space parametrizing abelian schemes with RM by $F$, a $\cfr$-polarization and $(\Gamma(l), \Gamma_{00}(p^\infty))$-level structure. In such a situation, we have a universal abelian scheme $\A \to \Mcal$, over which we can construct the $EK^{\beta, \alpha}_{\Gamma, \A}(f,x)$. By taking advantage of the Hodge decomposition, respectively the unit root decomposition, we define \emph{specialization maps}
$$\begin{tikzcd}[column sep = 0.005cm]
& \{ \Cat^\infty\text{-Hilbert modular forms over } \C \} \\
H^0(\Mcal, \TSym^{\alpha + \underline{1}}(\omega_{\A/\Mcal}) \otimes \TSym^\beta(H^1_{dR}(\A/\Mcal))) \arrow{ur}{\varphi_\infty} \arrow[swap]{dr}{\varphi_p} & \\
 & \{ p\text{-adic Hilbert modular forms over } \Ofr_{\C_p} \}
\end{tikzcd}$$
In other words, we can understand our algebraic Eisenstein-Kronecker classes as either $\cfr$-Hilbert modular forms over the complex numbers which are smooth (so not necessarily holomorphic), or as $p$-adic $\cfr$-Hilbert modular forms over $\Ofr_{\C_p}$.

By using the mentioned result (1) of Kings-Sprang, we can relate the complex specialization $\varphi_\infty(EK^{\beta, \alpha}_{\Gamma, \A}(f,x))$ to the Eisenstein series $G_{k, H}$ constructed by Katz. Moreover, we show that the construction (2) of such a $p$-adic measure still holds in the universal case, and now will define a measure which takes values in the ring $V(\cfr, \Gamma(l), \Ofr_{\C_p})$ of $p$-adic $\cfr$-Hilbert modular forms over $\Ofr_{\C_p}$ with $\Gamma(l)$-level.

\begin{thm-intro}[\Cref{measures agree}]\label{intro:main thm 1}

	Let $\Gamma$ be a finite index subgroup of $\Ofr^\times$, and write $G := (\Ofr \otimes_\Z \Z_p)^{\times, 2}$ as well as $H(\Gamma)$ for the closure of $\Gamma$ inside $G$, where we embed $\Gamma$ via $\gamma \mapsto (\gamma^{-1}, \gamma)$. Assume that $\Gamma$ acts trivially on $\Ofr/l\Ofr$ for some positive integer $l$ coprime to $p$. Then, if $f \colon (\Ofr/l\Ofr)^2 \to \overline{\Q}$ is a function, and $x$ is a $\Gamma$-equivariant torsion point of $\A$ for some level coprime to $l$ and $p$, there exists a $p$-adic measure
	$$\mu_{\Gamma, \A}(f,x) \in \textnormal{Meas}(G/H(\Gamma), V(\cfr, \Gamma(l), \Ofr_{\C_p}))$$
	such that evaluating it at CM points recovers the $p$-adic measure of Kings-Sprang up to a twist by the norm. Moreover, we have an equality
	$$\frac{(-1)^{\frac{g(g+1)}{2}}}{\sqrt{d_F}} \cdot \mu_{\Gamma, \A}(f, e) = (\textnormal{inv} \times \textnormal{id})_*\mu_{\Gamma, KE}(f),$$
	where $d_F$ is the discriminant of $F$, $e$ is the zero section of $\A$, and $\textnormal{inv} \colon (\Ofr \otimes_\Z \Z_p)^\times \to (\Ofr \otimes_\Z \Z_p)^\times$ is the inversion map $x \mapsto x^{-1}$.

\end{thm-intro}

Still, this result does not provide a priori a comparison to the measure initially considered in \cite{Katz1978}, since that measure is of the form $\mu_{\Ofr^\times, KE}(f)$ with $f$ the characteristic function of $([0], [0])$, which clearly does not satisfy the zero sum property. As it turns out, in order to find such a comparison we may drop the $\Gamma(l)$-level assumption and consider instead a specific choice of $f$.

\begin{thm-intro}[\Cref{measures agree no full}]\label{intro:main thm 2}

	Set $\Gamma$, $G$ and $H(\Gamma)$ as above, and assume that $l = 1$. Consider two positive integers $b_1$, $b_2$ which are coprime to each other and to $p$, and such that $b_1\Ofr$ and $b_2\Ofr$ are not equal to $\Ofr$. Then, there exists a function $\widetilde{f}_{b_1, b_2}$ on $\A[b_1b_2] \setminus \{ e \}$ such that
	$$\frac{(-1)^{\frac{g(g+1)}{2}}}{\sqrt{d_F}} \cdot \mu_{\Gamma, \A}(\widetilde{f}_{b_1, b_2},e) = b_1^g(1 - [b_1]_*)(b_2^g[b_2]_* - 1)(\textnormal{inv} \times \textnormal{id})_*\mu_{\Gamma, KE},$$
	where $[m] \colon (\Ofr \otimes_\Z \Z_p)^2 \to (\Ofr \otimes_\Z \Z_p)^2$ is the map $(x,y) \mapsto (mx, y/m)$ for any integer $m$ coprime to $p$.

\end{thm-intro}

The importance behind these alternative constructions is that it is of an \emph{algebraic} nature, as we start with cohomology classes associated to certain abelian varieties. This is in contrast with the approach by Katz, which requires first the construction of the Eisenstein series $G_{k, H}$ in order to define $\mu_{KE}$ from its moments.  

\subsubsection*{Overview} 

In the first section, we set some notation and recall useful facts regarding abelian schemes with real multiplication as well as Hilbert moduli spaces and modular forms. More importantly, we also recall the construction of Katz of his Eisenstein measure, which we present here in its slightly more general version.

The second section starts with a retelling of the steps done by Kings-Sprang to construct the Eisenstein-Kronecker classes, as well as a discussion about how this procedure works in the case where the base scheme is non-affine and $\A$ only has real multiplication. We then discuss how the smooth splitting of the Hodge decomposition as well as the unit root decomposition give way to the complex and $p$-adic specialization respectively, and how to compare the specializations of the Eisenstein-Kronecker classes to the Eisenstein series of Katz.

Finally, in the third section we describe how the construction of a $p$-adic measure from the Eisenstein-Kronecker classes of Sprang can be used to define a universal measure taking values in $p$-adic Hilbert modular forms, and how the previous comparison results show that this measure agrees with the Eisenstein measure of Katz as described in the above \Cref{intro:main thm 1} and \Cref{intro:main thm 2}.

\subsubsection*{Acknowledgements}

The author would like to foremost thank Johannes Sprang for his advice and helpful insights through all of the making of this work. Thanks are due as well to Guido Kings for fruitful discussions during the author's visits to Regensburg. The majority of this work was part of the author's PhD thesis, and he thus wishes to thank the departments at the University of Duisburg-Essen and of Regensburg, who provided a great environment for working in. During this period, the author was supported by the SFB 1085 and wishes to extend his thanks to the DFG. The final details of this work were written as the author was in Keio University as a JSPS Postdoctoral fellow, and he thanks the JSPS as well as Kenichi Bannai for this opportunity. 

\section{Hilbert modular forms and the Eisenstein measure of Katz}

In this section, we recall the main facts about abelian schemes with real multiplication and their moduli spaces (the \emph{Hilbert moduli spaces}), as well as define Hilbert modular forms. Moreover, we discuss Katz's construction from \cite{Katz1978} of a $p$-adic measure which takes values in $p$-adic Hilbert modular forms. 

\subsection{Totally real fields and tensor symmetric products}

We start by fixing some notation. Let $F$ be a number field of degree $g$ with ring of integers $\Ofr$, and consider the set $J_F := \Hom_\Q(F, \R)$ of real embeddings of $F$. We fix an ordering of these embeddings, such that we can write $J_F = \{ \sigma_1, \ldots, \sigma_g\}$. An element $r \in F$ is said to be \emph{totally positive}, if $\sigma_i(r) > 0$ for all $i$ such that $\sigma_i$ is a real embedding. In this case, we write $r \gg 0$. For any fractional ideal $\afr$ of $F$, $\afr^+$\index[notation]{$a plus$@$\afr^+$} is the subset of totally positive elements of $\afr$.

Write as well $I_F := \Z[J_F]$\index[notation]{$I_F$}, the free abelian group on the real embeddings of $F$. Similarly, for any subset $\Sigma \subset J_F$, we write $I_\Sigma := \Z[\Sigma]$. If $\chi = \sum_{i = 1}^g \chi_i \sigma_i \in I_F$, we define
$$|\chi| := \sum_{i = 1}^g \chi_i,$$
which is always an integer.

\begin{definition}

	Let $\alpha = \sum_{i = 1}^g \alpha_i \sigma_i, \beta = \sum_{i = 1}^g \beta_i \sigma_i \in I_F$. 
	
	\begin{enumerate}[(i)]
	
		\item We say that $\alpha \geq \beta$ if $\alpha_i \geq \beta_i$ for all $i = 1, \ldots, g$. We write $I_F^+$ for the subset of all $\chi \in I_F$ such that $\chi \geq 0$, and similarly for $I_\Sigma^+$.
		
		\item We say that $\alpha$ is of \emph{norm type} if there exists a fixed $k \in \Z$ such that $\alpha_i = k$ for all $i$, and we write $\alpha = \underline{k}$.
		
	\end{enumerate}

\end{definition}

If we are considering modules over a ring which admit an action by the ring of integers $\Ofr_{F^{\text{Gal}}}$ of a Galois closure $F^{\text{Gal}}$ of $F$, and over which the discriminant $d_F$ of $F$ is invertible, then we have the following isomorphism between duals.

\begin{lemma}\label{duality of modules with F action}

	Let $R$ be an $\Ofr_{F^{\textnormal{Gal}}}[d_F^{-1}]$-algebra, and let $M$ be an $\Ofr \otimes_\Z R$-module. Then, the trace map $\textnormal{Tr}_{F/\Q}$ induces an isomorphism of $\Ofr \otimes_\Z R$-modules:
	$$\Hom_{\Ofr \otimes_\Z R}(M, \Ofr \otimes_\Z R) \cong \Hom_R(M,R), \quad \varphi \mapsto (\textnormal{Tr}_{F/\Q} \otimes \textnormal{id}) \circ \varphi$$

\end{lemma}

More generally, for any element $z = (z_1, \ldots, z_g) \in \C^g$, we define its \emph{norm} and \emph{trace} as
$$N(z) := \prod_{i = 1}^g z_i, \quad \text{Tr}(z) := \sum_{i = 1}^g z_i.$$
These clearly agree with the usual field norm and trace of $F$ under the embedding $F \hookrightarrow F \otimes_\Q \C \cong \C^g$ given by the elements of $J_F$.

For any fractional ideal, we want to discuss when it is ``prime'' to a ring in the following sense.

\begin{definition}\label{prime to R def}

	Let $\afr$ be a fractional ideal of $F$, and $R$ a ring. We say that $\afr$ is \emph{coprime to $R$} if $\afr$ is prime to every prime $\ell \in \Z$ which is not invertible in $R$.

\end{definition}

Note that if $\afr$ is coprime to $R$, one has an \emph{equality} $\afr \otimes_\Z R = \Ofr \otimes_\Z R$. 

\begin{notation}

	In what follows, $F$ will be a totally real field of degree $g \geq 2$ over $\Q$. We also write $\dfr^{-1}$ for the inverse different of $F$. For any fractional ideal $\cfr$, $\cfr^\vee := \cfr^{-1}\dfr^{-1}$ is its dual with respect to the field trace.

\end{notation}

Next, let us introduce an algebraic torus associated to the field $F$. We follow the exposition from \cite[Section 1]{Kings2025}.

\begin{definition}\label{torus def}\index[notation]{$T$@$\T_{F,S}$}

	Consider the algebraic torus $\T_F := \text{Res}_{\Ofr / \Z} \G_m$ given by Weil restriction. In other words, for any $\Z$-algebra $R$, $\T_F(R) := \G_m(R \otimes_\Z \Ofr) = (R \otimes_\Z \Ofr)^\times$. For any noetherian scheme $S$, we take the base change
	$$\T_{F,S} := \T_F \times_\Z S.$$
	We also consider the \emph{character group of $\T_{F,S}$} as the group of morphisms of $S$-group schemes
	$$I_{F,S} := \Hom_S(\T_{F,S}, \G_{m,S}).$$
	
\end{definition}

It is clear that whenever $S$ is a Spec$(\Ofr_{F^{\text{Gal}}}[d_F^{-1}])$-scheme, the torus $\T_{F,S}$ is split: if Spec$(R)$ is an $S$-scheme, then we have an isomorphism of $R$-algebras
$$\Ofr \otimes_\Z R \xrightarrow{\sim} \prod_{i = 1}^g R, \quad a \otimes 1 \mapsto (\sigma_1(a), \ldots, \sigma_g(a)).$$
In particular, we see that $I_{F,S} = I_F$. For the remainder of this section, we place ourselves in this situation.

\begin{definition}[{\cite[Exposé I, Definition 4.7.1]{SGA3}}]

	Let $\Fscr$ be a quasi-coherent $\Ocal_S$-module. We say that $\Fscr$ is an \emph{algebraic $\T_{F,S}$-module} if it has a $\Ocal_S[I_F]$-comodule structure. 

\end{definition}

The following result shows how we can decompose these algebraic $\T_{F,S}$-modules in terms of the characters of the torus.

\begin{proposition}[{\cite[Exposé I, Proposition 4.7.3]{SGA3}}]

	Let $\Fscr$ be a quasi-coherent algebraic $\T_{F,S}$-module. Then, there exists an isomorphism
	$$\Fscr \cong \bigoplus_{\chi \in I_F} \Fscr(\chi),$$
	where $\T_{F,S}$ acts on $\Fscr(\chi)$ via the character $\chi : \T_{F,S} \to \G_m$.

\end{proposition}

\begin{definition}

	The \emph{type} of an algebraic $\T_{F,S}$-module $\Fscr$ is the set of $\chi \in I_F$ such that $\Fscr(\chi) \neq 0$. 

\end{definition}

Next, we will adopt the following notation from \cite[Definition 1.6]{Kings2025} when considering tensor symmetric products of algebraic $\T_{F,S}$-modules.

\begin{notation}\label{TSym notation}

	Let $\Fscr$ be an algebraic $\T_{F,S}$-module. For any subset $\Sigma \subset J_F$ of the real embeddings of $F$, we write
	$$\Fscr(\Sigma) := \bigoplus_{\sigma \in \Sigma} \Fscr(\sigma),$$
	and for any character $\chi := \sum_{\sigma \in \Sigma} \chi_\sigma \cdot \sigma \in I^+_\Sigma$, we define 
	$$\TSym^\chi_{\Ocal_S}(\Fscr(\Sigma)) := \bigotimes_{\sigma \in \Sigma} \TSym^{\chi_\sigma}_{\Ocal_S}(\Fscr(\sigma)).$$
	By abuse of notation, we will also use these definitions for modules which are of the form $\bigoplus_{\sigma \in \Sigma} \Fscr(-\sigma)$. Lastly, for any tuple $s = (s_\sigma)_{\sigma \in \Sigma} \in \Fscr(\Sigma)$, we define
	$$s^{[\chi]} := \bigotimes_{\sigma \in \Sigma} s_\sigma^{[\chi_\sigma]},$$
	where $s_\sigma^{[\chi_\sigma]}$ is the tensor product of $s_\sigma$ with itself $\chi_\sigma$-times.

\end{notation}

\begin{remark}

	It is clear from these definitions that one has an isomorphism
	\begin{equation}\label{eq:TSym decomposition}
	\TSym^k_{\Ocal_S}(\Fscr(\Sigma)) \cong \bigoplus_{\substack{\chi \in I^+_\Sigma \\ |\chi| = k}} \TSym^\chi_{\Ocal_S}(\Fscr(\Sigma)).
	\end{equation}

\end{remark}

\begin{remark}\label{TSym char 0}

	In general, for any commutative ring $R$ and any \emph{invertible} $R$-module $M$, we have injective maps of $R$-modules 
	$$\TSym^n_R(M) \otimes_R \TSym^m_R(M) \hookrightarrow \TSym^{n+m}_R(M)$$
	for all $n, m \geq 0$ induced by the $R$-algebra structure of $\TSym^\bullet_R(M)$. In particular, using our previous notation, if $\Fscr$ is an algebraic $\T_{F,S}$-module and $\Sigma \subset J_F$ is such that the $\Fscr(\sigma)$ are invertible $\Ocal_S$-modules for all $\sigma \in \Sigma$, we obtain for any $\chi_1, \chi_2 \in I_\Sigma^+$ an injective map
	$$\TSym^{\chi_1}_{\Ocal_S}(\Fscr(\Sigma)) \otimes_{\Ocal_S} \TSym^{\chi_2}_{\Ocal_S}(\Fscr(\Sigma)) \hookrightarrow \TSym^{\chi_1 + \chi_2}_{\Ocal_S}(\Fscr(\Sigma)).$$
	Note that both of these maps become isomorphisms when working over a field of characteristic zero.
	
\end{remark}

Later on, we will study the action of finite index subgroups $\Gamma \leq \Ofr^\times$ on algebraic $\T_{F,S}$-modules. In particular, if $S$ is a flat scheme over Spec$(\Z)$, we have an embedding $\Ofr \subset \Ofr \otimes_\Z \Ocal_S$, and thus of the units $\Ofr^\times$ into $\T_{F,S}(R)$ for any algebra $R$ with a map $\text{Spec}(R) \to S$. Therefore, for any algebraic $\T_{F,S}$-module $\Fscr$ and any $\chi \in I_F$, the units $\gamma \in \Ofr^\times$ act on $\Fscr(\chi)$ by multiplication with 
$$\prod_{i = 1}^g \sigma_i(\gamma)^{\chi_i} \in R^\times.$$ 
The well-known fact that the infinity type of a Hecke character over a totally real field is always of norm type implies the following result.

\begin{lemma}\label{norm type implies vanishing}

	Let $\chi = \sum_{i = 1}^g \chi_i \sigma_i \in I_F$ and $\Gamma \leq \Ofr^\times$ a finite index subgroup. We have that
	$$\prod_{i = 1}^g \sigma_i(\gamma)^{\chi_i} \in \{ \pm 1 \}$$
	for all $\gamma \in \Gamma$ if and only if $\chi$ is of norm type.

\end{lemma}

\begin{remark}

	In particular, note that if $\chi = \underline{k}$ with $k$ an even integer, then the action of $\Gamma$ on $\Fscr(\underline{k})$ is always trivial for any choice of $\Gamma$.

\end{remark}

\begin{definition}

	We say that a character $\chi \in I_F$ is of \emph{critical type for $\Gamma$} if $\chi$ is such that $\prod_{i=1}^g \sigma_i(\gamma)^{\chi_i} = 1$ for all $\gamma \in \Gamma$. In particular, $\chi$ is always of norm type.

\end{definition}

In order to follow the notation and conventions of \cite{Kings2025}, we will make the following convention regarding duals of algebraic $\T_{F,S}$-modules.

\begin{convention}\label{signs of dual}

	Let $\Fscr$ be an algebraic $\T_{F,S}$-module. We consider the dual module $\Fscr^\vee := \Homsh_{\Ocal_S}(\Fscr, \Ocal_S)$ as an algebraic $\T_{F,S}$-module as well, by taking the \emph{contragradient representation}: for any sections $g$ of $\T_{F,S}$ and $f$ of $\Fscr^\vee$, $g \cdot f$ is defined by $s \mapsto f(g^{-1} \cdot s)$. In particular, this means that the type of $\Fscr^\vee$ is the set of characters $\mu$ such that $\mu^{-1}$ is in the type of $\Fscr$.

\end{convention}

\subsection{Abelian schemes}

Let $S$ be a noetherian scheme and $\A/S$ an abelian scheme, i.e.~a smooth, proper group scheme with geometrically connected fibers. Write
$$\pi \colon \A \to S, \quad e \colon S \to \A$$
for its structure map and its zero section, respectively. We assume that its relative dimension is $g$ (which we recall is also the degree of $F$ over $\Q$), and recall the following objects associated to $\A$. First, we write $\A^\vee/S$ for its dual abelian scheme (again of relative dimension $g$), and $\A^\natural/S$ for its universal vectorial extension. We set
$$\omega_{\A / S} := \pi_*\Omega^1_{\A/S} \cong e^*\Omega^1_{\A/S}$$
for the sheaf of translation invariant differential forms. Note that the Lie algebra of $\A$ is isomorphic to the dual of this sheaf: $\text{Lie}(\A/S) \cong \Homsh_{\Ocal_S}(\omega_{\A/S}, \Ocal_S)$. We will write, for any positive integer $i$,
$$\omega^i_{\A/S} := \bigwedge^i \omega_{\A/S} \quad \text{and} \quad \omega^{-i}_{\A/S} := \Homsh_{\Ocal_S}(\omega^i_{\A/S}, \Ocal_S) \cong \bigwedge^i \text{Lie}(\A / S)$$
for the exterior products (as $\Ocal_S$-modules) of these sheaves. 

We may also construct the first de Rham cohomology sheaf of $\A$, which is a quasi-coherent $\Ocal_S$-module $H^1_{dR}(\A/S)$. In particular, if we set $\Hscr_\A := H^1_{dR}(\A^\vee / S) \cong \Homsh_{\Ocal_S}(H^1_{dR}(\A/S), \Ocal_S)$\index[notation]{$H A$@$\Hscr_\A$}, the degeneration of the Hodge to de Rham spectral sequence gives the following short exact sequence:
\begin{equation}\label{eq:Hodge filtration ses}
\begin{tikzcd}
0 \arrow{r} & \omega_{\A^\vee / S} \arrow{r} & \Hscr_{\A} \arrow{r} & \text{Lie}(\A/S) \arrow{r} & 0.
\end{tikzcd}
\end{equation}

\subsubsection{Completions and the Poincaré bundle}

Let us recall several notions regarding the formal completion of group schemes, as well as the Poincaré bundle of an abelian scheme. We follow the notations and definitions of \cite[Section 2]{Kings2025} closely, and refer to it for more details. 

Fix an abelian scheme $\pi \colon \A \to S$ as above. Since the zero section $e$ is a closed immersion, we may take the formal completion of $\A$ along the image of $e$. The resulting formal scheme $\widehat{\A}/S$\index[notation]{$A compl$@$\widehat{\A}$} is then a formal group scheme obtained as the direct limit of the formal neighbourhoods
$$\A^{(n)} := \underline{\Spec}_{\Ocal_S}(\Ocal_\A / \Iscr^n),$$
where $\Iscr$ is the ideal sheaf associated to $e$. In particular, we write
$$\widehat{\pi} \colon \widehat{\A} \to S, \quad \pi^{(n)} \colon \A^{(n)} \to S$$
for the structure morphisms, as well as $\widehat{e}$, $e^{(n)}$ for the zero sections. We also remark that the coproduct in $\Ocal_{\widehat{\A}}$ induces a moment map
\begin{equation}\label{eq:moment map def}
\mom{}_{\widehat{\A}} \colon \Ocal_{\widehat{\A}} \to \widehat{\TSym}(\omega_{\A^\vee/S}).
\end{equation}

Recall the following constructions.

\begin{definition}\index[notation]{$Poincare bundle$@$\Po$, $\Po^\natural$}

	The \emph{Poincaré bundle} $\Po$ is the universal rigidified and algebraically equivalent to zero line bundle $\Po$ over $\A \times_S \A^\vee$. If we furthermore consider line bundles with this structure as well as an integrable $\Ocal_S$-connection, we obtain the \emph{Poincaré bundle with connection} $(\Po^\natural, \nabla)$, which is now a line bundle over $\A \times_S \A^\natural$.

\end{definition}

In particular, we are interested in studying the Poincaré bundle after taking the completion of $\A \times_S \A^\vee$ along the second coordinate. 

\begin{definition}\index[notation]{$Poincare bundle completed$@$\coPo$, $\coPo^\natural$}

	The \emph{completion of the Poincaré bundle} is defined as
	$$\coPo := (\id_{\A} \times \widehat{\pi}^\vee)_* \iota_{\A \times_S \widehat{\A}^\vee}^*\Po,$$
	where $\iota_{\A \times_S \widehat{\A}^\vee} \colon \A \times_S \widehat{\A}^\vee \to \A \times_S \A^\vee$ is the map induced by the completion. Similarly, we have the version with connection:
	$$\coPo^\natural := (\id_{\A} \times \widehat{\pi}^\natural)_* \iota_{\A \times_S \widehat{\A}^\natural}^*\Po,$$
	which is again a sheaf on $\A$.

\end{definition}

One of the crucial main properties of this Poincaré bundle is the following \emph{splitting principle} (see \cite[Corollary 2.9]{Kings2025}): if $\varphi \colon \A \to \B$ is an isogeny of abelian schemes over $S$ such that $\varphi^\vee$ is étale, then for any $\varphi$-torsion section $x \colon S \to \ker(\varphi) \subset \A$, we have an isomorphism of $\Ocal_S$-modules
$$x^*\coPo_\A \cong \Ocal_{\widehat{\A}^\vee},$$
where $\coPo_\A$ is the completion of the Poincaré bundle of $\A$. In fact, we can combine this result with the moment map \eqref{eq:moment map def} as follows: for any isogeny $\varphi \colon \A \to \B$ with étale dual, and any $\varphi$-torsion section $x$, we may take the composition
$$\mom{\varrho}_x \colon x^*\coPo_\A \cong \Ocal_{\widehat{\A}^\vee} \to \widehat{\TSym}(\omega_{\A^\vee/S}).$$
If we instead assume that $\varphi^\natural$ is étale, the splitting principle still holds for $\coPo_\A^\natural$, and we obtain
$$\mom{\varrho^\natural}_x \colon x^*\coPo_\A^\natural \to \widehat{\TSym}(\omega_{\A^\natural/S}).$$
Note that $\omega_{\A^\natural/S} \cong \Hscr_\A$. In both cases, we can take the following projections for any positive integer $b$:
\begin{equation}\label{eq:moment map proj}\index[notation]{$mom b x$@$\mom{\varrho}_x^b$, $\mom{\varrho^\natural}_x^b$}
\mom{\varrho}_x^b \colon x^*\coPo_\A \to \TSym^b(\omega_{\A^\vee/S}), \quad \mom{\varrho^\natural}_x^b \colon x^*\coPo_\A^\natural \to \TSym^b(\Hscr_\A).
\end{equation}

Lastly, we briefly comment on a variant of this setup of formal completions which will appear in the sequel. In particular, we will be working later over certain rings which are complete and separated in the topology induced by the ideal generated by $p$.

\begin{definition}

	We say that a ring $R$ is \emph{$p$-adic} if it is complete and separated with respect to its $p$-adic topology. In other words, the canonical map
	$$R \xrightarrow{\sim} \varprojlim_n R/p^nR$$
	is an isomorphism.

\end{definition}

Assume then that $S$ is an $R_0$-scheme for some $p$-adic ring $R_0$. We will then take the formal completion of our abelian scheme $\A$ not along $e$, but rather along the special fiber of the zero section:
$$e_p \colon S_{R_0/pR_0} \to \A_{R_0/pR_0},$$
which is still a closed immersion. Write $\widehat{\A}_p$\index[notation]{$A compl p$@$\widehat{\A}_p$} for this completion. The reason for this is that we also want to consider the formal completion $\widehat{S}$ of $S$ along $S_{R_0/pR_0}$, and by the functoriality of completions we obtain a map $\widehat{\A}_p \to \widehat{S}$. In fact, since all of the diagrams defining the group structure on $\A$ map $e_p(S_{R_0/pR_0})$ to $S_{R_0/pR_0}$, it follows that $\widehat{\A}_p/\widehat{S}$ is indeed a formal group scheme, and thus the construction of the moment map still holds. Although this completion gives a different formal scheme to $\widehat{\A}$, one can show that they have the same ring of global sections. In other words, the natural map $\widehat{\A}_p \to \widehat{\A}$ induces an isomorphism
\begin{equation}\label{eq:both formal compls agree}
H^0(\widehat{\A}_p, \Ocal_{\widehat{\A}_p}) \xrightarrow{\sim} H^0(\widehat{\A}, \Ocal_{\widehat{\A}})
\end{equation}
of rings \emph{without} topology.

\subsubsection{Real multiplication}

The abelian schemes with which we will work in what follows will all be equipped with an action by $\Ofr$, the ring of integers of $F$, as well as a polarization and some level structure. We now briefly recall the main definitions regarding these abelian schemes, and refer the reader to \cite{Rapoport1978} and \cite{Deligne1994} for further details. 

To state what these structures are, let us first recall the following construction.

\begin{definition}\index[notation]{$A a$@$\A(\afr)$, $\A \otimes_R \afr$}

	 Let $R$ be a ring with unity, possibly non-commutative, which is finite free over $\Z$. Let also $\A / S$ be an abelian scheme together with an injective ring homomorphism $R \hookrightarrow \text{End}_S(\A)$. Then, for any finitely presented projective right $R$-module $\afr$, we define the following functor on $S$-schemes, known as the \emph{Serre construction}:
	$$T \mapsto \A(\afr)(T) := \afr \otimes_R \A(T).$$
	The functor $\A(\afr)$ is represented by an abelian scheme over $S$, which we denote again by $\A(\afr)$, and satisfies that $\text{Lie}(\A(\afr)/S) \cong \afr \otimes_R \text{Lie}(\A/S)$. In particular, if $\afr$ is of rank 1, $\A(\afr)$ is also of relative dimension $g$. We will write as well $\afr \otimes_R \A := \A(\afr)$, or $\A(\afr) = \A \otimes_R \afr$ if $R$ is commutative.

\end{definition}

As we mentioned, we now introduce the main objects of study.

\begin{definition}

	Let $\cfr$ be a fractional ideal. A \emph{$\cfr$-polarized abelian scheme with RM (real multiplication) by $F$ over $S$} is a triple $(\A/S, m, \lambda)$ where
	\begin{enumerate}[(1)]
	
		\item $\A/S$ is an abelian scheme over $S$;
		
		\item $m \colon \Ofr \hookrightarrow \text{End}_S(\A)$ is an injective morphism of rings such that $\text{Lie}(\A/S)$ is, locally on $S$, an invertible $\Ofr \otimes_\Z \Ocal_S$-module (this is known as the \emph{Rapoport condition});
		
		\item $\lambda \colon \A(\cfr) \xrightarrow{\sim} \A^\vee$ is a \emph{$\cfr$-polarization}, meaning an isomorphism of abelian schemes respecting the $\Ofr$-structure and such that the map $\Hom_\Ofr(\A, \A(\cfr)) \to \Hom_\Ofr(\A, \A^\vee)$ induces bijections
		$$\cfr \cong \{\Ofr\text{-linear symmetric maps } \A \to \A^\vee\} \quad \text{and} \quad \cfr^+ \cong \{\Ofr\text{-linear polarizations}\}.$$
	
	\end{enumerate}
	In particular, an abelian scheme \emph{with RM by $F$} is a pair $(\A,m)$ with $m$ an embedding as in (2), and the Rapoport condition immediately implies that $\A$ has relative dimension $g$. Such a pair is also called in the literature a \emph{Hilbert-Blumenthal abelian scheme}. 

\end{definition}

\begin{remark}

	Observe that if $1 \in \cfr$ (which is equivalent to asking that $\cfr^{-1}$ is an integral ideal), then to any $\cfr$-polarization we can canonically associate a polarization $\A \to \A^\vee$ by choosing the image of 1 under the mentioned isomorphisms. Moreover, whenever $\A$ is equipped with a $\cfr$-polarization, we may also equip it with a $\xi\cfr$-polarization for all $\xi \in F^+$. In particular, we can always choose $\xi$ such that $1 \in \xi\cfr$.

\end{remark}

We may use the RM structure in order to generalize the notion of the ``multiplication by an integer'' maps. More concretely, if $\afr \subset \bfr$ are fractional ideals of $F$, the obvious homomorphism $\A(\afr) \to \A(\bfr)$ is an isogeny of degree $[\bfr : \afr]$. The following case will in particular appear several times in the sequel.

\begin{definition}\index[notation]{$a$@$[\afr]$}

	Let $\afr \subset \Ofr$ be an integral ideal. The \emph{multiplication by $\afr$ map} is the morphism
	$$[\afr] \colon \A \to \A(\afr^{-1})$$
	induced by the inclusion $\Ofr \subset \afr^{-1}$. In particular, it is an isogeny of degree $N\afr$. We also write $\A[\afr] := \ker([\afr])$.

\end{definition}

We now make certain observations regarding the induced $\T_{F,S}$-module structure. Consider $(\A/S,m)$ an abelian scheme with RM by $F$ such that $\Ocal_S$ is an $\Ofr_{F^{\text{Gal}}}[d_F^{-1}]$-algebra. The action by $\Ofr$ on $\A$ makes some of the sheaves which we have defined before into $\T_{F,S}$-modules. For example, the Rapoport condition implies that we have the following decomposition of $\text{Lie}(\A/S)$ as a $\T_{F,S}$-module:
$$\text{Lie}(\A/S) \cong \bigoplus_{i = 1}^g \text{Lie}(\A/S)(\sigma_i).$$
Following our \Cref{signs of dual}, taking the contragradient representation we get the following decomposition:
$$\omega_{\A/S} \cong \bigoplus_{i = 1}^g \omega_{\A/S}(-\sigma_i).$$
Note that this is precisely the \emph{inverse} of the action given on $\omega_{\A/S}$ by $m$, as now any $\gamma \in \Ofr^\times$ acts as $\gamma^{-1}$. Moreover, by our definition of RM on $\A^\vee$ as the dual $m^\vee$, we also adopt this inverse action convention on $\A^\vee$, which gives decompositions
$$\text{Lie}(\A^\vee/S) \cong \bigoplus_{i = 1}^g \text{Lie}(\A^\vee/S)(-\sigma_i), \quad \omega_{\A^\vee/S} \cong \bigoplus_{i = 1}^g \omega_{\A^\vee/S}(\sigma_i).$$
In particular, this means that if we have an $\Ofr$-linear polarization $\lambda \colon \A \to \A^\vee$ which is étale, the induced isomorphism $\omega_{\A^\vee/S} \cong \omega_{\A/S}$ is \emph{not} a map of algebraic $\T_{F,S}$-modules (despite being an $\Ofr$-linear isomorphism!). Instead, for any finite index subgroup $\Gamma \leq \Ofr^\times$, we obtain the following isomorphism which respects the $\T_{F,S}$-action:
\begin{equation}\label{dual colie algebra iso}
\omega_{\A^\vee / S} \otimes_{\Ocal_S[\Gamma], \varphi} \Ocal_S[\Gamma] \xrightarrow{\sim} \omega_{\A/S},
\end{equation}
where $\varphi : \Ocal_S[\Gamma] \to \Ocal_S[\Gamma]$ is the map given by $\gamma \mapsto \gamma^{-1}$.

\begin{convention}

	From now on, we consider the $\Gamma$-action on $\omega_{\A^\vee/S}$ and $\text{Lie}(\A^\vee/S)$ induced by the contragradient representation, and \emph{not} the one induced by a choice of $\cfr$-polarization.

\end{convention}

Some other $\T_{F,S}$-modules of relevance are the first de Rham cohomology sheaves $\Hscr_\A$ and $\Hscr_{\A^\vee}$, to which we apply similar conventions. In particular, the maps in the short exact sequence \eqref{eq:Hodge filtration ses} and its dual all respect our choices of $\T_{F,S}$-module structure, which leads, by an analogue of \cite[Corollary 1.12]{Kings2025}, to an isomorphism
$$\omega^g_{\A^\vee/S} \cong \omega^g_{\A/S}$$
of $\Gamma$-modules for any finite index subgroup $\Gamma \leq \Ofr^\times$.

Since the $\Ocal_S$-modules $\omega_{\A/S}(-\sigma_i)$ are locally free of rank 1, we introduce the following notation.

\begin{definition}\label{bases for sheaf inv diffs}\index[notation]{$omega A$@$\omega(\A)$}

	Let $(\A, m)$ be an abelian scheme with RM by $F$. A \emph{basis} $\omega(\A)$ of $\omega_{\A/S}$ will be a collection of global sections $\omega(\A)(-\sigma_i) \in H^0(S, \omega_{\A/S}(-\sigma_i))$ for $i = 1, \ldots, g$, which generate $\omega_{\A/S}$ as a $\Ofr \otimes_\Z \Ocal_S$-module. We define in the same way a basis $\omega(\A^\vee) = (\omega(\A^\vee)(\sigma_i))_i$ of $\omega_{\A^\vee/S}$.

\end{definition}

\begin{remark}\label{basis is equiv to iso}

	By duality and \Cref{duality of modules with F action}, if $\Ocal_S$ is an $\Ofr_{F^{\text{Gal}}}[d_F^{-1}]$-algebra, the datum of such a basis $\omega(\A)$ is equivalent to that of an $\Ofr \otimes_\Z \Ocal_S$-linear isomorphism 
	$$\text{Lie}(\A/S) \cong \dfr^{-1} \otimes_\Z \Ocal_S.$$

\end{remark}

\begin{remark}

	Note that such bases $\omega(\A)$, $\omega(\A^\vee)$ need not exist over general $S$, and one might require to change our base scheme in order to obtain them. 

\end{remark}

Whenever we have such a basis $\omega(\A)$, we use the following notation for the induced isomorphisms for the tensor symmetric products:
\begin{equation}\label{isom given by basis}
\omega(\chi) \colon \Ocal_S \xrightarrow{\sim} \TSym^\chi(\omega_{\A/S})
\end{equation}
for all $\chi \in I_F^+$. We write the image of such an isomorphism as $f \mapsto f(\omega(\A)^{[\chi]})$.

Lastly, we briefly discuss the case when the base scheme is $\Spec(\C)$. It is well known that abelian varieties over the complex numbers can be understood as polarizable complex tori, and when we introduce real multiplication into the situation, the lattices which appear can be described as follows.

\begin{definition}

	\begin{enumerate}[(1)]
	
		\item A \emph{lattice} $\Lambda$ in $F \otimes_\Q \C$ is a locally free $\Ofr$-submodule of rank 2 such that $\Lambda \otimes_\Z \R \cong F \otimes_\Q \C$. It is well-known that any such lattice admits an isomorphism of $\Ofr$-modules $\Lambda \cong \afr \oplus \Ofr$ for some fractional ideal $\afr$ (see \cite[Chapter II, Theorem 13]{Froehlich1993}).
		
		\item Let $\Lambda$ be a lattice in $F \otimes_\Q \C$, and $\cfr$ a fractional ideal of $F$. A \emph{$\cfr$-polarization on $\Lambda$} is the choice of an alternating $\Ofr$-bilinear form
		$$\langle -,- \rangle \colon \Lambda \wedge_\Ofr \Lambda \xrightarrow{\sim} \cfr^\vee,$$
		where recall that we defined $\cfr^\vee$ as the fractional ideal $\dfr^{-1}\cfr^{-1}$, and that the trace induces a canonical isomorphism $\cfr^\vee \cong \Hom_\Z(\cfr, \Z)$.
		
	\end{enumerate}

\end{definition}

One then has the following correspondences.

\begin{proposition}[{\cite[Section 1.4]{Katz1978}}]\label{correspondence lattices with compl AVs}

	We have a bijection:
	$$\{\text{lattices $\Lambda$ in } F \otimes_\Q \C\} \leftrightarrow \{ \text{triples } (\A/\C, m, \omega(\A)) \},$$
	where the triples in the right-hand side consist of complex abelian varieties with RM by $F$ and a basis of translation-invariant differentials. Moreover, if we fix a fractional ideal $\cfr$ we obtain the following correspondence:
	$$\{\text{$\cfr$-polarized lattices $(\Lambda, \langle -,- \rangle)$ in } F \otimes_\Q \C\} \leftrightarrow \{ \text{quadruples } (\A/\C, m, \lambda, \omega(\A)) \},$$
	where $\lambda$ on the right-hand side is a $\cfr$-polarization.

\end{proposition}

\subsubsection{Level structures}

The last piece of data required to define the Hilbert moduli schemes is the level structure. This will help us rigidify the moduli problem in order for it to be representable by a scheme, and moreover will provide us with the existence of several sections which are required for our constructions further on. 

We first recall the following notation regarding locally free group schemes.

\begin{notation}

	Let $G$ be a locally free group scheme over $S$ of finite rank. We write $G^\circ$ for the subgroup scheme given by the connected component of the zero section, and $G^\text{ét}$ for the maximal étale quotient of $G$ (i.e.~the maximal quotient of $G$ which defines an étale group scheme).

\end{notation}

It is well-known (see, for example, \cite[Corollary 4.4]{Hida2004}), that for any abelian scheme $(\A/S, m)$ with RM by $F$, if $S^{\text{red}}$ is of characteristic $p$, we then have a short exact sequence
\begin{equation}\label{eq:torsion ses for ab schemes with RM}
\begin{tikzcd}
	0 \arrow{r} & \A[p^k]^\circ \arrow{r} & \A[p^k] \arrow{r} & \A[p^k]^{\textnormal{ét}} \arrow{r} & 0
\end{tikzcd}
\end{equation}
splitting over every geometric point of $S$, as well as isomorphisms
$$\A[p^k]^\circ \cong \mu_{p^k} \otimes_\Z \dfr^{-1}, \quad \A[p^k]^{\textnormal{ét}} \cong \underline{\Ofr / p^k\Ofr}$$
for all $k \geq 1$.

With this explicit description in mind, we give the following (standard) definitions for our level structures.

\begin{definition}

	Let $(\A/S, m, \lambda)$ be a $\cfr$-polarized abelian scheme with RM by $F$ for some fractional ideal $\cfr$ with $\cfr^{-1}$ being integral, and let $n$ be a positive integer.
	
	\begin{enumerate}[(1)]
	
		\item Assume that $\cfr$ is coprime to $n$. A \emph{$\Gamma_{\textnormal{arith}}(n)$-level structure} (also referred to in the literature as a \emph{symplectic level $n$ structure}) is the datum of an isomorphism of group schemes with an $\Ofr$-action:
		$$\phi \colon (\mu_n \otimes_\Z \dfr^{-1}) \times \underline{(\Ofr / n\Ofr)} \xrightarrow{\sim} \A[n],$$
		such that the composition of the Weil pairing $e_n : \A[n] \times \A^\vee[n] \to \mu_n$ together with $\lambda$ agrees (via $\phi$) with the following natural symplectic pairing:
		\begin{equation}\label{eq:arith level str pairing}
		\begin{aligned}
		(\mu_n \otimes_\Z \dfr^{-1}) \times \underline{(\Ofr / n\Ofr)} & \longrightarrow \mu_n \\
		(\zeta \otimes s, [r]) & \longmapsto \zeta^{\text{Tr}_{F/\Q}(rs)}.
		\end{aligned}
		\end{equation}
		Note that the pairing $\A[n] \times \A[n] \to \mu_n$ obtained by composing $e_n$ with $\lambda$ is alternating, and thus we can reduce its expression using $\phi$ to the above description.
		
		\item A \emph{$\Gamma(n)$-level structure} (also known as a \emph{full} or \emph{naive level $n$ structure}) is the datum of an isomorphism of group schemes with an $\Ofr$-action:
		$$\alpha \colon \underline{(\Ofr / n\Ofr)}^2 \xrightarrow{\sim} \A[n].$$
		
		\item A \emph{$\Gamma_{00}(n)$-level structure} (also called in the literature a \emph{$\mu_n$-level structure}) is the datum of an embedding of group schemes with an $\Ofr$-action:
		$$\beta \colon \mu_n \otimes_\Z \dfr^{-1} \hookrightarrow \A[n].$$
	
	\end{enumerate}

\end{definition}

\begin{remark}

	It is clear that the datum of a $\Gamma_{\text{arith}}(n)$-level structure on $\A$ gives canonically a $\Gamma_{00}(n)$-structure on the same abelian scheme via the obvious inclusion. On the other hand, a $\Gamma(n)$-level structure may only be defined over a scheme $S$ where $n$ is invertible.

\end{remark}

\begin{definition}\label{sympletic level str def}

	Let $(\A/S, m, \lambda, \alpha)$ be a $\cfr$-polarized abelian scheme with RM by $F$ and a $\Gamma(n)$-level structure. By the same reasoning as in \cite[1.21]{Rapoport1978} (see also \cite[2.4]{Chai1990}), $\alpha$ induces an isomorphism
	$$ \mu_n \otimes_\Z \dfr^{-1}/n\dfr^{-1} \cong \underline{\cfr/n\cfr}.$$
	In particular, if $\cfr$ is coprime to $n$, we can rewrite this as $\mu_n \otimes_\Z \dfr^{-1} \cong \underline{\Ofr/n\Ofr}$. If this map is induced by an isomorphism $\varphi \colon \mu_n \to \underline{\Z/n\Z}$ of group schemes over $S$, we say that $\alpha$ is \emph{symplectic} and that $\varphi$ is the \emph{determinant of $\alpha$}. 
	
\end{definition}

Consider a tuple $(\A/R, m, \lambda, \beta)$ consisting of a $\cfr$-polarized abelian scheme over a ring $R$ together with a $\Gamma_{00}(p^n)$-level structure. If $\cfr$ is prime to $p$, we can construct an embedding
\begin{equation}\label{eq:dual ab sch structure}
\beta^\vee \colon \mu_{p^n} \otimes_\Z \dfr^{-1} \hookrightarrow \A^\vee[p^n],
\end{equation}
giving a $\Gamma_{00}(p^n)$-level structure on $\A^\vee$ by following the procedure in \cite[Section 3.3, p.252--253]{Katz1978}, which we here recall. Indeed, the level structure on $\A$ induces an embedding
$$\beta \otimes \text{id} : (\mu_{p^n} \otimes_\Z \dfr^{-1} )\otimes_\Ofr \cfr \hookrightarrow \A \otimes_\Ofr \cfr.$$
Using the isomorphism $\A \otimes_\Ofr \cfr \cong \A^\vee$ given by the $\cfr$-polarization, it is then enough to show that 
$$(\mu_{p^n} \otimes_\Z \dfr^{-1} )\otimes_\Ofr \cfr \cong \mu_{p^n} \otimes_\Z \dfr^{-1},$$ 
which follows from the fact that $\mu_{p^n}$ is a $\Z_p$-module and $\Z_p \otimes_\Z \dfr^{-1}\cfr \cong \Z_p \otimes_\Z \dfr^{-1}$ by the coprimality of $p$ and $\cfr$. 

Since we will be working regularly with abelian schemes with real multiplication and other pieces of data, we define the following objects.

\begin{definition}\label{test objects def finite case}

	Let $S$ be a noetherian scheme, $\cfr$ a fractional ideal of $F$ with $\cfr^{-1}$ integral, and $n$ a positive integer. For $\Gamma = \Gamma_{\text{arith}}(n), \Gamma(n), \Gamma_{00}(n)$, we define a \emph{$\Gamma$-test object} as a quadruple 
	$$(\A/S, m, \lambda, *)$$ 
	where $(\A/S, m, \lambda)$ is a $\cfr$-polarized abelian scheme with RM by $F$ and $*$ is the corresponding level structure. This also extends to define $(\Gamma(n), \Gamma_{00}(n'))$-test objects and so on whenever $n$ and $n'$ are coprime.

If we have now two arbitrary positive integers $n$ and $n'$, a \emph{$(\Gamma_\arith(n), \Gamma_{00}(n'))$-test object} is a tuple $(\A/S, m, \lambda, \phi, \beta)$ such that the two level structures are \emph{compatible}, i.e. if $d = (n, n')$, then we ask the following diagram to commute:
$$\begin{tikzcd}
\mu_d \otimes \dfr^{-1} \arrow[hook]{dd} \arrow[hook]{dr}{\beta} & \\
& \A[d], \\
(\mu_d \otimes_\Z \dfr^{-1}) \times \underline{(\Ofr / d\Ofr)} \arrow["\phi"', "\sim"]{ur} & 
\end{tikzcd}$$
where the vertical map is the obvious inclusion.

\end{definition}

\subsubsection{\textit{p}-adic trivializations of abelian schemes}

Similarly to the case over the complex numbers, we want to obtain an explicit description of our abelian schemes when we are working over $p$-adic rings. As we will see, by imposing a level structure which is of \emph{infinite level}, we will automatically obtain a description of the formal completion of our abelian schemes. Fix then a $\cfr$-polarized abelian scheme $(\A/S, m, \lambda)$.

\begin{definition}\label{inf gamma_00 level str}

	Let $N$ be a supernatural number, i.e.~an inverse system of natural numbers $(N_i)_i$ with respect to division. The datum of a $\Gamma_{00}(N)$\emph{-level structure} on $\A$ is an $\Ofr$-equivariant closed immersion
	$$\mu_{N} \otimes_\Z \dfr^{-1} \lhook\joinrel\longrightarrow \A,$$
	where $\mu_N$ is the ind-scheme $\varinjlim_i \mu_{N_i}$. As in \Cref{test objects def finite case}, we define $\Gamma_{00}(N)$-test objects, $(\Gamma(n), \Gamma_{00}(N))$-test objects (with $n$ coprime to $N$), $(\Gamma_\arith(n), \Gamma_{00}(N))$-test objects (for arbitrary $n$) and so on in the obvious way.

\end{definition}

\begin{remark}

	For any such $N = (N_i)_i$, it is clear that an abelian scheme with $\Gamma_{00}(N)$-level structure also admits a $\Gamma_{00}(N_i)$-level structure by the canonical injective maps $\mu_{N_i} \hookrightarrow \mu_N$. In particular, any such $\A$ is also ordinary at every geometric point of characteristic dividing $N$.

\end{remark}

Consider now the case where $S = \text{Spec}(R)$ for some $p$-adic ring $R$, and assume that $d_F$ is invertible in $R$ (therefore, $\dfr^{-1}$ is coprime to $R$). The following result shows that a $\Gamma_{00}(p^\infty)$-structure on $\A$ gives us a $p$-adic \emph{trivialization} of the formal group scheme $\widehat{\A}_p$ over $\text{Spf}(R)$ (which was defined as the formal completion of $\A$ along the special fiber of the zero section $e$).

\begin{lemma}[{\cite[Lemma 1.10.1]{Katz1978}}] \label{mu level structure datum}

	Let $\A / R$ as above. Then, the datum of an embedding
	$$\mu_{p^\infty} \otimes_{\Z} \dfr^{-1} \lhook\joinrel\longrightarrow \A$$
	is equivalent to an isomorphism of formal groups over $\textnormal{Spf}(R)$
	$$\widehat{\G}_m \otimes_{\Z} \dfr^{-1} \xlongrightarrow{\sim} \widehat{\A}_p.$$

\end{lemma}

\begin{remark}

	If we choose our prime $p$ such that it is coprime to the discriminant of $F$, we may rewrite the formal group in this lemma as
	$$\widehat{\G}_m \otimes_{\Z} \dfr^{-1} \cong \widehat{\G}_m \otimes_{\Z_p} (\Z_p \otimes_\Z \dfr^{-1}) \cong \widehat{\G}_m \otimes_{\Z_p} (\Z_p \otimes_\Z \Ofr).$$
	
\end{remark}

Fix now a $\Gamma_{00}(p^\infty)$-test object $(\A/R, m, \lambda, \beta)$ for the remainder of this subsection. Directly from this result, we obtain a canonical choice of basis for the translation-invariant differentials $\omega_{\A/R}$. Indeed, the isomorphism 
$$\widehat{\G}_m \otimes_\Z \dfr^{-1} \xrightarrow{\sim} \widehat{\A}_p$$
of \Cref{mu level structure datum} induces an isomorphism of Lie algebras
$$\text{Lie}(\widehat{\A}_p/R) \cong \text{Lie}(\widehat{\G}_m/R) \otimes_\Z \dfr^{-1} \cong R \otimes_\Z \dfr^{-1} = R \otimes_\Z \Ofr,$$
where we are using the identification $\text{Lie}(\widehat{\G}_m/R) \cong R$ given by the canonical non-vanishing differential $(1+T)\frac{\partial}{\partial T}$. Using \Cref{basis is equiv to iso} we obtain our basis.

\begin{definition}\label{canonical basis affine def}

	The basis of $\omega_{\A/R}$ associated to the isomorphism 
	$$\text{Lie}(\A/R) \cong \text{Lie}(\widehat{\A}_p/R) \cong R \otimes_\Z \dfr^{-1} = R \otimes_\Z \Ofr$$ 
	just described above is the \emph{canonical basis of $\A$}, and we denote it by $\omega(\A)_{can}$. 
	
\end{definition}

Assume that $N\cfr^{-1}$ is invertible in $R$, or equivalently, that $\cfr$ is coprime to $R$. Since the $\cfr$-polarization induces a polarization $\lambda \colon \A \to \A^\vee$ which is étale, the isomorphism \eqref{dual colie algebra iso} gives a basis $\lambda^*\omega(\A)_{can}$ from the canonical basis $\omega(\A)_{can}$. On the other hand, we can also equip $\A^\vee$ with a $\Gamma_{00}(p^\infty)$-level structure by repeatedly using \eqref{eq:dual ab sch structure}, which by this definition gives us a canonical basis $\omega(\A^\vee)_{can}$. By these constructions, both bases $\lambda^*\omega(\A)_{can}$ and $\omega(\A^\vee)_{can}$ correspond to the following dashed map which makes the diagram commute:
\begin{equation}\label{eq:relation canonical bases}
\begin{tikzcd}[column sep = 2cm]
\text{Lie}(\A^\vee/R) \arrow["\sim", "\lambda"']{d} \arrow[dashed]{r} & R \otimes_\Z \dfr^{-1} \arrow[equals]{d} \\
\text{Lie}(\A/R) \otimes_\Ofr \cfr \arrow["\sim", "\omega(\A)_{can} \otimes \text{id}"']{r} & R \otimes_\Z \dfr^{-1}\cfr.
\end{tikzcd}
\end{equation}
In particular, $\lambda^*\omega(\A)_{can} = \omega(\A^\vee)_{can}$.

In what follows, we will consider $\A$ together with some arbitrary basis $\omega(\A)$, which will make it necessary to compare it to this canonical choice of basis.

\begin{definition}\label{p-adic periods def}\index[notation]{$Omega$@$\Omega_p$}

	Let $\omega(\A)$ be an arbitrary basis of $\omega_{\A/R}$. The \emph{$p$-adic period associated to $(\A, \omega(\A))$} is the element $\Omega_{p,\A} := \Omega_p \in \Ofr \otimes_\Z R$ such that
	$$\omega(\A) = \Omega_p \cdot \omega(\A)_{can}.$$
	It is clear that this period is unchanged under étale isogenies.

\end{definition}

This infinite level structure can also be used to obtain further results about the completion $\coPo$ of the Poincaré bundle. Firstly, we have the following trivialization, which is a slightly more general version of \cite[Proposition 5.9]{Kings2025} and can be proven in the same way.

\begin{proposition}\label{Poincare p-adic triv}

	There exists a canonical isomorphism
	$$\coPo|_{\widehat{\A}_p} \xlongrightarrow{\sim} \textnormal{pr}_{\widehat{\A}_p,*}\Ocal_{\widehat{\A}_p \times \widehat{\A}_p^\vee},$$
	where $\textnormal{pr}_{\widehat{\A}_p} \colon \widehat{\A}_p \times \widehat{\A}_p^\vee \to \widehat{\A}_p$ is the projection map.

\end{proposition}

\begin{remark}

	We will write this isomorphism from now on as $\widehat{\Po}|_{\widehat{\A}_p} \cong \Ocal_{\widehat{\A}_p \times \widehat{\A}_p^\vee}$ by omitting the pushforward, directly seeing the sheaf $\Ocal_{\widehat{\A}_p \times \widehat{\A}_p^\vee}$ as an $\Ocal_{\widehat{\A}_p}$-module.

\end{remark}

As another consequence of the invariance of $\Po$ under isogenies, we also obtain that the completion of the Poincaré bundle is stable under translations.

\begin{lemma}[{\cite[Lemma 5.12]{Kings2025}}]\label{Poincare translation inv}

	Let $\varphi \colon \A \to \A'$ be an isogeny with étale dual, and $y \in \A(R)$ be a $\varphi$-torsion point. We can associate to it the translation by $y$ map $T_y \colon \A \to \A$. There is a canonical isomorphism
	$$T_y^*\coPo \cong \coPo.$$

\end{lemma}

For any such $\varphi$ and $y$ as in this lemma, we can write the following isomorphism
\begin{equation}\label{eq:Poincare isom triv}
\widehat{\varrho}_y \colon (T_y^*\coPo)|_{\widehat{\A}_p} \xlongrightarrow{\sim} \coPo|_{\widehat{\A}_p} \xlongrightarrow{\sim} \Ocal_{\widehat{\A}_p \times \widehat{\A}_p^\vee}
\end{equation}
given by composing the maps in \Cref{Poincare p-adic triv} and \Cref{Poincare translation inv}.

\subsubsection{Relative Frobenius and Verschiebung}

Let $(\A / S, m, \lambda)$ be a $\cfr$-polarized abelian scheme defined over a noetherian scheme $S$ and such that $\cfr$ is coprime to $p$. Our objective now is to construct certain isogenies of abelian schemes which correspond to the classical lifts of Frobenius and Verschiebung when $S = \Spec(R)$ with $R$ a $p$-adic ring. We note that the following construction is a straight-forward generalization of the classical theory of canonical subgroups, in particular seen in \cite[Section 1.11]{Katz1978}, where Katz works over $p$-adic rings.

As it will be of relevance now and through the rest of this work, we set the following notation for Cartier duality.

\begin{notation}

	If $G$ is a locally group scheme over $S$ of finite rank, $G^t := \Hom_S(G, \G_{m,S})$ is its Cartier dual.

\end{notation}

\begin{definition}\label{quot by can subgps general ring def}\index[notation]{$C n$@$C_n$} \index[notation]{$A (n)$@$\A_{(n)}$} \index[notation]{$Frobenius$@$F^n_{\A/S}$} \index[notation]{$Verschiebung$@$V^n_{\A/S}$}

	Let $(\A / S, m, \lambda, \beta)$ be a $\Gamma_{00}(p^n)$-test object. We define its \emph{$n$-th canonical subgroup} as $C_n := \beta(\mu_{p^n} \otimes_\Z \dfr^{-1})$, and we obtain abelian schemes $\A_{(n)} := \A/C_n$ together with isogenies $F_{\A/S}^n \colon \A \to \A_{(n)}$, which we call the \emph{relative $n$-th Frobenius morphism}. If we take the projection
	$$\begin{tikzcd}
	V^n_\A \colon \A_{(n)} \arrow{r} & \A/\A[p^n] \arrow{r}{\sim} & \A
	\end{tikzcd}$$
	(where the last isomorphism is given by the isogeny $[p^n]$) we obtain the \emph{$n$-th relative Verschiebung}, an isogeny which satisfies that $V^n_{\A/S} \circ F^n_{\A/S} = [p^n]$. 

\end{definition}

\begin{remark}

	If $\A$ is equipped with a $\Gamma_{00}(p^\infty)$-level structure, we may equip $\A_{(n)}$ with one as well by the procedure seen in the proof of \cite[Lemma 1.11.6]{Katz1978}. More generally, if we assume that $\A$ is equipped with a $\Gamma_{00}(p^m)$-level structure for $m \geq 2n$, then we may equip $\A_{(n)}$ with a $\Gamma_{00}(p^n)$-level structure $\beta_{(n)}$ by requiring the following diagram to commute:
	$$\begin{tikzcd}
	0 \arrow{r} & \mu_{p^n} \otimes_\Z \dfr^{-1} \arrow{d}{\text{id}} \arrow{r} & \mu_{p^{2n}} \otimes_\Z \dfr^{-1} \arrow{d} \arrow{r}{\cdot p^n} & \mu_{p^n} \otimes_\Z \dfr^{-1} \arrow{d}{\beta_{(n)}} \arrow{r} & 0 \\
	0 \arrow{r} & \mu_{p^n} \otimes_\Z \dfr^{-1} \arrow{r} & \A \arrow[swap]{r}{F^n_\A} & \A_{(n)} \arrow{r} & 0.
	\end{tikzcd}$$

\end{remark}

\begin{remark}

	In general, the $\cfr$-polarization on $\A$ gives an $\Ofr$-linear polarization $\lambda_{(n)}$ on $\A_{(n)}$ as follows:
	$$\begin{tikzcd}
	\A \arrow{r}{\lambda} & \A^\vee \arrow{d}{(V^n_{\A})^\vee} \\
	\A_{(n)} \arrow{u}{V^n_\A} \arrow[dashed]{r} & (\A_{(n)})^\vee,
	\end{tikzcd}$$
	but this may not be a $\cfr$-polarization in general.

\end{remark}

One of the convenient properties of the quotient $\A_{(n)}$ is that the level structure on $\A$ induces on it a $\Gamma_\arith(p^n)$-level structure.

\begin{lemma}\label{splitting p-tors Frob twist}

	Let $(\A/S, m, \lambda, \beta)$ be a $\Gamma_{00}(p^{2n})$-test object. Then $\A_{(n)}$ can be equipped canonically with a $\Gamma_\arith(p^n)$-test object structure.

\end{lemma}

\begin{proof}

	First, note that $F^n_\A \circ V^n_\A$ is equal to multiplication by $p^n$ in $\A_{(n)}$. It thus follows that we have a short exact sequence
	$$\begin{tikzcd}
	0 \arrow{r} & \ker(V^n_\A) \arrow{r} & \A_{(n)}[p^n] \arrow{r}{V^n_\A} & \ker(F^n_\A) \arrow{r} & 0.
	\end{tikzcd}$$
	On the other hand, we claim that the Verschiebung map $V^n_\A$ carries $\beta_{(n)}(\mu_{p^n} \otimes_\Z \dfr^{-1})$ into $\ker(F^n_\A)$ isomorphically. Indeed, by construction of $\beta_{(n)}$ we have a commutative diagram
	$$\begin{tikzcd}
	\A \arrow{r}{F^n_\A} & \A_{(n)} \arrow{r}{V^n_\A} & \A, \\
	\mu_{p^{2n}} \otimes_\Z \dfr^{-1} \arrow{u}{\beta} \arrow[swap]{r}{\cdot p^n} & \mu_{p^n} \otimes_\Z \dfr^{-1} \arrow[swap]{u}{\beta_{(n)}} & 
	\end{tikzcd}$$
	and since the multiplication by $p^n$ map is surjective, we have that 
	$$V^n_\A(\beta_{(n)}(\mu_{p^n} \otimes_\Z \dfr^{-1})) =V^n_\A \circ F^n_\A \circ \beta(\mu_{p^{2n}} \otimes_\Z \dfr^{-1}) = $$
	$$ = [p^n](\beta(\mu_{p^{2n}} \otimes_\Z \dfr^{-1})) = \beta(\mu_{p^n} \otimes_\Z \dfr^{-1}) = \ker(F^n_\A).$$
	Thus, the inclusion $\beta_{(n)}(\mu_{p^n} \otimes_\Z \dfr^{-1}) \subset \A_{(n)}[p^n]$ provides us with a splitting of the previous short exact sequence. 
	
	Consider then the Weil pairing $\A_{(n)}[p^n] \times_S \A_{(n)}[p^n] \to \mu_{p^n}$ (obtained by using $\lambda_{(n)}$), which is perfect and alternating. In particular, this gives an injective morphism
	$$\ker(V^n_\A) \to \Hom(\mu_{p^n} \otimes_\Z \dfr^{-1}, \mu_{p^n}) \cong \underline{\Ofr/p^n\Ofr}.$$
	Since $V^n_\A$ is étale, this is an embedding of finite étale group schemes of the same rank, and thus an isomorphism. Our previous splitting is then of the form $\A_{(n)}[p^n] \cong (\mu_{p^n} \otimes_\Z \dfr^{-1}) \times (\underline{\Ofr/p^n\Ofr})$, which will give the $\Gamma_\arith(p^n)$-level structure.
	
	As for the desired $\cfr$-polarization, one may proceed as in the proof of \cite[Lemma 1.11.6]{Katz1978} by using the short exact sequence
	$$\begin{tikzcd}
	0 \arrow{r} & \underline{\Ofr/p^n\Ofr} \arrow{r} & \A_{(n)} \arrow{r}{V^n_\A} & \A \arrow{r} & 0.
	\end{tikzcd}$$

\end{proof}

\begin{remark}

	Note that the isogeny $V^n_\A$ is always étale: indeed, we may check this on the fibers of $\ker(V^n_\A) \to S$. For any point $s \in S$, we have two possibilities: if the characteristic of $k(s)$ is different from $p$, then the fiber of Verschiebung is étale (because multiplication by $p^n$ is étale). Otherwise, $k(s)$ is a $p$-adic ring, and one may reason using the short exact sequence \eqref{eq:torsion ses for ab schemes with RM} to show that $\ker(V^n_{\A_s/k(s)}) \cong \A_s[p^n]^{\text{ét}}$.

\end{remark}

\begin{definition}

	Let $n \geq 1$ be an integer. For any $\Gamma_{00}(p^{2n})$-test object $(\A/S, m, \lambda, \beta)$, we will write $(\A_{(n)}/S, m_{(n)}, \lambda_{(n)}, \phi_{(n)})$ for the induced $\Gamma_\arith(p^n)$-test object as described in \Cref{splitting p-tors Frob twist}. Note that if we further equip $\A$ with a $\Gamma_{00}(p^\infty)$-level structure extending $\beta$, then $\A_{(n)}$ is canonically a $(\Gamma_\arith(p^n), \Gamma_{00}(p^\infty))$-test object.

\end{definition}

We use this level structure to give the following notion of a ``partial'' Fourier transform for functions which are defined on the kernels of the Frobenius and Verschiebung maps. 

\begin{definition}[{\cite[Definition 5.24]{Kings2025}}]\label{partial Fourier alg def}\index[notation]{$P_\A H$}

	Let $R$ be a ring and $S$ a noetherian $R$-scheme together with a $\Gamma_{00}(p^\infty)$-test object $\A/S$. Note that Cartier duality provides us with a perfect pairing
	$$\langle -,- \rangle_{F^n} \colon \ker(F^n_\A) \times \ker(F^n_\A)^t \to \mu_{p^n, S}.$$ 
	Consider a map of sets $H \colon \ker(F^n_\A)^t(S) \times \ker(V^n_\A)(S) \to R$. The \emph{partial Fourier transform of $H$ relative to $\A$} is the map
	$$P_\A H \colon \beta_{(n)}(\mu_{p^n} \otimes_{\Z} \dfr^{-1})(S) \times \ker(V^n_\A)(S) \cong \A_{(n)}[p^n](S) \to H^0(S, \Ocal_S)$$
	defined as
	$$P_\A H(r,s) := \frac{1}{p^{ng}} \sum_{t \in \ker(F^n_\A)^t(S)} H(t, s) \langle V^n_\A(r), t\rangle_{F^n}^{-1}.$$
	
\end{definition}

\begin{remark}\label{partial Fourier alg rewriting}

	For later comparison, consider a function $H \colon (\Ofr/p^n\Ofr)^2 \to R$ and assume that $S$ is connected. In such a case, the $\Gamma_\arith(p^n)$-level structure gives us isomorphisms
	$$\ker(F^n_\A)^t \cong \underline{\Ofr/p^n\Ofr} \cong \ker(V^n_\A),$$
	allowing us to define the function $P_\A H$.

\end{remark}

\subsection{Hilbert moduli spaces and modular forms}

Next, we describe the moduli spaces parametrizing the test objects we previously introduced, as well as the automorphic forms related to them. 

\subsubsection{Hilbert moduli spaces}

As we will be working with abelian schemes with a $\Gamma(l)$-level structure from now on, we make the following technical assumption.

\begin{convention}

	Fix a positive integer $l$ and a primitive $l$-th root of unity $\zeta_l \in \overline{\Q}$. From now on, for every noetherian $\Z[\zeta_l]$-scheme $S$, every $\Gamma(l)$-level structure over $S$ will be assumed to be symplectic and with the determinant $\mu_l \xrightarrow{\sim} \Z/l\Z$ induced by $\zeta_l$.  

\end{convention}

We may now define the moduli space parametrizing $\Gamma(l)$-test objects.

\begin{definition}\label{Hilbert moduli scheme def}

	Let $\cfr$ be a fractional ideal of $F$ which is coprime to $l$ and such that $\cfr^{-1}$ is integral. We define the \emph{Hilbert moduli problem with $\Gamma(l)$-level structure} as the functor
	\begin{align*}
	\Mcal(\cfr, \Gamma(l)) \colon (\Z[l^{-1}, \zeta_l]\text{-schemes}) & \longrightarrow (\text{Sets})\\
	S & \longmapsto \{ \Gamma(l)\text{-test objects over } S\} / \cong.
	\end{align*}
	
\end{definition}

This moduli problem is indeed representable for $l$ big enough.

\begin{theorem}[{\cite[Lemme 1.23, Théorème 1.28]{Rapoport1978}}]

	Let $\cfr$ be as above. If $l \geq 3$, $\Mcal(\cfr, \Gamma(l))$ is representable by a smooth scheme over $\textnormal{Spec}(\Z[l^{-1}, \zeta_l])$ of relative dimension $g$ (in particular, it is quasi-compact and quasi-separated). Moreover, the base change of this scheme to $\Q(\zeta_l)$ is geometrically connected.

\end{theorem}

Whenever this moduli problem is representable, we have a \emph{universal abelian scheme}
$$\pi \colon \A \to \Mcal(\cfr, \Gamma(l)),$$
which can also be considered as a $\Gamma(l)$-test object.

We are interested in the study of Hilbert moduli schemes not only with $\Gamma(l)$-level structure, but also with an additional level structure at $p$. For this we consider the following modifications of the previous moduli problem.

\begin{definition}

	Let $\cfr$ be a fractional ideal which is coprime to $l$ and such that $\cfr^{-1}$ is integral. Fix as well a prime $p$ which is coprime to $l$ and $d_F$. We then define the following moduli problems.
	
	\begin{enumerate}[(1)]
	
		\item The \emph{$p$-ordinary locus} of the Hilbert moduli problem with $\Gamma(l)$-level structure is given by the functor
		\begin{align*}
		\Mcal(\cfr, \Gamma(l))^{\text{ord}} \colon (\Z[l^{-1}, \zeta_l]\text{-schemes}) & \longrightarrow (\text{Sets})\\
		S & \longmapsto \{p\text{-ordinary } \Gamma(l)\text{-test objects over } S\} / \cong.
		\end{align*}
		Note that by the theory of Hasse invariants, this functor is represented by an open subscheme of $\Mcal(\cfr, \Gamma(l))$ whenever $l \geq 3$.
		
		\item For all $n \geq 1$, we consider the \emph{Hilbert moduli problem with $(\Gamma(l), \Gamma_{00}(p^n))$-level structure} as the functor
	\begin{align*}
	\Mcal(\cfr, \Gamma(l), \Gamma_{00}(p^n)) \colon (\Z[l^{-1}, \zeta_l]\text{-schemes}) & \longrightarrow (\text{Sets})\\
	S & \longmapsto \{ (\Gamma(l), \Gamma_{00}(p^n))\text{-test objects over } S\} / \cong.
	\end{align*}
			
	\end{enumerate}

\end{definition}

From this definition, it is clear that we have ``forgetting the level structure'' maps
$$\Mcal(\cfr, \Gamma(l), \Gamma_{00}(p^n)) \to \Mcal(\cfr, \Gamma(l))^{\text{ord}},$$
for all $n \geq 1$. In fact, each of these maps is a Galois covering with Galois group $\Ofr/p^n\Ofr$ (the inverse system of these covers is known as the \emph{Hilbert modular Igusa tower}). This can be used in order to give the following definition as the inverse limit of these schemes.

\begin{definition}

	Let $l$, $p$, $\cfr$ be as above. Then, the \emph{Hilbert moduli scheme of $(\Gamma(l), \Gamma_{00}(p^\infty))$-level} is the inverse limit
	$$\Mcal(\cfr, \Gamma(l), \Gamma_{00}(p^\infty)) := \varprojlim_n \Mcal(\cfr, \Gamma(l), \Gamma_{00}(p^n)).$$
	Clearly, this scheme represents the moduli functor of $(\Gamma(l), \Gamma_{00}(p^\infty))$-test objects.

\end{definition}

Let $\Gamma$ be a finite index subgroup of $\Ofr^\times$. We may then consider, for any $N \in \Z_{>0} \cup \{\infty\}$, an action of $\Gamma$ on any Hilbert moduli scheme $\Mcal := \Mcal(\cfr, \Gamma(l), \Gamma_{00}(p^N))$. This is done as follows: for any $\gamma \in \Gamma$, we define a map $f_\gamma \colon \Mcal \to \Mcal$ given on $S$-points by
$$(\B/S, m, \lambda, \alpha, \beta) \mapsto (\B/S, m, \lambda, \gamma\alpha, \gamma\beta).$$
If $\A/\Mcal$ is the universal abelian scheme, this action on $\Mcal$ gives a $\Gamma$-action on $\A$ by taking the fiber product:
\begin{equation}\label{eq:Gamma act on A M}
\begin{tikzcd}
\A \arrow{r}{\sim} \arrow{d} & \A \arrow{d} \\
\Mcal \arrow["\sim"', "f_\gamma"]{r} & \Mcal.
\end{tikzcd}
\end{equation}
Note that this action is \emph{not} the one induced by the real multiplication structure on $\A$, as these induced automorphisms $\A \xrightarrow{\sim} \A$ are morphisms over $\Mcal$, and therefore do not commute a priori with the maps $f_\gamma$. 

We fix the following notation for the rest of this work.

\begin{notation}\label{Hilb mod problem notation}

	We fix $l$ a positive integer, $p$ a prime and $\cfr$ a fractional ideal such that $l$, $p$, $\cfr$ are  pairwise coprime, $p$ and $d_F$ are coprime, and $\cfr^{-1}$ is integral. We will then work in the sequel with the Hilbert moduli problems 
	$$\Mcal := \Mcal(\cfr, \Gamma(l), \Gamma_{00}(p^N)),$$ 
	for $N$ possibly being infinite. Recall that we assume that the $\Gamma(l)$-level structures are symplectic and with fixed determinant, so that $\Mcal$ is (geometrically) connected. If this moduli problem is representable (which happens whenever $l$ or $N$ are big enough), we will work with the universal abelian scheme $\pi \colon \A \to \Mcal$, which is a $(\Gamma(l), \Gamma_{00}(p^\infty))$-test object. Note that if $N = \infty$, $\Mcal$ is representable even when $l = 1$.

\end{notation}

\subsubsection{Hilbert modular forms}

Following \cite{Katz1978}, we will now describe a certain type of automorphic forms which are associated to the Hilbert moduli problem $\Mcal = \Mcal(\cfr, \Gamma(l), \Gamma_{00}(p^N))$ as in \Cref{Hilb mod problem notation}. We fix a noetherian $\Z[l^{-1}, \zeta_l]$-algebra $R_0$, and recall that we defined the group $I_{F,R_0}$ of characters of the torus $\T_{F, R_0}$ as morphisms of group schemes $\T_{F, R_0} \to \G_{m, R_0}$. In particular, if $\chi \in I_{F, R_0}$, we obtain group homomorphisms
$$\chi(R) \colon (\Ofr \otimes_\Z R)^\times \to R^\times$$
for all $R_0$-algebras $R$, and these morphisms commute with extension of scalars. We also write $\chi_R \in I_{F,R}$ for the base change of $\chi$. 

For any abelian scheme $\B/R$ with RM by $F$, the Rapoport condition implies that $\omega_{\B/R}$ is a locally free $\Ofr \otimes_\Z R$-module of rank 1. Therefore, it can be seen as a cocycle in the first sheaf cohomology group $H^1(\text{Spec}(R), (\Ofr \otimes_\Z R)^\times)$. Since $\chi_R$ induces a map
$$H^1(\text{Spec}(R), (\Ofr \otimes_\Z R)^\times) \to H^1(\text{Spec}(R), R^\times),$$
we write $\omega_{\B/R}(\chi_R)$ for the line bundle on $R$ defined by the image of the cocycle associated to $\omega_{\B/R}$. Note that if $R$ is an $\Ofr_{F^{\text{Gal}}}[d_F^{-1}]$-algebra, then $\omega_{\B/R}(\chi_R) \cong \TSym^{\chi_R}(\omega_{\B/R})$.

\begin{definition}\label{Hilbert mod forms first def}\index[notation]{$M_{\chi}(\cfr, \Gamma(l), \Gamma_{00}(p^N), R_0)$}

	Let $R_0$ be a ring and $\chi \in I_{F, R_0}$. A \emph{$\cfr$-Hilbert modular form of weight $\chi$ and $(\Gamma(l), \Gamma_{00}(p^N))$-level over $R_0$} is a rule which associates to any $R_0$-algebra $R$ and any $(\Gamma(l), \Gamma_{00}(p^N))$-test object $(\B/R, m, \lambda, \alpha, \beta)$ over $R$ an element
	$$f(\B/R, m, \lambda, \alpha, \beta) \in H^0(\text{Spec}(R), \omega_{\B/R}(\chi_R))$$
	such that:
	\begin{enumerate}[(1)]
	
		\item $f(\B/R, m , \lambda, \alpha, \beta)$ depends only on the isomorphism class of the test object, and
		
		\item it commutes with base change by extension of scalars $R \to R'$ of $R_0$-algebras. 
	
	\end{enumerate}
	Write $M_{\chi}(\cfr, \Gamma(l), \Gamma_{00}(p^N), R_0)$ for the space of $\cfr$-Hilbert modular forms of weight $\chi$ and with $(\Gamma(l), \Gamma_{00}(p^N))$-level defined over $R_0$. When the level and polarization ideal are clear, we will also write $M_\chi(R_0)$.
	
\end{definition}

If we also consider our abelian schemes together with a non-vanishing differential, we may see our Hilbert modular forms as taking values in the ring $R$: for any $(\Gamma(l), \Gamma_{00}(p^N))$-test object $(\B/R, m, \lambda, \alpha, \beta)$ and any basis $\omega(\B)$ of $\omega_{\B/R}$, we write $f(\B/R, m , \lambda, \alpha, \beta, \omega(\B)) \in R$ for the element such that
$$f(\B/R, m, \lambda, \alpha, \beta) = f(\B/R, m, \lambda, \alpha, \beta, \omega(\B)) \cdot \omega(\B)^{[\chi]}.$$
From our definition, this gives a new rule which satisfies the above properties (1) and (2) as well as
\begin{enumerate}[(3)]
	
	\item for any such tuple $(\B/R, m , \lambda, \alpha, \beta, \omega(\B))$ and any $a \in (\Ofr \otimes_\Z R)^\times$, we have that 
	$$f(\B/R, m , \lambda, \alpha, \beta, a^{-1} \cdot \omega(\B)) = \chi_R(a) \cdot f(\B/R, m , \lambda, \alpha, \beta, \omega(\B)).$$

\end{enumerate}

Whenever our moduli problem $\Mcal$ is representable and $R_0$ is an $\Ofr_{F^{\text{Gal}}}[d_F^{-1}]$-algebra, it follows that
$$M_\chi(\cfr, \Gamma(l), \Gamma_{00}(p^N), R_0) \cong H^0(\Mcal_{R_0}, \TSym^\chi(\omega_{\A_{R_0}/\Mcal_{R_0}})).$$

We will not go in depth in the theory of $q$-expansions, although their theory underlies several of the results which we use, and we will make use of the $q$-expansion principle. We refer the interested reader to \cite[Section 1.1]{Katz1978} for the corresponding definitions and discussion. 

\subsubsection{\textit{p}-adic Hilbert modular forms}

Set now $\Mcal := \Mcal(\cfr, \Gamma(l), \Gamma_{00}(p^\infty))$, so we are working in the infinite level case. We now move to describing $p$-adic Hilbert modular forms by following the exposition of \cite[Section 1.9]{Katz1978}.

Fix a $p$-adic $\Z[l^{-1}, \zeta_l]$-algebra $R_0$, and consider $\A$, $\Mcal$ as schemes over $R_0$ by base change. Unlike in the case where our abelian schemes were defined directly over $R_0$, which is $p$-adically complete, the scheme $\Mcal$ isn't, and we will therefore work regularly with its formal completion.

\begin{notation}\label{formal compl Hilbert mod scheme notation}

	We write $\widehat{\Mcal}$ for the formal completion of $\Mcal$ along its special fiber $\Mcal_{R_0/pR_0}$. This formal scheme is then defined over Spf$(R_0)$.

\end{notation}

Similarly to our definitions of the classical Hilbert modular forms, their $p$-adic variants will be sections of the structure sheaf of $\widehat{\Mcal}$. We may then see them as certain rules on abelian schemes defined over $p$-adic rings which commute with \emph{continuous} base change.

\begin{definition}\label{p-adic mod forms def}\index[notation]{$V(\cfr, \Gamma(l), R_0)$}

	A \emph{$p$-adic $\cfr$-Hilbert modular form over $R_0$} is a rule which to each $p$-adic $R_0$-algebra $R$ and $(\Gamma(l), \Gamma_{00}(p^\infty))$-test object $(\B, m, \lambda, \alpha, \beta)$ associates an element 
	$$f(\B,m,\lambda, \alpha, \beta) \in R$$ 
	satisfying the following properties:
	\begin{enumerate}[(1)]
	
		\item $f(\B,m,\lambda, \alpha, \beta)$ depends only on the isomorphism class of the tuple, and
		
		\item it commutes with base change by continuous extension of scalars $R \to R'$ of $p$-adic $R_0$-algebras. 
	
	\end{enumerate}
	We write $V(\cfr, \Gamma(l), R_0)$ for the ring of all $p$-adic $\cfr$-Hilbert modular forms over $R_0$, or $V(R_0)$ if the polarization ideal and level are understood.

\end{definition}

\begin{remark}

	As it was the case with classical Hilbert modular forms, we obtain an isomorphism
	$$V(\cfr, \Gamma(l), R_0) \cong H^0(\widehat{\Mcal}, \Ocal_{\widehat{\Mcal}}).$$

\end{remark}

In the way that we have defined $p$-adic modular forms, we can see that a priori they are different objects from the classical modular forms defined over $R_0$, as the latter satisfy a more general base change condition. Still, we can relate both types of modular forms as follows, by using the canonical bases of \Cref{canonical basis affine def}.

\begin{theorem}[{\cite[Theorem 1.10.15]{Katz1978}}]\label{map from classical to p-adic mod forms}

	We have a ring homomorphism
	$$\bigoplus_{\chi \in I_{F,R}} M_\chi(\cfr,\Gamma(l), \Gamma_{00}(p^\infty), R_0) \to V(\cfr, \Gamma(l), R_0), \quad f \mapsto \widetilde{f}$$
	given by the formula $\widetilde{f}(\B, m, \lambda, \alpha, \beta) := f(\B, m, \lambda, \alpha, \beta, \omega(\B)_{can})$ for any $(\Gamma(l), \Gamma_{00}(p^\infty))$-test object $(\B, m, \lambda, \alpha, \beta)$. Moreover, this homomorphism preserves $q$-expansions.

\end{theorem}

Next, we want to describe how to obtain a $p$-adic trivialization for the universal abelian scheme $\A/\Mcal$, extending \Cref{mu level structure datum}. We fix the following notation regarding the formal completion of $\A$.

\begin{notation}\label{formal compl special fiber notation}

	Take the special fiber of the zero section,
	$$e_p \colon \Mcal_{R_0/pR_0} \to \A_{R_0/pR_0},$$
	and write $\widehat{\A}_p$ for the completion of $\A$ along the closed subscheme $e_p(\Mcal_{R_0/pR_0})$. By the functoriality of formal completions, this is a formal group scheme $\widehat{\A}_p \to \widehat{\Mcal}$.
	
\end{notation}

The following shows that the $p$-adic trivialization of \Cref{mu level structure datum} holds for the universal abelian variety.

\begin{lemma}\label{universal Poincare props}

	The $\Gamma_{00}(p^\infty)$-level structure on $\A$ induces an isomorphism of formal groups
	$$\widehat{\G}_{m, \widehat{\Mcal}} \otimes_{\Z_p} (\Ofr \otimes_\Z \Z_p) \xrightarrow{\sim} \widehat{\A}_p.$$

\end{lemma}

\begin{proof}

	After choosing a finite affine cover of $\Mcal$ with affine intersections (recall that $\Mcal$ is quasi-compact and quasi-separated, and see \cite[p.212-213]{Grothendieck1960}), this follows directly from \Cref{mu level structure datum}.
	
\end{proof}

Again mirroring the affine case, this $p$-adic trivialization provides us with a canonical choice of basis for $\omega_{\A/\Mcal}$.

\begin{definition}\label{canonical basis p-adic}

	Assume that $R_0$ is an $\Ofr_{F^{\text{Gal}}}[d_F^{-1}]$-algebra. The \emph{canonical basis of $\omega_{\A/\Mcal}$}, which we denote by $\omega(\A)_{can}$, is the basis given by taking Lie algebras in the isomorphism
	$$\widehat{\G}_{m, \widehat{\Mcal}} \otimes_\Z \dfr^{-1} \xlongrightarrow{\sim} \widehat{\A}_p$$
	of \Cref{universal Poincare props}. Recall as well that for any character $\chi \in I_F$, we also obtain a trivialization
	$$\omega(\chi)_{can} \colon \Ocal_{\widehat{\Mcal}} \xlongrightarrow{\sim} \TSym^\chi(\omega_{\widehat{\A}_p/\widehat{\Mcal}}).$$

\end{definition}

\subsection{Katz's Eisenstein measure}

We now recall how Katz constructs in \cite{Katz1978} certain Hilbert modular forms which generalize the classical Eisenstein series, as well as a $p$-adic measure taking values in $V(\cfr, R)$ for any $p$-adic ring $R$ which interpolates them. We also provide a slight generalization of this measure, which will make comparisons later on more straight-forward. 

\subsubsection{Katz's Eisenstein series}

Let $R$ be an arbitrary $\Z[l^{-1}, \zeta_l]$-algebra, and let $\cfr$ be a fractional ideal of $F$ with integral inverse and which is coprime to $p$, $l$ and $R$ (recall \Cref{prime to R def}). We also fix a finite index subgroup $\Gamma \leq \Ofr^\times$. We would like to associate certain Hilbert modular forms to locally constant functions on the product
$$(\Ofr \otimes_\Z \Z_p)^2 \times (\Ofr/l\Ofr)^2,$$
where we see $\Ofr/l\Ofr$ with the discrete topology. Moreover, as we want to also take into account the action by the group $\Gamma$, we will assume that these functions satisfy the following property with respect to the $\Gamma$-action.

\begin{definition}

	Let $f \colon M \to X$ be a map of sets, where $M$ is a $\Gamma$-module and $X$ is an abelian group, and let $k$ be an integer. Then, we say that $f$ is of \emph{parity} $k$ if $f(\gamma \cdot m) = N(\gamma)^kf(m)$ for all $m \in M$ and $\gamma \in \Gamma$.

\end{definition}

Since we are now working with products, we consider them as $\Gamma$-modules through the following convention.

\begin{convention}\label{convention Gamma action}

	Unless explicitly stated, whenever we consider a product of $\Ofr$-modules of the form
	$$(\Ofr \otimes_\Z \Z_p)^2 \quad \text{or} \quad (\Ofr/p^n\Ofr)^2,$$
	for some integral ideal $\nfr$, we see it as a $\Gamma$-module via the action $\gamma \cdot (x,y) := (\gamma^{-1}x, \gamma y)$. We will also assume that $\Gamma$ acts trivially on $\Ofr/l\Ofr$.

\end{convention} 

Let us then discuss the notion of a partial Fourier transform for these functions, which we will see later agrees with the one defined in \Cref{partial Fourier alg def} for some concrete cases.

\begin{definition}\label{partial Fourier transf}\index[notation]{$PH$}

	Let $n \geq 0$ be an integer and assume that there exists a primitive $p^n$-th root of unity $\zeta \in R$ (which we thus fix). If $H \colon (\Ofr/p^n\Ofr)^2 \to R$ is a function, its \emph{partial Fourier transform} is the function
	$$PH \colon (p^{-n}\dfr^{-1}/\dfr^{-1}) \times (\Ofr / p^n\Ofr) \to R$$
	given by the formula
	$$PH([d/p^n], [y]) := \frac{1}{p^{ng}} \sum_{[x] \in \Ofr/p^n\Ofr} H([x], [y]) \cdot \zeta^{-\text{Tr}_{F/\Q}(xd)}.$$
	We can then see $PH$ as a compactly supported function on $(\dfr^{-1} \otimes_\Z \Q_p/\Z_p) \times (\Ofr \otimes_\Z \Z_p)$.
	
\end{definition}

\begin{remark}

	A simple computation then verifies the following formula for partial Fourier inversion: 
	\begin{equation}\label{eq:partial Fourier inversion}
	H([x], [y]) = \sum_{[d/p^n] \in p^{-n}\dfr^{-1} / \dfr^{-1}} PH([d/p^n], [y]) \cdot  \zeta^{\text{Tr}_{F/\Q}(xd)}.
	\end{equation}

\end{remark}

\begin{remark}

	Note that if $H$ has parity $k$, then $PH$ has as well, where the $\Gamma$-action on the product $(p^{-n}\dfr^{-1}/\dfr^{-1}) \times (\Ofr / p^n\Ofr)$ is now the standard (diagonal) one.

\end{remark}

We want to interpret these functions as being defined on the $p^n$-torsion of certain test objects $\B/R$. If $\B$ were equipped with a $\Gamma_\arith(p^n)$-level structure, this would be straight-forward by using the isomorphism
$$(\mu_{p^n} \otimes_\Z \dfr^{-1}) \times \underline{(\Ofr/p^n\Ofr)} \cong \B[p^n].$$
In general, we will work with $(\Gamma(l), \Gamma_{00}(p^N))$-test objects, meaning that we do not have such an isomorphism for $\B$. On the other hand, $\B_{(n)}$ \emph{is} canonically equipped with a $\Gamma_\arith(p^n)$-level structure by \Cref{splitting p-tors Frob twist}. In fact, if we consider the partial Fourier transform $P_\B H$ from \Cref{partial Fourier alg def}, we have two functions defined on $\B_{(n)}[p^n]$, which we now will show agree.

\begin{lemma}\label{Fourier both types agree}

	Assume that $R$ has no non-trivial idempotents and that it has a primitive $p^n$-th root of unity $\zeta \in R$, which we now fix. Then, for any function $H \colon (\Ofr / p^n\Ofr)^2 \to R$, and any $\Gamma_{00}(p^\infty)$-test object $\B/R$, we have
	$$P_\B H = PH \quad \text{as functions} \quad \B_{(n)}[p^n] \to R.$$
	
\end{lemma}

\begin{proof}
	
	Since $\Spec(R)$ is connected, we have isomorphisms
	$$\psi \colon \ker(V^n_\B)(R) \cong \Ofr/p^n\Ofr, \quad \mu_{p^n}(R) \otimes_\Z \dfr^{-1} \cong p^{-n}\dfr^{-1}/\dfr^{-1},$$
	where the latter can be written explicitly as $\zeta \otimes d' \mapsto [d'/p^n]$. By our formulas for the partial Fourier transforms, for any point $(\zeta \otimes d, s) \in (\mu_{p^n}(R) \otimes_\Z \dfr^{-1}) \times \ker(V^n_\B)(R)$ we have:
	$$P_\B H(\zeta \otimes d, s) = \frac{1}{p^{ng}} \sum_{[x] \in \Ofr/p^n\Ofr} H([x], \psi(s)) \cdot \langle V^n_\B(\zeta \otimes d), [x] \rangle_{F^n}^{-1},$$
	$$PH(\zeta \otimes d, s) = \frac{1}{p^{ng}} \sum_{[x] \in \Ofr/p^n\Ofr} H([x], \psi(s)) \cdot \zeta^{-\text{Tr}_{F/\Q}(xd)}.$$
	It follows that in order to prove these expressions agree, it is enough to show that
	$$ \langle V^n_\B(\zeta \otimes d), [x] \rangle_{F^n} = \zeta^{\text{Tr}_{F/\Q}(xd)}$$
	for all $[d] \in \dfr^{-1}/p^n\dfr^{-1}$ and $x \in \Ofr/p^n\Ofr$. 
	
	By construction, the pairing $\langle -,- \rangle_{F^n}$ becomes
	\begin{align*}
	(\mu_{p^n} \otimes_\Z \dfr^{-1}) \times \underline{(\Ofr / p^n\Ofr)} & \longrightarrow \mu_{p^n} \\
	(\zeta' \otimes s, [r]) & \longmapsto (\zeta')^{\text{Tr}(rs)}.
	\end{align*} 
	after identifying $\ker(F^n_\B)$ with $\mu_{p^n} \otimes_\Z \dfr^{-1}$ (and analogously for its Cartier dual). Using our identification of $\mu_{p^n}$ with $\underline{p^{-n}\Z/\Z}$ via the $p^n$-th root of unity $\zeta$, this pairing becomes
	\begin{align*}
	(p^{-n}\dfr^{-1}/\dfr^{-1}) \times (\Ofr/p^n\Ofr) & \longrightarrow \mu_{p^n}(R)\\
	\left(\left[ \frac{d}{p^n} \right], [x]\right) & \longmapsto \zeta^{\text{Tr}_{F/\Q}(xd)}.
	\end{align*}
	In particular, by the proof of \Cref{splitting p-tors Frob twist}, we have a commutative diagram
	$$\begin{tikzcd}
	\mu_{p^n} \otimes_\Z \dfr^{-1} \arrow[hook, swap]{d}{\beta_{(n)}} \arrow{r}{\text{id}} & \mu_{p^n} \otimes_\Z \dfr^{-1} \arrow[hook]{d}{\beta} \\
	\B_{(n)} \arrow[swap]{r}{V^n_\B} & \B,
	\end{tikzcd}$$
	which finishes our proof.

\end{proof}

Let us consider the case $R = \C$, in which we fix $e^{2\pi i /p^n}$ as our primitive root of unity. In this situation, Katz takes a different approach to defining the partial Fourier transform $PH$ as a function on data associated to the abelian variety $\B/\C$, constructing the ``partial Tate module'' of $\B$ (see \cite[Definition 3.1.12]{Katz1978}). We recall here its construction. Let $(\Lambda, \langle -,- \rangle)$ be a $\cfr$-polarized lattice in $F \otimes_\Q \C$. By the correspondences of \Cref{correspondence lattices with compl AVs}, we can associate to this pair a complex abelian variety $(\B/\C, m, \lambda, \omega(\B))$ with RM by $F$, a $\cfr$-polarization and a basis. If we further assume that $\B$ is equipped with a $\Gamma_{00}(p^n)$-level structure $\beta_n$ for some $n$, it will induce an $\Ofr$-linear injection
$$\mu_{p^n}(\C) \otimes_\Z \dfr^{-1} \hookrightarrow p^{-n}\Lambda/\Lambda.$$
Taking the isomorphism $\mu_{p^n}(\C) \cong \Z/p^n\Z$ given by the primitive $p^n$-th root of unity $e^{2\pi i /p^n}$, we may rewrite this injection as
$$\beta_n \colon \dfr^{-1} \otimes_\Z \Z/p^n\Z \hookrightarrow \Lambda \otimes \Z/p^n\Z.$$
More generally, if $\B$ has a $\Gamma_{00}(p^\infty)$-level structure, taking the limit of all of these $\beta_n$ gives an injective $\Ofr \otimes_\Z \Z_p$-linear map
$$\beta \colon \dfr^{-1} \otimes_\Z \Z_p \hookrightarrow \Lambda \otimes_\Z \Z_p.$$

On the other hand, recall that our $\cfr$-polarization induced an isomorphism $\Lambda \wedge_\Ofr \Lambda \cong \cfr^\vee$. Taking the tensor product with $\Z_p$, this becomes an isomorphism
$$(\Lambda \otimes_\Z \Z_p) \wedge_{\Ofr \otimes_\Z \Z_p} (\Lambda \otimes_\Z \Z_p) \cong \dfr^{-1}\cfr^{-1} \otimes_\Z \Z_p = \dfr^{-1} \otimes_\Z \Z_p,$$
where we are using that $\cfr$ is coprime to $p$, and thus to $\Z_p$. Combining this with the injection $\beta$, we obtain a short exact sequence
$$\begin{tikzcd}
0 \arrow{r} & \dfr^{-1} \otimes_\Z \Z_p \arrow{r}{\beta} & \Lambda \otimes_\Z \Z_p \arrow{r} & \Ofr \otimes_\Z \Z_p \arrow{r} & 0.
\end{tikzcd}$$
If we base change to $\Q_p$, it becomes
$$\begin{tikzcd}
0 \arrow{r} & \dfr^{-1} \otimes_\Z \Q_p \arrow{r}{\beta} & \Lambda \otimes_\Z \Q_p \arrow{r} & \Ofr \otimes_\Z \Q_p \arrow{r} & 0.
\end{tikzcd}$$

\begin{definition}\label{partial Tate module def}

	Let $(\Lambda, \langle -,- \rangle)$ and $\beta \colon \dfr^{-1} \otimes_\Z \Z_p \hookrightarrow \Lambda \otimes_\Z \Z_p$ a $\Gamma_{00}(p^\infty)$-level structure as above. The \emph{partial Tate module of $(\Lambda, \langle -, - \rangle, \beta)$} is the following $\Ofr \otimes_\Z \Z_p$-submodule of $\Lambda \otimes_\Z \Q_p$:
	$$PV_p(\Lambda) := \Lambda \otimes_\Z \Z_p + \beta(\dfr^{-1} \otimes_\Z \Q_p).$$

\end{definition}

\begin{remark}

	By observing that the intersection of the submodules $\Lambda \otimes_\Z \Z_p$ and $\beta(\dfr^{-1} \otimes_\Z \Q_p)$ in $\Lambda \otimes_\Z \Q_p$ is precisely $\beta(\dfr^{-1} \otimes_\Z \Z_p)$, we obtain the following two short exact sequences:
	$$\begin{tikzcd}
	0 \arrow{r} & \Lambda \otimes_\Z \Z_p \arrow{r} & PV_p(\Lambda) \arrow{r}{\text{pr}_1}  & \dfr^{-1} \otimes_\Z \Q_p/\Z_p \arrow{r}& 0,
	\end{tikzcd}$$
	$$\begin{tikzcd}
	0 \arrow{r} & \dfr^{-1} \otimes_\Z \Q_p \arrow{r}{\beta} & PV_p(\Lambda) \arrow{r}{\text{pr}_2}  & \Ofr \otimes_\Z \Z_p \arrow{r} & 0.
	\end{tikzcd}$$

\end{remark}

In particular, it follows that if $PH$ is defined on $(p^{-n}\dfr^{-1}/\dfr^{-1}) \times (\Ofr/p^n\Ofr)$, it can be extended by zero (in the first coordinate) to a compactly supported function
$$PH \colon (\dfr^{-1} \otimes_\Z \Q_p/\Z_p) \times (\Ofr \otimes_\Z \Z_p) \to \C.$$
We may then see this function as defined on $PV_p(\Lambda)$ by $PH(\lambda) := PH(\text{pr}_1(\lambda), \text{pr}_2(\lambda))$. To define the transcendental expression of the Eisenstein series, Katz considers sums over the submodule
$$\Lambda[p^{-1}] \cap PV_p(\Lambda) \subset \Lambda \otimes_\Z \Q.$$
Then, we can see that under the right hypotheses on $H$, its partial Fourier transform is supported in a specific lattice within this submodule.

\begin{proposition}\label{partial Fourier defined on lattice}

	Let $H$ be a locally constant function $(\Ofr \otimes_\Z \Z_p)^2 \to \overline{\Q}$ which is constant in cosets modulo $p^n$ and supported in the units. Consider the restriction of its partial Fourier transform $PH$ to the submodule $\Lambda[p^{-1}] \cap PV_p(\Lambda) \subset \Lambda \otimes_\Z \Q$. Then, this restriction is supported on the lattice
	$$\Lambda_{(n)}':= \Lambda + \beta(\dfr^{-1} \otimes_\Z p^{-n}\Z) \subset p^{-n}\Lambda.$$	
	Moreover, $PH$ factors through the finite quotient $\Lambda_{(n)}'/p^n\Lambda_{(n)}'$.
	
\end{proposition}

\begin{proof}

	Throughout this proof, we will use the following well-known properties of modules over a ring $R$: if $M, N, N'$ are $R$-modules which are contained in the same $R$-module, then
	
	\begin{enumerate}[(1)]
	
		\item if $M$ is flat as an $R$-module, $M \otimes_R (N \cap N') = (M \otimes_R N) \cap (M \otimes_R N')$ inside $M \otimes_R (N + N')$; and
		
		\item $(M \cap N) + (M \cap N') \subset M \cap (N + N')$. If $N \subset M$ or $N' \subset M$, this is an equality.
	
	\end{enumerate}
	Note as well that $\Lambda_{(n)}'$ is obviously a lattice in $F \otimes_\Q \C$, since we have inclusions $\Lambda \subset \Lambda_{(n)}' \subset p^{-n}\Lambda$.
	
	As an immediate consequence of the formula in \Cref{partial Fourier transf}, we know that $PH$ is supported on the space
	$$(\dfr^{-1} \otimes_\Z p^{-n}\Z_p/\Z_p) \times (\Ofr \otimes_\Z \Z_p).$$
	Seeing $PH$ as a map defined on $PV_p(\Lambda)$, this means that it is supported in the submodule $\text{pr}_1^{-1}(\dfr^{-1} \otimes_\Z p^{-n}\Z_p/\Z_p) = \Lambda \otimes_\Z \Z_p + \beta(\dfr^{-1} \otimes_\Z p^{-n}\Z_p)$. Therefore, by taking its restriction to $\Lambda[p^{-1}] \cap PV_p(\Lambda)$, it becomes a function supported in
	$$\Lambda[p^{-1}] \cap (\Lambda \otimes_\Z \Z_p + \beta(\dfr^{-1} \otimes_\Z p^{-n}\Z_p)) = \Lambda[p^{-1}] \cap (\Lambda_{(n)}' \otimes_\Z \Z_p).$$
	We claim that this intersection is equal to $\Lambda_{(n)}'$. In fact, it is enough to show that $\Lambda[p^{-1}] = \Lambda_{(n)}'[p^{-1}]$, as then the equality (1) will prove that
	$$\Lambda[p^{-1}] \cap (\Lambda_{(n)}' \otimes_\Z \Z_p) = \Lambda_{(n)}'[p^{-1}] \cap (\Lambda_{(n)}' \otimes_\Z \Z_p) = \Lambda_{(n)}' \otimes_\Z (\Z[p^{-1}] \cap \Z_p) = \Lambda_{(n)}'.$$
	Clearly, $\Lambda[p^{-1}] \subset \Lambda_{(n)}'[p^{-1}]$. On the other hand, since $\Lambda_{(n)}' \subset p^{-n}\Lambda$, we have
	$$\Lambda_{(n)}'[p^{-1}] \subset (p^{-n}\Lambda)[p^{-1}] = \Lambda[p^{-1}],$$
	showing that $\Lambda_{(n)}'[p^{-1}] \subset \Lambda[p^{-1}]$, and thus our desired equality.
	
	Lastly, recall that we know that this partial Fourier transform factors through the quotient
	$$(\dfr^{-1} \otimes_\Z p^{-n}\Z_p/\Z_p) \times (\Ofr/p^n\Ofr).$$ 
	When seen as a function in $\Lambda[p^{-1}] \cap PV_p(\Lambda)$, it factors through $\Lambda_{(n)}' / (\text{pr}_2^{-1}(\Ofr \otimes_\Z p^n\Z_p) \cap \Lambda_{(n)}')$. Explicitly, we have that
	$$\text{pr}_2^{-1}(\Ofr \otimes_\Z p^n\Z_p) = p^n\Lambda \otimes_\Z \Z_p + \beta(\dfr^{-1} \otimes_\Z \Q_p).$$
	If we express again $\Lambda_{(n)}' = \Lambda + \beta(\dfr^{-1} \otimes_\Z p^{-n}\Z)$, the fact that the submodule $\beta(\dfr^{-1} \otimes_\Z p^{-n}\Z)$ is contained in $\text{pr}_2^{-1}(\Ofr \otimes_\Z p^n\Z_p)$ allows us to use (2) to write
	$$\text{pr}_2^{-1}(\Ofr \otimes_\Z p^n\Z_p) \cap \Lambda_{(n)}' = \text{pr}_2^{-1}(\Ofr \otimes_\Z p^n\Z_p) \cap \Lambda + \text{pr}_2^{-1}(\Ofr \otimes_\Z p^n\Z_p) \cap \beta(\dfr^{-1} \otimes_\Z p^{-n}\Z).$$
	Now, if we look at each summand separately and apply the version of (2) which is just an inclusion, it follows that
	$$\text{pr}_2^{-1}(\Ofr \otimes_\Z p^n\Z_p) \cap \Lambda \supset p^n\Lambda + p^n\beta(\dfr^{-1} \otimes_\Z p^{-n}\Z) = p^n\Lambda_{(n)}'.$$
	Therefore, $\text{pr}_2^{-1}(\Ofr \otimes_\Z p^n\Z_p) \cap \Lambda_{(n)}'$ contains $p^n\Lambda_{(n)}'$, and thus $PH$ factors through $\Lambda_{(n)}'/p^n\Lambda_{(n)}'$.

\end{proof}

In fact, we can give a geometric argument for the appearance of this lattice. The lattice $\Lambda_{(n)}'$ which we have just defined corresponds to the quotient $\B_{(n)}$ together with the choice of a certain basis. For our later purposes, we are interested in a scalar multiple of it (and thus, of $\Lambda_{(n)}'$).

\begin{lemma}\label{lattice for quot by can subgrp}

	Let $(\B/\C, m, \lambda, \beta, \omega(\B))$ be a $\Gamma_{00}(p^\infty)$-test object with a basis, and $(\Lambda, \langle -,- \rangle)$ its associated $\cfr$-polarized lattice. Recall the (étale) isogeny
	$$V^n_\B \colon \B_{(n)} \to \B$$
	of \Cref{quot by can subgps general ring def}. Then, if $\omega(\B_{(n)})$ is the basis of $\omega_{\B_{(n)}/\C}$ defined as $(V^n_\B)^*(\omega(\B))$, the associated lattice to $(\B_{(n)}, m_{(n)}, \omega(\B_{(n)}))$ is $\Lambda_{(n)} := p^n\Lambda_{(n)}'$.

\end{lemma}

\begin{proof}

	Recall that the composition $F^n_\B \circ V^n_\B$ is equal to multiplication by $p^n$ in $\B_{(n)}$. We then consider the bases $\omega_1$, $\omega_2$ of $\omega_{\B_{(n)}/\C}$ obtained by the relations
	$$\omega_1 := \omega(\B_{(n)}) = (V^n_\B)^*(\omega(\B)), \quad \text{and} \quad \omega(\B) = (F^n_\B)^*(\omega_2).$$
	Then, $\omega_1 = (F^n_\B \circ V^n_\B)^*(\omega_2) = [p^n]^*(\omega_2) = p^n\omega_2$. If we write $\Lambda_i$ for the lattice associated to $(\B_{(n)}, \omega_i)$ for $i = 1, 2$, this means that $\Lambda_1 = p^n\Lambda_2$. It is then enough to show that $\Lambda_2 = \Lambda_{(n)}'$.
	
	To do this, recall that the relative Frobenius morphism was defined as the quotient map $\B \to \B/C_n = \B_{(n)}$. Therefore, $\Lambda_2$ is defined as the preimage of $C_n(\C) \cong \dfr^{-1} \otimes_\Z \mu_{p^n}$ under the universal covering map $F \otimes_\Q \C \to (F \otimes_\Q \C)/\Lambda$. An explicit computation shows that
	$$\Lambda_2 = \Lambda + \beta(\dfr^{-1} \otimes_\Z p^{-n}\Z) = \Lambda_{(n)}',$$
	which indeed concludes our argument.

\end{proof}

\begin{remark}

	Under this interpretation, we can indeed see the restriction of the partial Fourier transform $PH$ as a function on the $p^n$-torsion $\B_{(n)}[p^n](\C) \cong p^{-n}\Lambda_{(n)}/\Lambda_{(n)}$. 
	
\end{remark}

We now use this new description of the space of definition of $PH$ to slightly rewrite the following result of Katz, which is a way to construct Eisenstein series associated to functions on $(\Ofr \otimes_\Z \Z_p)^2$ and $(\Ofr/l\Ofr)^2$. We note that this can be proved through essentially the same methods as the original proof of Katz.

\begin{proposition}[{\cite[Theorem 3.2.3]{Katz1978}}]\label{Katz Eisenstein mod forms infty level}\index[notation]{$G$@$G_{k, Hf, \Gamma}$}

	Consider two functions
	$$H \colon (\Ofr\otimes_\Z \Z_p)^2 \to R, \quad f \colon (\Ofr/l\Ofr)^2 \to R,$$
	such that $H$ is constant modulo $p^n$, has parity $k$ and is supported on $(\Ofr \otimes_\Z \Z_p)^{\times, 2}$. Then, there exists a unique $\cfr$-Hilbert modular form 
	$$G_{k, Hf, \Gamma} \in M_{\underline{k}}(\cfr, \Gamma(l), \Gamma_{00}(p^\infty), R)$$
	with the following transcendental expression when $R = \C$. For any $(\Gamma(l),\Gamma_{00}(p^\infty))$-test object $\B/\C$ together with a basis $\omega(\B)$, the complex number $G_{k,Hf, \Gamma}(\B/\C, \omega(\B))$ is the value at $s=0$ of the entire function of $s$ whose expression for $\textnormal{Re}(s) > 1 - \frac{k}{2}$ is given by the absolutely convergent series
	$$\frac{(-1)^{kg}\Gamma(k+s)^g}{\sqrt{d_F}} \sideset{}{'}\sum_{\lambda \in \Gamma \backslash (lp^n)^{-1}\Lambda_{(n)}} \frac{PH([l\lambda])f(-[p^n\lambda])}{N(\lambda)^k |N(\lambda)|^{2s}},$$
	where $\Lambda$ is the associated lattice in $F \otimes_\Q \C$ to $(\B, \omega(\B))$ by \Cref{correspondence lattices with compl AVs}, and $\Lambda_{(n)}$ is the lattice associated to $(\B_{(n)}, \omega(\B_{(n)}))$ as described in \Cref{lattice for quot by can subgrp}.

\end{proposition}

\begin{remark}

	In this statement, we are seeing $PH$ and $f$ as functions
	$$PH \colon p^{-n}\Lambda_{(n)}/\Lambda_{(n)} \to \C \quad \text{and} \quad f \colon l^{-1}\Lambda_{(n)}/\Lambda_{(n)} \to \C$$
	via \Cref{partial Fourier defined on lattice} and the $\Gamma(l)$-level structure on $\B_{(n)}$. 

\end{remark}

In fact, from this description we obtain the following formulas for passing to a finite index subgroup of $\Gamma$.

\begin{corollary}\label{Katz Eis series in smaller subgps}

	Let $H, f$ be functions as in \Cref{Katz Eisenstein mod forms infty level}, and $\Gamma'$ a finite index subgroup of $\Gamma$. Then, we have that
	$$G_{k,Hf,\Gamma'} = [\Gamma : \Gamma']G_{k,Hf,\Gamma}.$$

\end{corollary}

We can use these functions in order to construct a $p$-adic version by taking the limit.

\begin{theorem}[{\cite[Theorem 3.4.1]{Katz1978}}]\label{p-adic Eisenstein series def}\index[notation]{$G$ p-adic@$\widehat{G}_{k, Hf, \Gamma}$}

	Let $R$ be a $p$-adic ring. We consider a continuous function 
	$$H \colon (\Ofr \otimes_\Z \Z_p)^2 \to R$$ 
	which is supported in $(\Ofr \otimes_\Z \Z_p)^{\times,2}$, and which is of parity $k$. Consider as well a function $f \colon (\Ofr/l\Ofr)^2 \to R$. Then, the element
	$$\widehat{G}_{k,Hf,\Gamma} := (G_{k, (H \textnormal{mod } p^n)f, \Gamma})_n \in \varprojlim_n V(\cfr, R/p^nR) \cong V(\cfr, \Gamma(l), R)$$
	is the \emph{$p$-adic Eisenstein series associated to $(H,f)$ of weight $k$}. Moreover, if $H$ is locally constant, then $\widehat{G}_{k,Hf, \Gamma}$ is the image of the function $G_{k,Hf, \Gamma}$ of \Cref{Katz Eisenstein mod forms infty level} under the map
	$$\bigoplus_{\chi \in I_{F,R}} M_\chi(\cfr, \Gamma(l), \Gamma_{00}(p^\infty), R) \to V(\cfr, \Gamma(l), R)$$
	of \Cref{map from classical to p-adic mod forms}. 

\end{theorem}

\subsubsection{The \textit{p}-adic measure}

Let us first recall some basic facts and notation about $p$-adic measures. 

\begin{definition}

	Let $X$ be a profinite topological space and $R$ a $p$-adic ring. Write $C(X,R)$ for the $R$-module of continuous maps of sets $X \to R$. The \emph{space of $R$-valued measures on $X$} is the following $R$-module:
	$$\text{Meas}(X,R) := \Hom_R(C(X,R),R).$$
	We write 
	$$\int_X H d\mu := \mu(H)$$
	for any $\mu \in \text{Meas}(X,R)$ and any continuous function $H \colon X \to R$.

\end{definition}

\begin{remark}\label{measure on intermediate algebra}

	If the ring $R$ is an $R_0$-algebra for some other $p$-adic ring $R_0$, we may also write
	$$\text{Meas}(X,R) \cong \Hom_{R_0}(C(X,R_0),R)$$
	(see \cite[(4.0.3)]{Katz1978}).

\end{remark}

It is quite straightforward to define new measures by pushforward or multiplying them by continuous functions.

\begin{definition}\label{maps of measures def}

	Let $X$ be a profinite space, $R$ a $p$-adic ring, and $\mu \in \text{Meas}(X,R)$ a measure.
	
	\begin{enumerate}[(1)]
	
		\item If $\varphi \colon X \to Y$ is a continuous map of profinite spaces, it induces a \emph{pushforward map}:
		$$\varphi_* \colon \text{Meas}(X,R) \to \text{Meas}(Y,R),$$
		given on $\mu$ and on any continuous function $H \colon Y \to R$ by $\int_Y H d(\varphi_*\mu) := \int_X (H \circ \varphi) d\mu$.
		
		\item If $H_0 \colon X \to R$ is a continuous function, we have a \emph{multiplication by $H_0$ map}:
		$$H_0 \cdot (-) \colon \text{Meas}(X,R) \to \text{Meas}(X,R),$$
		given on $\mu$ and on any continuous function $H \colon X \to R$ by $\int_X H d(H_0 \cdot \mu) := \int_X (H \cdot H_0) d\mu$. If there exists another continuous function $H_1$ such that $H_0 \cdot H_1$ is identically 1, then this map is an isomorphism.
	
	\end{enumerate}

\end{definition} 

Following the construction of Katz, we use the Eisenstein series $\widehat{G}_{1, Hf, \Gamma}$ to define $p$-adic measures taking values in the space of Hilbert modular forms. 

\begin{definition}\label{Katz Eisenstein measure - Katz def}\index[notation]{$mu Katz$@$\mu_{\Gamma, KE}(f)$}

	Consider the profinite group
	$$G := (\Ofr \otimes_\Z \Z_p)^{\times, 2}$$ 
	with the diagonal action of $\Gamma$, and fix a function $f \colon (\Ofr/l\Ofr)^2 \to R$. \emph{Katz's Eisenstein measure associated to $f$} is the measure with values on Hilbert $p$-adic modular forms 
	$$\mu_{\Gamma, KE}(f) \in \text{Meas}(G/\overline{\Gamma}, V(\cfr, \Gamma(l), R))$$ 
	given as follows: for any continuous $\Gamma$-equivariant function $H \colon (\Ofr \otimes_\Z \Z_p)^{\times, 2} \to R$, we set
	$$\int_{G/\overline{\Gamma}} H d\mu_{\Gamma, KE}(f) := \widehat{G}_{1, H'f, \Gamma},$$
	where we define $H'(x,y) := N(x)^{-1} H(x^{-1},y)$ on the units and extend it by zero to be a function on $(\Ofr \otimes_\Z \Z_p)^2$. If $f = \mathbbm{1}_{([0], [0])}$ is the characteristic function of $([0], [0])$, set $\mu_{\Gamma, KE} := \mu_{\Gamma, KE}(\mathbbm{1}_{([0], [0])})$.

\end{definition}

\begin{remark}

	The measures $\mu_{\Gamma, KE}(f)$ can also be understood as the restrictions of a more general measure $\mu_{\Gamma, KE}$ defined in the profinite group
	$$((\Ofr \otimes_\Z \Z_p)^{\times, 2} / \overline{\Gamma}) \times (\Ofr/l\Ofr)^2,$$
	where the integral of a product $H \cdot f$ is $\widehat{G}_{1, H'f, \Gamma}$ (with $H$ and $f$ as in the above definition). This is a direct generalization of the measure $\mu_N$ of \cite[Theorem 6.4.7]{Katz1976}, which is the case for elliptic curves.

\end{remark}

\begin{remark}

	This construction extends the measure constructed in \cite[Definition 4.2.5]{Katz1978}, which is precisely $\mu_{\Ofr^\times, KE}$. In the remainder of this work, we will proceed in the reverse direction to Katz's work: we provide an algebraic construction of a very similar measure, and show that this measure interpolates the Eisenstein series $\widehat{G}_{k, Hf, \Gamma}$ by the nature of its construction. This will then allow us to write our explicit comparison result of both measures.

\end{remark}

In order to simplify notation for the remainder of this work, we will drop $\Gamma$ from the notation when possible.

\begin{notation}

	Whenever $\Gamma$ is understood, we write $G_{k,Hf}$ instead of $G_{k,Hf,\Gamma}$.

\end{notation}

\section{Construction of the Eisenstein-Kronecker classes and specializations}

Following the construction of Kings-Sprang in \cite{Kings2025}, we now define the Eisenstein-Kronecker classes associated to abelian schemes with real multiplication. We then use the smooth Hodge decomposition and the unit root decomposition to show how to obtain smooth and $p$-adic Hilbert modular forms from these cohomology classes, and we relate these modular forms to the Eisenstein series $G_{k, Hf}$ and $\widehat{G}_{k, Hf}$ of Katz we discussed in the previous section. 

\subsection{Kings-Sprang's construction}

We start by recalling the construction of Kings-Sprang. We fix the following notation once and for all.

\begin{notation}\label{general setup A/S}

	Let $S$ be a noetherian scheme, and $\pi \colon \A \to S$ an abelian scheme of dimension $g$. We also consider an abstract group $\Gamma$ together with an action $f_\gamma$ on $\A$ and $g_\gamma$ on $S$. Moreover, we assume that $\pi$ is a $\Gamma$-equivariant morphism, and that for any $\gamma \in \Gamma$, the following diagrams commute:
	$$\begin{tikzcd}
	\A \arrow{r}{f_\gamma} \arrow[swap]{d}{\pi} & \A \arrow{d}{\pi} \\
	S \arrow[swap]{r}{g_\gamma} & S.
	\end{tikzcd}$$
	As we want our maps to respect the group structure, we will also impose that the induced morphism $\A \to \A \times_{S, g_\gamma} S$ is a homomorphism of $S$-group schemes (note that $f_\gamma \colon \A \to \A$ is not a morphism of $S$-schemes in general).
	
\end{notation}

The two cases where this situation will occur are when we consider an abelian scheme $\B$ with RM, and we then take the $\Gamma$-action given by its endomorphism structure; and when we work with the universal abelian scheme $\A$ over some Hilbert moduli scheme $\Mcal$, where the action of $\Gamma$ is then the one described in \eqref{eq:Gamma act on A M}.

One of the pieces of data required for our definition of Eisenstein-Kronecker classes will be a certain function $f$ defined on some closed subscheme of $\A$. In particular, we will consider the following space of functions for any such subscheme:

\begin{definition}\index[notation]{$Ocal S D 0 Gamma$@$\Ocal_S[\D]^{0, \Gamma}$}

	For $\D \subset \A$ a $\Gamma$-equivariant closed subscheme which is finite étale over $S$, write $\Ocal_S[\D] := H^0(\D, \Ocal_\D)$. Furthermore, we set 
	$$\Ocal_S[\D]^0 := \ker(H^0(\D, \Ocal_\D) \to H^0(S, \Ocal_S)),$$ 
	as the kernel of the classical trace map, and $\Ocal_S[\D]^{0, \Gamma}$ as its submodule of $\Gamma$-invariants.

\end{definition}

\begin{remark}\label{schematic decomposition of D}

	If $\delta \colon \A \to \B$ is a $\Gamma$-equivariant étale isogeny of abelian schemes over $S$, then the group scheme $\ker(\delta) \to S$ is a finite étale scheme which is stable under $\Gamma$. As a consequence, the zero section $e \colon S \to \ker(\delta)$ is an open immersion, meaning that $\ker(\delta) \setminus \{e(S)\}$ is still a closed subscheme of $\A$ which is finite and étale over $S$. It follows that, for $\D := \ker(\delta)$, we obtain a decomposition
	$$\Ocal_S[\D] \cong \Ocal_S[\D \setminus \{e(S)\}] \oplus \Ocal_S[e(S)].$$

\end{remark}

We then place ourselves in the setup given by this remark.

\begin{notation}\label{closed subscheme D notation}

	Let $\delta \colon \A \to \B$ be a $\Gamma$-invariant étale isogeny, such that its dual is étale. $\D$ will be one of the two closed and étale subschemes
	$$\D := \ker \delta, \quad \text{or} \quad \D := \ker \delta \setminus \{e(S)\}.$$
	We will also write $\Ucal_\D : = \A \setminus \D$ for the complement of $\D$.
	
\end{notation}

By construction, $\D$ is $\Gamma$-stable, and thus so is $\Ucal_\D$. We may then consider the sheaf $\coPo \otimes \Omega^g_{\A/S}$ on the open subscheme $\Ucal_\D$, which also admits a $\Gamma$-action. If we use the long exact sequence for equivariant cohomology with support, together with the vanishing of cohomology of the Poincaré bundle (see \cite[Corollary 2.17]{Kings2025}), we obtain the following exact sequence:
$$\begin{tikzcd}
0 \arrow{r} & H^{g-1}(\Ucal_\D, \Gamma ; \widehat{\Po} \otimes \Omega_{\A / S}^g) \arrow{r} & H^g_\D(\A, \Gamma ; \widehat{\Po} \otimes \Omega_{\A / S}^g) \arrow{r} & H^0(S, \Ocal_S)^\Gamma.
\end{tikzcd}$$
Kings-Sprang then show in \cite[Theorem 2.20]{Kings2025} that there exists an inclusion
$$\Ocal_S[\D]^{0,\Gamma} \hookrightarrow \ker \left( H^g_\D(\A, \Gamma ; \widehat{\Po} \otimes \Omega_{\A / S}^g)  \to H^0(S, \Ocal_S)^\Gamma \right) \cong H^{g-1}(\Ucal_\D, \Gamma ; \widehat{\Po} \otimes \Omega_{\A / S}^g),$$
leading to the following definition.

\begin{definition}\label{EK first def}\index[notation]{$EK$@$EK_{\Gamma, \A}(f)$}

	We denote the previous inclusion by
	$$EK_{\Gamma, \A} \colon \Ocal_S[\D]^{0,\Gamma} \to H^{g-1}(\Ucal_\D, \Gamma ; \widehat{\Po} \otimes \Omega_{\A / S}^g).$$
	For any function $f \in \Ocal_S[\D]^{0,\Gamma}$, its image $EK_{\Gamma, \A}(f) \in H^{g-1}(\Ucal_\D, \Gamma ; \widehat{\Po} \otimes \Omega^g_{\A/S})$ will be the \emph{equivariant Eisenstein-Kronecker class} associated to $f$. Further, via the map $\widehat{\Po} \to \widehat{\Po}^\natural$, we obtain a map of equivariant cohomology
	$$H^{g-1}(\Ucal_\D, \Gamma ; \widehat{\Po} \otimes \Omega^g_{\A/S}) \to H^{g-1}(\Ucal_\D, \Gamma ; \widehat{\Po}^\natural \otimes \Omega^g_{\A/S}),$$
	and we write $EK^\natural_{\Gamma, \A}(f)$ for the image of $EK_{\Gamma, \A}(f)$ under it. Whenever $\A$ is clear from the context, we will write $EK_\Gamma(f)$ and $EK^\natural_\Gamma(f)$.

\end{definition}

Next, recall the connection $\nabla \colon \widehat{\Po}^\natural \to \Omega^1_{\A / S} \otimes \widehat{\Po}^\natural$ on the Poincaré bundle $\widehat{\Po}^\natural$. Iterating it $a$ times for some positive integer $a$ gives 
$$\nabla^a \colon \widehat{\Po}^\natural \to \TSym^a(\Omega^1_{\A/S}) \otimes \widehat{\Po}^\natural.$$
We can thus extend our previous definition.

\begin{definition}

	Let $f \in \Ocal_S[\D]^{0, \Gamma}$ and $a \in \Z_{>0}$. The \emph{$a$-th derivative of the equivariant Eisenstein-Kronecker class} is the image of $EK^\natural_{\Gamma, \A}(f)$ under the map $\nabla^a$:
	$$\nabla^a EK^\natural_{\Gamma, \A}(f) \in H^{g-1}(\Ucal_\D, \Gamma ; \TSym^a(\Omega^1_{\A/S}) \otimes \widehat{\Po}^\natural \otimes \Omega^g_{\A/S}).$$
	Note that here we have used again the identification $\omega^g_{\A^\vee / S} \cong \omega^g_{\A/S}$.

\end{definition}

We now consider another isogeny $\varphi \colon \A \to \B'$ which is $\Gamma$-invariant and with étale dual. Let $x \colon S \to \Ucal_\D$ be a $\Gamma$-invariant $\varphi$-torsion section. We can pull back $\nabla^aEK_{\Gamma, \A}^\natural(f)$ via $x$, and we obtain
$$x^*\nabla^aEK_{\Gamma, \A}^\natural(f) \in H^{g-1}(S, \Gamma ; \TSym^a(\omega_{\A/S}) \otimes x^*\widehat{\Po}^\natural \otimes \omega^g_{\A/S}).$$
Then, we want to compose with the moment map \eqref{eq:moment map proj}:
$$\mom{\varrho^\natural}^b_x \colon x^*\widehat{\Po}^\natural \to \TSym^b(\Hscr_\A)$$
It follows that we have an element
$$\mom{\varrho^\natural}_x^b(\nabla^aEK_{\Gamma, \A}^\natural(f)) \in H^{g-1}(S, \Gamma ; \TSym^a(\omega_{\A/S}) \otimes \TSym^b(\Hscr_\A) \otimes \omega^g_{\A/S}).$$

Our construction can go further in the case that $S = \text{Spec}(R)$ is affine. Indeed, if we write
$$\Fscr^{a,b} := \TSym^a(\omega_{\A/S}) \otimes \TSym^b(\Hscr_\A) \otimes \omega^g_{\A/S},$$ 
then the degeneration of the spectral sequence given by composition of functors gives an isomorphism
\begin{equation}\label{equiv cohom affine iso}
H^{g-1}(S, \Gamma ; \Fscr^{a,b})  \cong H^{g-1}(\Gamma, H^0(S, \Fscr^{a,b})).
\end{equation}

Before giving our definition of these Eisenstein-Kronecker classes, we gather here the data that we have fixed for this construction.

\begin{notation}\label{pairs f x of EK classes general}\index[notation]{$f x$@$(f,x)$}

	We consider pairs $(f,x)$ of the following form:
	
	\begin{itemize}
	
		\item $x \colon S \to \A$ is a $\Gamma$-equivariant section of $\pi$ which is $\varphi$-torsion for a $\Gamma$-equivariant isogeny $\varphi \colon \A \to \B'$ with étale dual. 	
	
		\item $f \in \Ocal_S[\D]^{0,\Gamma}$, with $\D$ the closed subscheme of $\A$ defined as $\ker(\delta) \setminus (\{x(S)\} \cap \ker(\delta))$ for a $\Gamma$-equivariant étale isogeny $\delta \colon \A \to \B$ with étale dual. In particular, $\D$ is of the form mentioned in \Cref{closed subscheme D notation} and $x$ can be seen as a section of $\Ucal_\D \to S$.
	
	\end{itemize}
	
\end{notation}

\begin{definition}\label{EK specialized def}

	Let $(f,x)$ be a pair as in \Cref{pairs f x of EK classes general}, and let $a, b \in \Z_{>0}$. The \emph{Eisenstein-Kronecker class associated to $f$ and specialized at $x$} is the element
	$$EK^{b,a}_{\Gamma, \A}(f,x) := \mom{\varrho^\natural}^b_x(\nabla^aEK_{\Gamma, \A}^\natural(f)) \in H^{g-1}(S, \Gamma; \Fscr^{a,b}).$$
	By abuse of notation, if $S$ is affine, we write again $EK^{b,a}_{\Gamma, \A}(f,x) \in H^{g-1}(\Gamma, H^0(S, \Fscr^{a,b}))$ for the image of this class under the above isomorphism \eqref{equiv cohom affine iso}.

\end{definition}

Up to this point, we have rewritten the construction of Kings-Sprang in the slightly more general context where $S$ also admits a $\Gamma$-action. Next, we want to specialize these classes to elements in a cohomology group of the form
$$H^0(S, \TSym^{\alpha + \underline{1}}(\omega_{\A/S}) \otimes \TSym^\beta(\Hscr_\A))$$
for some characters $\alpha, \beta \in I_F^+$. Although our setup from now on (arbitrary base, only RM multiplication) diverges somewhat from that of Kings-Sprang, most of their arguments still hold with slight modifications, as we now show.

\begin{notation}

	Assume that $S$ is a $\text{Spec}(\Ofr_{F^{\text{Gal}}}[d_F^{-1}])$-scheme, and that $\A/S$ is equipped with real multiplication by $F$ as well as a $\cfr$-polarization $\lambda$, where $\cfr$ is a fractional ideal of $F$ with integral inverse such that multiplication by $N\cfr^{-1}$ is injective in $\Ocal_S$. We also fix $\Gamma$ to be a finite index subgroup of $\Ofr^\times$, which acts on $\A/S$ as described in \Cref{general setup A/S}.

\end{notation}

Recall the decomposition \eqref{eq:TSym decomposition}:
$$\TSym^a(\omega_{\A/S}) \otimes \TSym^b(\Hscr_\A) \cong \bigoplus_{\substack{ \alpha, \beta \in I_F^+ \\ |\alpha| = a, |\beta| = b}} \TSym^\alpha(\omega_{\A/S}) \otimes \TSym^\beta(\Hscr_\A).$$
We choose $\alpha, \beta$ such that $\alpha + \underline{1} - \beta = \underline{k}$ is of critical type for $\Gamma$ for some $k \in \Z_{> 0}$, and project to the $(\alpha, \beta)$-component. This gives a map 
\begin{equation}\label{proj to alpha, beta comp}
H^{g-1}(S, \Gamma; \Fscr^{a,b}) \to H^{g-1}(S, \Gamma ; \TSym^\alpha(\omega_{\A/S}) \otimes \TSym^\beta(\Hscr_\A) \otimes \omega^g_{\A/S}).
\end{equation}
By the same proof as \cite[Corollary 1.14]{Kings2025}, one can see that $\Gamma$ acts trivially on the sheaf $\Fscr^{\alpha, \beta} := \TSym^\alpha(\omega_{\A/S}) \otimes \TSym^\beta(\Hscr_\A) \otimes \omega^g_{\A/S}$. We can then apply \cite[Théorème 4.4.1]{Grothendieck1957}, which gives a (canonical) surjection
$$H^{g-1}(S, \Gamma ; \Fscr^{\alpha, \beta}) \twoheadrightarrow \bigoplus_{p + q = g-1} \Hom_\Z(H_p(\Gamma, \Z), H^q(S, \Fscr^{\alpha, \beta})).$$
Projecting to the term with $p = g-1$, we get a map
\begin{equation}\label{proj from Grothendieck spec seq}
H^{g-1}(S, \Gamma ; \Fscr^{\alpha, \beta}) \twoheadrightarrow \Hom_\Z(H_{g-1}(\Gamma, \Z), H^0(S, \Fscr^{\alpha, \beta})) \cong H^{g-1}(\Gamma, H^0(S, \Fscr^{\alpha, \beta})),
\end{equation}
where the isomorphism is the universal coefficient theorem for group cohomology (see \cite[Exercise III.1.3]{Brown1994}).

The next step follows similarly to \cite[Proposition 2.27]{Kings2025}: we consider a torsion-free subgroup $\Gamma' \leq \Gamma$ of finite index, and choose an element $\xi \in H_{g-1}(\Gamma, \Z)$ such that its image under the restriction map $\text{res}(\xi) \in H_{g-1}(\Gamma', \Z)$ is a generator. Then, taking the cap product with $\xi$ gives a map
\begin{equation}\label{cap prod with generator}
H^{g-1}(\Gamma, H^0(S, \Fscr^{\alpha, \beta})) \to H^0(S, \TSym^{\alpha + \underline{1}}(\omega_{\A/S}) \otimes \TSym^\beta(\Hscr_\A)).
\end{equation}
As the final step, note that the $\cfr$-polarization on $\A$ induces an injective map in de Rham cohomology $H^1_{dR}(\A^\vee/S) \hookrightarrow H^1_{dR}(\A/S)$, since by assumption multiplication by $N\cfr^{-1}$ is injective. 

\begin{definition}\label{EK nearly holom def}\index[notation]{$EK$ specialized@$EK^{\beta, \alpha}_{\Gamma, \A}(f,x)$}\index[notation]{$alpha beta$@$(\alpha, \beta)$}

	Let $\alpha, \beta \in I_F^+$ be such that $\alpha + \underline{1} - \beta = \underline{k}$ is of critical type for $\Gamma$ for some $k \in \Z_{>0}$. The \emph{specialized Eisenstein-Kronecker class} associated to $(\alpha, \beta)$ is the element
	$$EK_{\Gamma, \A}^{\beta, \alpha}(f,x) \in H^0(S, \TSym^{\alpha + \underline{1}}(\omega_{\A/S}) \otimes \TSym^\beta(\Hscr_{\A^\vee}))$$
	given as the image of $EK^{b,a}_{\Gamma, \A}(f,x)$ under the composition of the maps \eqref{proj to alpha, beta comp}, \eqref{proj from Grothendieck spec seq}, \eqref{cap prod with generator} and the (injective) map $\Hscr_{\A} \hookrightarrow \Hscr_{\A^\vee}$.

\end{definition}

To simplify notation further on, we may introduce a variant of these Eisenstein-Kronecker series where we replace $x$ by a function on the torsion of $\A$.

\begin{definition}

	Let $\alpha$, $\beta$ be as in the previous definition. Assume that $\Gamma$ acts trivially on the subgroup scheme $\ker(\varphi)$, for an isogeny $\varphi$ as in \Cref{pairs f x of EK classes general}. Then, for any map of sets $H \colon \ker(\varphi)(S) \to H^0(S, \Ocal_S)$, we may define the \emph{Eisenstein-Kronecker class associated to $(\alpha, \beta)$ and $(f,H)$} as
	$$EK^{\beta, \alpha}_{\Gamma, \A}(f, H) := \sum_{x \in \ker(\varphi)(S)} H(x) \cdot EK^{\beta, \alpha}_{\Gamma, \A}(f, x) \in H^0(S, \TSym^{\alpha + \underline{1}}(\omega_{\A/S}) \otimes \TSym^\beta(\Hscr_{\A^\vee})).$$

\end{definition}

From our description thus far, it is clear that the Eisenstein-Kronecker classes are stable under base change by any \textit{$\Gamma$-equivariant} map $T \to S$. What happens when we pull back along a map which is not $\Gamma$-equivariant? Consider in particular the following setup: we work over $S = \Mcal(\cfr, \Gamma(l), \Gamma_{00}(p^\infty))$, a Hilbert modular scheme of infinite level, and consider the universal abelian scheme $\A \to \Mcal$. Then, we want to study the functoriality of the Eisenstein-Kronecker classes associated to $\A/\Mcal$ after pulling back by some point $\xi \colon \Spec(R) \to \Mcal$ corresponding to a test object $\B/R$. As it turns out, if $\Gamma$ acts non-trivially on this test object, $\xi$ will \textit{not} be a $\Gamma$-equivariant map. Nevertheless, the functoriality still holds, as we will now show.

\begin{proposition}\label{EK stable under base change}

	Let $\A \to \Mcal:= \Mcal(\cfr, \Gamma(l), \Gamma_{00}(p^\infty))$ be the universal abelian scheme over the Hilbert moduli scheme of $(\Gamma(l), \Gamma_{00}(p^\infty))$-level. Let $R$ be an $\Ofr_{F^{\textnormal{Gal}}}[(ld_F)^{-1}, \zeta_l]$-algebra and consider $\xi \colon \Spec(R) \to \Mcal$ defining a test object $(\B/R, m_\B, \lambda_\B, \alpha_\B, \beta_\B)$. Consider as well a pair $(\alpha, \beta)$ as in \Cref{EK nearly holom def}, and a pair $(f,x)$ as in \Cref{pairs f x of EK classes general}, with the isogenies appearing in their definition given by $\varphi = [\ffr]$, $\delta = [\bfr]$ for some coprime integral ideals $\ffr, \bfr$ different from $\Ofr$. It then follows that:
	$$EK_{\Gamma, \B}^{\beta, \alpha}(\xi^*f, \xi^*x) = \xi^*(EK_{\Gamma, \A}^{\beta, \alpha}(f,x)).$$

\end{proposition}

\begin{proof}
	
	First, we reduce to the finite level case. Write $\Mcal_n := \Mcal(\cfr, \Gamma(l), \Gamma_{00}(p^n))$ for all $n$ big enough such that this functor is indeed represented by a scheme, as well as $\A_n \to \Mcal_n$ for the universal abelian scheme. We then have the following commutative diagram (where all the squares are cartesian):
	$$\begin{tikzcd}
	\B \arrow{d} \arrow{r} & \A \arrow{r} \arrow{d} & \A_n \arrow{d} \\
	\text{Spec}(R) \arrow[swap]{r}{\xi} & \Mcal \arrow{r} & \Mcal_n.
	\end{tikzcd}$$
	Because the map $\Mcal \to \Mcal_n$ is $\Gamma$-equivariant, if there exists a pair $(f_n, x_n)$ as in \Cref{pairs f x of EK classes general} for $\A_n$ such that after pullback to $\A$ it becomes $(f,x)$, then we would just need to show that 
	$$EK_{\Gamma, \B}^{\beta, \alpha}(\xi^*_nf_n, \xi^*_nx_n) = \xi^*_n(EK_{\Gamma, \A_n}^{\beta, \alpha}(f_n,x_n)),$$ 
	where $\xi_n \colon \text{Spec}(R) \to \Mcal_n$ is the induced point. 
	
	Since the Serre construction commutes with base change, the isogenies $[\ffr]$ and $[\bfr]$ are obtained by base change from the multiplication maps
	$$[\ffr]_n \colon \A_n \to \A_n(\ffr^{-1}) \quad \text{and} \quad [\bfr]_n \colon \A_n \to \A_n(\bfr^{-1})$$
	respectively. Consider then a ($\Gamma$-equivariant) point $x \colon \Mcal \to \A[\ffr]$. Since the finite group schemes $\A_n[\ffr] \to \Mcal_n$ are locally of finite presentation, the schemes $\Mcal_n$ are quasi-compact and quasi-separated (and the transition maps are affine), and $\A[\ffr]$ is the limit of the $\A_n[\ffr]$, the map
	$$\varinjlim_n \Hom_{\Mcal_n}(\Mcal_n, \A_n[\ffr]) \to \Hom_{\Mcal}(\Mcal, \A[\ffr])$$
	induced by the natural maps $\Hom_{\Mcal_n}(\Mcal_n, \A_n[\ffr]) \to \Hom_{\Mcal}(\Mcal, \A[\ffr])$ (given by base change) is a bijection (see \cite[Chapitre IV, Théorème 8.8.2]{Grothendieck1966}). In particular, this means that there exists some big enough $n$ such that $x$ is induced by base change from a point $x_n \colon \Mcal_n \to \A_n[\ffr]$. The final remaining fact to check is that $x_n$ is $\Gamma$-equivariant, but since the maps $\Mcal \to \Mcal_n$ are epimorphisms of schemes, this follows by some diagram-chasing.
	
	Write $n_x$ for an integer such that $x_n$ as above is defined for all $n \geq n_x$. Then, it follows that we can define the closed subschemes $\D_n := \A_n[\bfr] \setminus \{x_n(\Mcal_n)\}$, and that $\D \cong \varprojlim_{n \geq n_x} \D_n$. In particular, this means that $H^0(\D, \Ocal_\D) \cong \varinjlim_{n \geq n_x} H^0(\D_n, \Ocal_{\D_n})$. Since the trace map $H^0(\D, \Ocal_\D) \to H^0(\Mcal, \Ocal_{\Mcal})$ is defined by using the fact that $H^0(\D, \Ocal_\D)$ is a locally free $H^0(\Mcal, \Ocal_\Mcal)$-module, and localizations commute with colimits, it follows that this map is the colimit of the trace maps
	$$H^0(\D_n, \Ocal_{\D_n}) \to H^0(\Mcal_n, \Ocal_{\Mcal_n}).$$
	Moreover, since the category of $H^0(\Mcal, \Ocal_\Mcal)$-modules with a $\Gamma$-action is Grothendieck abelian, direct limits commute with kernels, and thus $\Ocal_\Mcal[\D]^0 \cong \varinjlim_{n \geq n_x} \Ocal_{\Mcal_n}[\D_n]^0$. Lastly, to show that taking invariants commutes with direct limits, by \cite[Theorem 1]{Brown1975} it is enough to show that $\Z$ admits a resolution (as a $\Z[\Gamma]$-module) by finitely generated projectives. Since \cite[Lemma 4.6]{Bannai2024} provides us with a free resolution, this shows that 
	$$\Ocal_\Mcal[\D]^{0,\Gamma} \cong \varinjlim_{n \geq n_x} \Ocal_{\Mcal_n}[\D_n]^{0, \Gamma},$$
	and we may then take $n \geq n_x$ big enough such that an $f_n \in \Ocal_{\Mcal_n}[\D_n]^{0, \Gamma}$ exists giving back $f$ after base change.
	
	Let us then replace $\A$ and $\Mcal$ by the corresponding $\A_n$, $\Mcal_n$, and remove the subscript $n$ from the notation. Since the number of possible $(\Gamma(l), \Gamma_{00}(p^n))$-level structures we can put on $(\B/R, m_\B, \lambda_\B)$ is finite, this means that its $\Gamma$-orbit is finite as well. Write then $\{\B_1 = \B, \ldots, \B_r\}$ for the orbit, and $\xi_\Gamma \colon S \to \Mcal$ for the coproduct of the maps $\xi_i \colon \text{Spec}(R) \to \Mcal$ defining the $\B_i$. In particular, $S$ is the disjoint union of $r$ copies of $\text{Spec}(R)$. If we consider the obvious $\Gamma$-action on $S$, the map $\xi_\Gamma$ \emph{is} $\Gamma$-equivariant, so by the mentioned functoriality we have that
	$$EK_{\Gamma, \B_S}^{\beta, \alpha}(\xi_\Gamma^*f, \xi_\Gamma^*x) = \xi_\Gamma^*(EK_{\Gamma, \A}^{\beta, \alpha}(f,x)).$$
	Note now that the abelian scheme $\B_S$ is clearly a disjoint union of $r$ copies of $\B$. In fact, since the $\Gamma$-action only affects the level structure, this means that all of these copies are equipped with the same real multiplication $m_\B$ and $\cfr$-polarization $\lambda_\B$. Since our construction of the Eisenstein-Kronecker classes does not depend on any level structure, it follows that all of the elements $EK_{\Gamma, \B_i}^{\beta, \alpha}(\xi_i^*f, \xi_i^*x)$ are equal to each other.
	
	Then, from the construction it is clear that we can split
	$$EK_{\Gamma, \B_S}^{\beta, \alpha}(\xi_\Gamma^*f, \xi_\Gamma^*x) = (EK_{\Gamma, \B_1}^{\beta, \alpha}(\xi_1^*f, \xi_1^*x), \ldots, EK_{\Gamma, \B_r}^{\beta, \alpha}(\xi_r^*f, \xi_r^*x)) = $$
	$$ = (EK_{\Gamma, \B}^{\beta, \alpha}(\xi^*f, \xi^*x), \ldots, EK_{\Gamma, \B}^{\beta, \alpha}(\xi^*f, \xi^*x)),$$
	and since taking the pullback of this tuple by $\xi = \xi_1$ is equivalent to projecting to the first element, this finishes the proof.

\end{proof}

By essentially the same proof, we also have a similar result when we substitute the universal abelian scheme $\A$ by its quotient $\A_{(n)}$ which we described in \Cref{quot by can subgps general ring def}.

\begin{corollary}\label{EK stable under base change Frob}

	Let $K$ be a field with characteristic not dividing $p$, $l$ nor $d_F$, and assume that $F^{\textnormal{Gal}} \subset K$. Consider, as in \Cref{EK stable under base change}, the universal abelian scheme $\A \to \Mcal := \Mcal(\cfr, \Gamma(l), \Gamma_{00}(p^\infty))_K$, a point $\xi \colon \Spec(K) \to \Mcal$, and a pair $(\alpha, \beta)$. Then, for any pair $(f,x)$ as in \Cref{pairs f x of EK classes general} for the abelian scheme $\A_{(n)}/\Mcal$ and again for the isogenies $\varphi = [\ffr]$, $\delta = [\bfr]$ (here $\ffr$, $\bfr$ are integral ideals coprime to each other and different from $\Ofr$), we have that
	$$EK^{\beta, \alpha}_{\Gamma, \B}(f_\B, x_\B) = \xi^*(EK^{\beta, \alpha}_{\Gamma, \A_{(n)}}(f,x)),$$
	where $\B := \A_{(n)} \times_{\Mcal, \eta} \Spec(K)$.  

\end{corollary}

\begin{proof}

	Going through the proof of \Cref{EK stable under base change} it is clear than in order to prove this result, it is enough to show that for big enough $m \geq 0$, there exist abelian schemes $\A'_m \to \Mcal_m := \Mcal(\cfr, \Gamma(l), \Gamma_{00}(p^m))$ and cartesian squares
	$$\begin{tikzcd}
	\A_{(n)} \arrow{r} \arrow{d} & \A'_m \arrow{d} \\
	\Mcal \arrow{r} & \Mcal_m,
	\end{tikzcd}$$
	where the bottom horizontal map is the canonical projection. 
	
	Let $\A_m \to \Mcal_m$ be the universal abelian schemes for all $m \geq 0$, and assume that $m \geq n$, which allows us to define the quotient $\A_{m, (n)}$ of $\A_m$ by the image of its $\Gamma_{00}(p^n)$-level structure. Since the $\Gamma_{00}(p^m)$-level structure of $\A$ is obtained by base change from $\A_m$, and the maps $\Mcal \to \Mcal_m$ are flat, the quotient commutes with base change and we obtain our desired commutative diagram with $\A'_m := \A_{m, (n)}$.
	
\end{proof}

\subsection{The complex specialization}

Now that we have constructed the Eisenstein-Kronecker classes, we want to focus on obtaining certain \emph{specializations}, maps which will give us Hilbert modular forms from our Eisenstein-Kronecker classes. Let us fix the following setup.

\begin{notation}\label{complex and p-adic embeddings}

	Fix a prime $p$ and field embeddings $\iota_p \colon \overline{\Q} \hookrightarrow \C_p$ and $\iota_\infty \colon \overline{\Q} \hookrightarrow \C$. 
	
\end{notation}

\subsubsection{The definition}

Fix $N \in \Z_{>0} \cup \{\infty\}$ and set $\Mcal_N = \Mcal(\cfr, \Gamma(l), \Gamma_{00}(p^N))_\C$ for the Hilbert moduli scheme of $(\Gamma(l), \Gamma_{00}(p^N))$-level defined over $\C$, and $\A_N$ for its corresponding universal abelian scheme.

The main result which we use in order to obtain our specialization is the following analogue of the splitting of the Hodge decomposition \eqref{eq:Hodge filtration ses} for abelian varieties over $\C$. If we write $\Mcal_N^{\text{an}}$ for the complex manifold $\Mcal_N(\C)$, we then have an embedding
\begin{equation}\label{eq:inclusion holom into smooth}
H^0(\Mcal^{\text{an}}_N, \Ocal_{\Mcal_N^{\text{an}}}) \subset H^0(\Mcal_N^{\text{an}}, \Cat^\infty_{\Mcal_N^{\text{an}}}),
\end{equation}
with $\Cat^\infty_{\Mcal_N^{\text{an}}}$ the sheaf of smooth functions on $\Mcal_N^{\text{an}}$. Then, we have the following isomorphism:
\begin{equation}\label{eq:Hodge decomp smooth}
	H^1_{dR}(\A^{\text{an}}_N/\Mcal_N^{\text{an}}) \otimes_\C \Cat^\infty_{\Mcal_N^{\text{an}}} \cong (\omega_{\A^{\text{an}}_N/\Mcal_N^{\text{an}}} \otimes_\C \Cat^\infty_{\Mcal_N^{\text{an}}}) \oplus (\text{Lie}(\A^{\vee, \text{an}}_N / \Mcal_N^{\text{an}}) \otimes_\C \Cat^\infty_{\Mcal_N^{\text{an}}}).
\end{equation}
Here we again write $\A_N^{\text{an}}$ to refer to the complex manifold $\A_N(\C)$.

Fix $\alpha, \beta \in I_F^+$, and write $\chi := \alpha + \underline{1} + \beta$. We consider the sheaves
$$\Gscr_N^{\alpha, \beta} :=  \TSym^{\alpha + \underline{1}}(\omega_{\A_N/\Mcal_N}) \otimes \TSym^\beta(\Hscr_{\A_N^\vee}).$$
By \eqref{eq:inclusion holom into smooth}, we have an inclusion
$$H^0(\Mcal_N, \Gscr^{\alpha, \beta}_N) \subset H^0(\Mcal_N, \Gscr^{\alpha, \beta}_N \otimes \Cat^\infty_{\Mcal_N^{\text{an}}}),$$
so we may apply the Hodge decomposition \eqref{eq:Hodge decomp smooth}, which gives elements in
$$H^0(\Mcal_N, \TSym^{\alpha + \underline{1}}(\omega_{\A_N/\Mcal_N}) \otimes \TSym^\beta(\omega_{\A_N / \Mcal_N}) \otimes \Cat^\infty_{\Mcal_N}) \cong M_{\chi}(\cfr, \Gamma(l), \Gamma_{00}(p^m), \Cat^\infty).$$
Note that we are using that we are working over a field of characteristic zero in order to identify $\TSym^{\alpha + \underline{1}}(\omega_{\A_N/\Mcal_N}) \otimes \TSym^\beta(\omega_{\A_N / \Mcal_N})$ and $\TSym^{\alpha + \underline{1} + \beta}(\omega_{\A_N/\Mcal_N})$ (recall \Cref{TSym char 0}).

\begin{definition}

	For any $\alpha, \beta \in I_F^+$, the \emph{complex specialization associated to $(\alpha, \beta)$} is the morphism
	$$\varphi_{\infty}^{\beta, \alpha} \colon H^0(\Mcal_N, \Gscr^{\alpha, \beta}_N) \to M_{\chi}(\cfr, \Gamma(l), \Gamma_{00}(p^m), \Cat^\infty)$$
	defined as the above composition. We will drop $\alpha, \beta$ from the notation when the characters are understood.

\end{definition}

\begin{notation}

	Let $\omega(\A_N)$ be an $\Ofr \otimes_\Z \Ocal_{\Mcal}$-basis of $\omega_{\A_N/\Mcal_N}$. By abuse of notation, the image of $\omega(\A_N)^{[\alpha + \underline{1}]} \otimes \omega(\A)^{[\beta]}$ under the isomorphism
	$$\TSym^{\alpha + \underline{1}}(\omega_{\A_N/\Mcal_N}) \otimes \TSym^\beta(\omega_{\A_N / \Mcal_N}) \cong \TSym^\chi(\omega_{\A_N/\Mcal_N})$$
	will be denoted by $\omega(\A_N)^{[\alpha + \underline{1}]} \otimes \omega(\A_N)^{[\beta]}$ as well. Note that a priori, this basis is \emph{not} equal to $\omega(\A_N)^{[\chi]}$. We similarly adopt this notation for bases for abelian varieties over $\C$.

\end{notation}

Consider a $\C$-point $\xi_0 \colon \Spec(\C) \to \Mcal_N$, defining a $(\Gamma(l), \Gamma_{00}(p^m))$-test object $\B_0 / \C$. In particular, we have a Hodge decomposition 
$$H^1_{dR}(\B_0/\C) \cong \omega_{\B_0/\C} \oplus \text{Lie}(\B_0^\vee/\C),$$
so we can repeat the same procedure as above to obtain a complex specialization
$$\varphi_{\infty, \B_0}^{\beta, \alpha} \colon H^0(\Spec(\C), \TSym^{\alpha + \underline{1}}(\omega_{\B_0/\C}) \otimes \TSym^\beta(\Hscr_{\B_0^\vee})) \to H^0(\Spec(\C), \TSym^{\chi}(\omega_{\B_0 / \C})).$$
Again, we omit $\alpha$ and $\beta$ from the notation when they are understood. From this description, we clearly have a commutative diagram
	\begin{equation}\label{eq:complex specialization after pullback to point}
	\begin{tikzcd}[column sep = 1.25cm]
	H^0(\Mcal_N, \Gscr^{\alpha, \beta}_N) \arrow{r}{\varphi_{\infty}^{\beta, \alpha}} \arrow[swap]{d}{\xi_0^*} & M_{\chi}(\cfr, \Gamma_{00}(p^m), \Cat^\infty) \arrow{d}{\xi_0^*} \\
	H^0(\Spec(\C), \xi_0^*\Gscr^{\alpha, \beta}_N) \arrow[swap]{r}{\varphi_{\infty, \B_0}^{\beta, \alpha}} & H^0(\Spec(\C), \TSym^{\chi}(\omega_{\B_0 / \C}))
	\end{tikzcd}
	\end{equation}
showing that the complex specialization commutes with taking the pullback to a point. Write $\Gscr_{\B_0}^{\beta, \alpha} := \xi_0^*\Gscr^{\alpha, \beta}_\C$, and fix a basis $\omega(\B_0)$ for $\B_0$. Then, if we take another point $\xi_1 \colon \Spec(\C) \to \Mcal_N$ defining a test object $\B_1/\C$, and consider an étale isogeny $\phi \colon \B_0 \to \B_1$, pullback under $\phi$ induces the following commutative diagram:
	\begin{equation}\label{eq:complex specialization after etale pullback}
	\begin{tikzcd}[row sep = 0.5cm, column sep = 1.5cm] 
	H^0(\Spec(\C), \Gscr^{\alpha, \beta}_{\B_1}) \arrow{r}{\varphi^{\beta, \alpha}_{\infty, \B_1}} \arrow[swap]{dd}{\phi^*} & H^0(\Spec(\C), \TSym^{\chi}(\omega_{\B_{1}/\C})) \arrow{dr}{\omega(\B_1)^{[\alpha + \underline{1}]} \otimes \omega(\B_1)^{[\beta]}} & \\
	& & \C, \\
	H^0(\Spec(\C), \Gscr^{\alpha, \beta}_{\B_0}) \arrow[swap]{r}{\varphi^{\beta, \alpha}_{\infty, \B_0}} & H^0(\Spec(\C), \TSym^{\chi}(\omega_{\B_{0}/\C})) \arrow["\omega(\B_0)^{[\alpha + \underline{1}]} \otimes \omega(\B_0)^{[\beta]}"']{ur} &
	\end{tikzcd}
	\end{equation}
where $\omega(\B_1)$ is the basis of $\B_1$ induced by $\phi$.

\subsubsection{Application to Eisenstein-Kronecker classes}

Using a result of Kings-Sprang, we now give an explicit description of the complex specialization of the Eisenstein-Kronecker classes we previously defined in terms of the complex Eisenstein series $G_{k, Hf}$ of Katz. We start by defining a pair $(f,x)$ for $\A_N/\Mcal_N$ as in \Cref{pairs f x of EK classes general}.

\begin{notation}\label{pairs f x of EK classes not equiv}
	
	Let $\bfr$, $\ffr$ be two coprime integral ideals which are distinct from $\Ofr$. We consider a pair $(f,x)$ consisting of:
	
	\begin{itemize} 

		\item a $\Gamma$-equivariant $[\ffr]$-torsion point $x \colon \Mcal_N \to \A_N[\ffr]$, and

		\item a $\Gamma$-invariant function $f \colon \A_N[\bfr] \setminus (\{x(\Mcal_N)\} \cap \A_N[\bfr]) \to H^0(\Mcal_N, \Ocal_{\Mcal_N}) = \C$ with trace zero (the last equality comes from the Köcher principle, see \cite[Chapter I, Proposition 4.7]{Freitag1990}).
		
	\end{itemize}
	For any $(\Gamma(l), \Gamma_{00}(p^N))$-test object $(\B,m,\lambda, \alpha, \beta)$ over $\C$, we can associate to it, by pullback, an analogous pair $(f_\B, x_\B)$ given by
	$$x_\B \in \B[\ffr](\C), \quad f_\B \in \C[\B[\bfr] \setminus \{x_\B(\C)\}]^{0,\Gamma}.$$
	
\end{notation}

Fix then a $(\Gamma(l), \Gamma_{00}(p^N))$-test object $\B/\C$, and let $\omega(\B)$ be a basis of $\B$. Write $(\Lambda, \langle -,- \rangle)$ for the $\cfr$-polarized lattice in $F \otimes_\Q \C$ associated to $(\B, \omega(\B))$ by \Cref{correspondence lattices with compl AVs}. We may then express any pair $(f_\B, x_\B)$ as in \Cref{pairs f x of EK classes not equiv} as a $\Gamma$-invariant element
$$x_\B \in \ffr^{-1}\Lambda / \Lambda$$
and a $\Gamma$-equivariant function (up to removing $0$ if $x_\B = 0$)
$$f_\B : \bfr^{-1}\Lambda / \Lambda \to \C$$
such that $\sum_{[\lambda] \in \bfr^{-1}\Lambda/\Lambda} f([\lambda]) = 0$. 

Consider then the following complex Eisenstein series, where we follow the notation set in \cite[Definition 3.24]{Kings2025}.

\begin{definition}

	Let $\Lambda'$ be a lattice in $F \otimes_\Q \C$, $t \in \Lambda' \otimes_\Z \Q$ a torsion point and $\alpha, \beta \in I_F^+$ such that $\alpha + \underline{1} - \beta = \underline{k}$ is of critical type for $\Gamma$ for some positive integer $k$. The \emph{Eisenstein series associated to $t$ and $\Lambda'$} is the following series on $s \in \C$:
	$$E^{\beta, \alpha}(t, s; \Lambda', \Gamma) := \sideset{}{'}\sum_{\lambda \in \Gamma \backslash (\Lambda' + \Gamma t)} \frac{\overline{\lambda}^\beta}{\lambda^{\alpha + \underline{1}} |N(\lambda)|^{2s}}.$$
	This series is absolutely convergent for all $s$ with $\text{Re}(s) > 1 - \frac{k}{2}$, and can be extended to an entire function on $\C$.

\end{definition}

In order to compare the Eisenstein series, we will use the following transcendental expression of the complex realization of the Eisenstein-Kronecker classes, which follows directly from a result of Kings-Sprang.

\begin{proposition}\label{result Kings-Sprang for EK classes}

	Set $N = \infty$, let $\alpha, \beta \in I_F^+$ such that $\alpha + \underline{1} - \beta$ is of critical type for $\Gamma$, and let $(\B/\C, m, \lambda, \alpha, \beta, \omega(\B))$ be a $(\Gamma(l), \Gamma_{00}(p^\infty))$-test object with a basis. Then, for any pair $(f,x)$ as in \Cref{pairs f x of EK classes not equiv}, $\varphi_{\infty, \B}(EK^{\beta, \alpha}_{\Gamma, \B}(f_\B, x_\B))$ is equal to 
	$$\frac{(-1)^{\frac{g(g-1)}{2}}\alpha!}{\langle \overline{\omega(\B)}, \omega(\B^\vee) \rangle_{dR,\B}^\beta} \sum_{t \in \Gamma \backslash (\bfr^{-1}\Lambda/\Lambda)} f_\B(-t) \cdot E^{\beta, \alpha}(t + x, 0; \Lambda, \Gamma) \cdot (\omega(\B)^{[\alpha + \underline{1}]} \otimes \omega(\B)^{[\beta]}),$$
	where $\Lambda$ is the associated lattice in $F \otimes_\Q \C$ to $(\B, m, \omega(\B))$, and $\omega(\B^\vee)$ is obtained from $\omega(\B)$ by using $\lambda$ and \eqref{dual colie algebra iso}.

\end{proposition}

\begin{proof}

	Let us write $\A := \A_\infty$ for simplicity. By the diagram \eqref{eq:complex specialization after pullback to point} together with \Cref{EK stable under base change}, it is clear that
	$$\varphi_{\infty, \B}(EK^{\beta, \alpha}_{\Gamma, \B}(f_\B, x_\B)) = \varphi_{\infty}(EK^{\beta, \alpha}_{\Gamma, \A}(f, x))(B).$$
	Although $\varphi_{\infty}(EK^{\beta, \alpha}_{\Gamma, \A}(f, x))$ is a priori a Hilbert modular form of infinite level, by construction we have that
	$$M_{\underline{k}}(\cfr, \Gamma(l), \Gamma_{00}(p^\infty), \Cat^\infty) = \varinjlim_i M_{\underline{k}}(\cfr, \Gamma(l), \Gamma_{00}(p^i), \Cat^\infty).$$
	Therefore, there exists a (finite) $m > 0$ such that $\varphi_{\infty}(EK^{\beta, \alpha}_{\Gamma, \A}(f, x)) \in M_{\underline{k}}(\cfr, \Gamma(l), \Gamma_{00}(p^m), \Cat^\infty)$.

	Recall that we assumed that our $\Gamma(l)$-level structure was symplectic and fixed its determinant, meaning that the complex manifold $\Mcal_m(\C)$ is isomorphic to a quotient of $\Hcal^g$ by a discrete subgroup of $\SL_2(\R)^g$. More concretely, any element in $M_\chi(\cfr, \Gamma(l), \Gamma_{00}(p^m), \Cat^\infty)$ can be seen as a smooth function $\Hcal^g \to \C$ satisfying certain transformation properties. We may then interpret $\varphi_\infty(EK^{\beta, \alpha}_{\Gamma, \A}(f,x))$ as such a function, with the property that evaluating at any point $\tau \in \Hcal^g$ representing $\B$ gives $\varphi_{\infty,\B}(EK^{\beta, \alpha}_{\Gamma, \B}(f_\B,x_\B))$ (note that here we see $\B$ with the $\Gamma_{00}(p^m)$-level structure induced by its infinite one). 
	
	On the other hand, the covering $\Hcal^g \twoheadrightarrow \Mcal(\cfr, \Gamma(l), \Gamma_{00}(p^\infty))(\C)$ can be described by sending an element $\tau$ to a $\cfr$-polarized lattice $\Lambda_\tau$ with the appropriate level structure (see, for example, Proposition 1.6 and its following paragraph in \cite[Chapter IX]{Geer1988}). It follows that we can see the series $ E^{\beta, \alpha}(t + x, 0; \Lambda, \Gamma)$ as a function on the points $\tau \in \Hcal^g$. One then easily checks that it is a smooth function $\Hcal^g \to \C$.
	
	After a straightforward computation, \cite[Corollary 3.28]{Kings2025} shows that the equality holds whenever $\B$ has complex multiplication by a CM extension of $F$ (note that the result of Kings-Sprang does not require any level structure assumptions on $\B$ except for the existence of $x_\B$). Fix then such an extension $L/F$. We thus have two smooth functions on $\Hcal^g$ which agree on the following subset:
	$$S_L := \{ \tau \in \Hcal^g \mid \tau \text{ defines a $\Gamma(l)$-test object over $\C$ with complex multiplication by $L$} \}.$$
	By a result of Shimura (see \cite[Proposition 24.13]{Shimura1998}), $S_L$ is precisely the intersection of $\Hcal^g$ with the image of the map $L \to \C^g$, $w \mapsto (\varsigma_1(w), \ldots, \varsigma_g(w))$, defined by choosing complex embeddings $\varsigma_i \colon L \to \C$ lifting the $\sigma_i$. Since the image of this map is dense in $\C^g$, $S_L$ is dense in $\Hcal^g$, and therefore both smooth functions are equal.

\end{proof}

By taking advantage of the fact that we have fixed a $\Gamma(l)$-level structure on our test objects, we set $\ffr = p^n\Ofr$ and $\bfr = l\Ofr$. In particular, this means that we can see the function $f$ we previously fixed as a function
$$f \colon (\Ofr/l\Ofr)^2 \to \C.$$
Set then $\A := \A_\infty$ and $\Mcal := \Mcal_\infty$ for the objects associated to the $\Gamma_{00}(p^\infty)$-level structure case. We then obtain the following comparison result between the complex realization of the Eisenstein-Kronecker classes and the Eisenstein series of Katz.

\begin{theorem}\label{EK classes comparison full level case}

	Let $\bfr = l\Ofr$ and $\ffr = p^n\Ofr$, and consider a function $f \colon (\Ofr/l\Ofr)^2 \to \C$ such that $\sum_d f(d) = 0$. Assume that $\Gamma$ acts trivially on $\A[lp^n]$. Then, for any locally constant function 
	$$H \colon (\Ofr \otimes_\Z \Z_p)^2 \to \C$$ 
	which is of parity $k$, supported on the units, and constant in cosets modulo $p^n$, we have the following equality inside $M_{\underline{k}}(\cfr, \Gamma(l), \Gamma_{00}(p^\infty), \C)$:
	$$G_{k,Hf, \Gamma} = \frac{(-1)^{\frac{g(g - 2k - 1)}{2}}}{\sqrt{d_F}} \cdot  \varphi_\infty((V^n_\A)^*EK^{\underline{0}, \underline{k} - \underline{1}}_{\Gamma, \A_{(n)}}(f, P_{\A}H)),$$
	where $(V^n_\A)^*$ is the isomorphism $H^0(\Mcal, \TSym^{\underline{k}}(\omega_{\A_{(n)}/\Mcal})) \cong H^0(\Mcal, \TSym^{\underline{k}}(\omega_{\A/\Mcal}))$ induced by the étale isogeny $V^n_{\A}$.

\end{theorem}

\begin{proof}
	
	To prove this, we may compare both sides at $(\Gamma(l), \Gamma_{00}(p^\infty))$-test objects with a basis $(\B/\C, \omega(\B))$. Let us write $\eta \colon \Spec(\C) \to \Mcal$ for the point defining $\B$. We then have the following diagram where all squares are cartesian:
	$$\begin{tikzcd}
	\B_{(n)} \arrow{r} \arrow{d} & \A_{(n)} \arrow{d} \arrow{r} & \A \arrow{d} \\
	\Spec(\C) \arrow[swap]{r}{\eta} & \Mcal \arrow[swap]{r}{\sigma} & \Mcal,
	\end{tikzcd}$$
	where $\sigma$ is the endomorphism given by the maps
	\begin{align*}
	\{(\Gamma(l), \Gamma_{00}(p^\infty))\text{-test objects over } S\} & \longrightarrow \{(\Gamma(l), \Gamma_{00}(p^\infty))\text{-test objects over } S\} \\
	\B'/S & \longmapsto \B'_{(n)}/S
	\end{align*}
	for any noetherian $\C$-scheme $S$.

	This induces a commutative diagram
	$$\begin{tikzcd}
	H^0(\Mcal, \TSym^{\underline{k}}(\omega_{\A_{(n)}/\Mcal})) \arrow{r}{(V^n_\A)^*}  \arrow[swap]{d}{\eta^*} & H^0(\Mcal, \TSym^{\underline{k}}(\omega_{\A/\Mcal})) \arrow{d}{\eta^*} \\
	H^0(\Spec(\C), \TSym^{\underline{k}}(\omega_{\B_{(n)}/\C})) \arrow[swap]{r}{(V^n_\B)^*} & H^0(\Spec(\C), \TSym^{\underline{k}}(\omega_{\B/\C})),
	\end{tikzcd}$$
	with $(V^n_\B)^*$ being defined in the same way as for $\A$. Using \eqref{eq:complex specialization after pullback to point} and \eqref{eq:complex specialization after etale pullback} together with this diagram, we see that
	$$\eta^* \circ \varphi_\infty \circ (V^n_\A)^* = \varphi_{\infty, \B} \circ \eta^* \circ (V^n_\A)^* = \varphi_{\infty, \B} \circ (V^n_\B)^* \circ \eta^* = \varphi_{\infty, \B_{(n)}} \circ \eta^*.$$
	\Cref{EK stable under base change Frob} then shows that we have
	$$\varphi_\infty((V^n_\A)^*EK^{\underline{0}, \underline{k} - \underline{1}}_{\Gamma, \A_{(n)}}(f, P_{\A}H))(\B, \omega(\B)) = \varphi_{\infty, \B_{(n)}}(EK^{\underline{0}, \underline{k} - \underline{1}}_{\Gamma, \B_{(n)}}(f_{\B_{(n)}}, P_\B H))(\omega(\B)^{[\underline{k}]}).$$
	 By \Cref{result Kings-Sprang for EK classes}, this is equal to
	$$(-1)^{\frac{g(g-1)}{2}} ((k-1)!)^g \sum_{x \in p^{-n}\Lambda_{(n)}/\Lambda_{(n)}} P_\B H(x) \sum_{t \in l^{-1}\Lambda_{(n)}/\Lambda_{(n)}} f_{\B_{(n)}}(-t) E^{\underline{0}, \underline{k} - \underline{1}}(t + x, 0 ; \Lambda_{(n)}, \Gamma),$$
	with $\Lambda_{(n)}$ the lattice associated to $\B_{(n)}$ as described in \Cref{lattice for quot by can subgrp}. 
	
	On the other hand, as we saw in \Cref{Katz Eisenstein mod forms infty level}, we may write
	$$G_{k, Hf, \Gamma}(\B, \omega(\B)) = \frac{(-1)^{kg}((k-1)!)^g}{\sqrt{d_F}} \sideset{}{'}\sum_{\lambda \in \Gamma \backslash (lp^n)^{-1}\Lambda_{(n)}} \frac{PH([l\lambda])f(-[p^n\lambda])}{N(\lambda)^k |N(\lambda)|^{2s}}\Bigg|_{s=0}.$$
	Since $(lp^n)^{-1}\Lambda_{(n)}/\Lambda_{(n)} \cong (l^{-1}\Lambda_{(n)}/\Lambda_{(n)}) \oplus (p^{-n}\Lambda_{(n)}/\Lambda_{(n)})$, the sum can be rewritten to be
	$$\sum_{x \in p^{-n}\Lambda_{(n)}/\Lambda_{(n)}} PH(x) \sum_{t \in l^{-1}\Lambda_{(n)}/\Lambda_{(n)}} f(-t) \sideset{}{'}\sum_{\lambda \in \Gamma \backslash \Lambda_{(n)}} \frac{1}{N(\lambda + x + t)^k |N(\lambda + x + t)|^{2s}} = $$
	$$ = \sum_{x \in p^{-n}\Lambda_{(n)}/\Lambda_{(n)}} PH(x) \sum_{t \in l^{-1}\Lambda_{(n)}/\Lambda_{(n)}} f(-t) \cdot E^{\underline{0}, \underline{k} - \underline{1}}(t + x, 0; \Lambda_{(n)}, \Gamma),$$
	proving the desired equality (note that $P_\B H$ is equal to $PH$ by using \Cref{Fourier both types agree}).

\end{proof}

What happens in the case where we do not have $\Gamma(l)$-level structure? In this situation, although we cannot use this level structure to define arbitrary functions $f$, one can still construct an example of a function $f$ which works in general. Let us then return momentarily to working over an arbitrary noetherian scheme $S$. Fix $(\B/S, m)$ to be an abelian scheme with RM by $F$, and assume that $\Ocal_S$ is an $\Ofr_{F^{\text{Gal}}}[d_F^{-1}]$-algebra. 

\begin{notation}\label{b f ideals notation}

	Let $\bfr$, $\ffr$ be two coprime integral ideals (different from $\Ofr$) such that $[\bfr], [\bfr]^\vee$ and $[\ffr]^\vee$ are étale isogenies. Moreover, we assume that we have a decomposition $\bfr = \bfr_1\bfr_2$ such that the $\bfr_i$ are coprime integral ideals different from $\Ofr$ and the isogenies $[\bfr_i]$ and $[\bfr_i]^\vee$ are étale. 

\end{notation}

We will now construct a pair $(f,x)$ as in \Cref{pairs f x of EK classes not equiv} for the ideals $\bfr$ and $\ffr$. For an arbitrary $[\ffr]$-torsion point $x$, the closed subscheme $\D := \B[\bfr] \setminus \{x(S)\}$ is one of the two following options:
$$\D = \begin{cases}
\B[\bfr] & \text{if } x \neq e, \\
\B[\bfr] \setminus \{e(S)\} & \text{otherwise.}
\end{cases}$$
Thus, we first construct an $f$ for the case $x \neq e$, and show how to obtain a function for the more general case from it. 

Since $[\bfr_1]$ is étale, the zero section $e \colon S \to \B[\bfr_1]$ is an open and closed immersion (see \Cref{schematic decomposition of D}) and we have a schematic decomposition
$$\B[\bfr_1] = (\B[\bfr_1] \setminus \{e(S)\}) \sqcup \{e(S) \}.$$
In particular, there exists a section $1_{e(S)} \in H^0(\B[\bfr_1], \Ocal_{\B[\bfr_1]})$ which is 1 at $e(S)$ and 0 everywhere else.

\begin{definition}\label{f for comparison with Katz}\index[notation]{$f B b$@$\widetilde{f}_{\B, [\bfr_1]}$}
	
	Following \cite[p.84]{Scheider2014} and \cite[(4.4.1)]{Kings2025}, we define the following function on $\B[\bfr_1]$:
	$$f_{\B,[\bfr_1]} := N\bfr_1 \cdot 1_{e(S)} - 1_{\B[\bfr_1]},$$
	where $1_{\B[\bfr_1]}$ is the constant unit section. This defines a function on $\Ocal_S[\B[\bfr_1]]^{0, \Gamma}$. Consider the closed subscheme 
	$$(\B[\bfr_2] \setminus \{e(S)\}) \times \B[\bfr_1] \subset \B[\bfr] \setminus \{e(S)\},$$ 
	and take the pullback of $f_{\B,[\bfr_1]}$ to it via the map given by projection to the second coordinate. Extending it by zero, this gives a function
	$$\widetilde{f}_{\B,[\bfr_1]} \in \Ocal_{S}[\B[\bfr] \setminus \{e(S)\}]^{0,\Gamma}.$$
	This means that for all $[\ffr]$-torsion sections $x$, the pairs $(\widetilde{f}_{\B,[\bfr_1]}, x)$	are as in \Cref{pairs f x of EK classes not equiv}. If $\bfr_1 = b_1\Ofr$ and $\bfr_2 = b_2\Ofr$, we write $f_{\B, [b_1]} := f_{\B, [b_1\Ofr]}$ and $\widetilde{f}_{\B, [b_1]} := \widetilde{f}_{\B,[b_1\Ofr]}$. Whenever the base abelian scheme is understood, we will drop $\B$ from the notation and write $f_{[\bfr_1]}, \widetilde{f}_{[\bfr_1]}$.
	
\end{definition}

\begin{remark}

	We note that this construction can be done for arbitrary abelian schemes with RM without needing to fix any level structure in relation to $\bfr$. 
	
\end{remark}

One final piece of notation will be regarding the locally constant functions $H$, which will appear altered in our comparison result. Consider an integer $m$ which is coprime to $p$. Then, we want to take the pullback under the multiplication map
$$[m] \colon (\Ofr \otimes_\Z \Z_p)^2 \to (\Ofr \otimes_\Z \Z_p)^2, \quad (x,y) \mapsto (mx, y/m).$$
In other words, for any function $H \colon (\Ofr \otimes_\Z \Z_p)^2 \to R$, we write $[m]^*H := H \circ [m]$.

We then return to our setup over $\C$, where we can state the following version of \Cref{EK classes comparison full level case}, which now sets a relation between the complex specialization of the Eisenstein-Kronecker classes and the Eisenstein series $G_{k, H}$ as they appear in \cite{Katz1978}.

\begin{proposition}\label{distribution relation Eisenstein complex}

	Let $\bfr_1 = b_1\Ofr$, $\bfr_2 = b_2\Ofr$ and $\ffr = p^n\Ofr$ be as in \Cref{b f ideals notation} for some $n \geq 1$. Assume that $\Gamma$ acts trivially on  $\A[p^n]$. Then, for any locally constant function 
	$$H \colon (\Ofr \otimes_\Z \Z_p)^2 \to \C$$ 
	which is of parity $k$, supported on the units, and constant in cosets modulo $p^n$, we have the following equality:
	$$b_1^gb_2^{kg} G_{k,[b_2]^*H} - b_1^gG_{k,H} + b_1^{kg}G_{k, [b_1]^*H} - (b_1b_2)^{kg}G_{k, [b_1b_2]^*H} = $$
	$$ = \frac{(-1)^{\frac{g(g - 2k - 1)}{2}}}{\sqrt{d_F}} \varphi_{\infty}((V^n_\A)^*EK^{\underline{0}, \underline{k} - \underline{1}}_{\Gamma, \A_{(n)}}(\widetilde{f}_{[b_1]}, P_{\A}H)).$$

\end{proposition}

\begin{proof}

	The proof follows by a very similar argument to the one of \Cref{EK classes comparison full level case}, where for any lattice $\Lambda$ in $F \otimes_\Q \C$ corresponding to a $\Gamma_{00}(p^\infty)$-test object $\B/\C$ with a basis, we use the following distribution relation for the function $\widetilde{f}_{[b_1]}$: for all $x \in p^{-n}\Lambda_{(n)}/\Lambda_{(n)}$, we have that
	$$\sum_{t \in \Gamma \backslash (\bfr^{-1}\Lambda_{(n)}/\Lambda_{(n)})} \widetilde{f}_{[b_1]}(-t) \cdot E^{\beta, \alpha}(t + x, 0 ; \Lambda, \Gamma) = b_1^g E^{\beta, \alpha}(x, 0; b_2^{-1}\Lambda_{(n)}, \Gamma) - $$
	$$ - b_1^g E^{\beta, \alpha}(x, 0; \Lambda_{(n)}, \Gamma) + E^{\beta, \alpha}(x, 0; b_1^{-1}\Lambda_{(n)}, \Gamma) - E^{\beta, \alpha}(x, 0; (b_1b_2)^{-1}\Lambda_{(n)}, \Gamma)$$
	(see the proof of Theorem 4.10 in \cite{Kings2025}). Then, for any $m \in \{1, b_1, b_2, b_1b_2\}$, the rescaling
	$$E^{\underline{0}, \underline{k} - \underline{1}}(x, 0; m^{-1}\Lambda_{(n)}, \Gamma) = m^{kg}E^{\underline{0}, \underline{k} - \underline{1}}(mx, 0; \Lambda_{(n)}, \Gamma)$$
	allows us to rewrite
	$$\sum_{x \in p^{-n}\Lambda_{(n)} / \Lambda_{(n)}} PH(x) \cdot E^{\underline{0}, \underline{k} - \underline{1}}(mx, 0; \Lambda_{(n)}, \Gamma) =  \sum_{x' \in p^{-n}\Lambda_{(n)} / \Lambda_{(n)}} PH(m^{-1}x') \cdot E^{\underline{0}, \underline{k} - \underline{1}}(x', 0; \Lambda_{(n)}, \Gamma).$$
	To finish, we then observe that the partial Fourier inverse of the function
	$$(p^{-n}\dfr^{-1}/\dfr^{-1}) \oplus (\Ofr / p^n\Ofr) \to \C, \quad ([d], [y]) \mapsto PH([d/m], [y/m])$$ 
	is $[m]^*H$.

\end{proof}

\subsection{The \textit{p}-adic specialization}

Next, we describe how to obtain a $p$-adic version of the specialization $\varphi_\infty$, as well as similar comparison results for the $p$-adic Eisenstein series $\widehat{G}_{k,Hf}$.

\subsubsection{The definition}

Set $\Mcal = \Mcal(\cfr, \Gamma(l), \Gamma_{00}(p^\infty))_{\C_p}$ and $\A \to \Mcal$ for the universal abelian scheme. Recall as well our notation $\widehat{\Mcal}$, $\widehat{\A}_p$ for the formal completions of these schemes, as we set in \Cref{formal compl Hilbert mod scheme notation} and \Cref{formal compl special fiber notation}.

In order to make the complex construction work in this case, we require a $p$-adic analogue of the smooth Hodge decomposition \eqref{eq:Hodge decomp smooth}. This is provided by the following result.

\begin{theorem}[{\cite[Section 7]{Katz1973a}}]\label{unit root decomposition thm}

	There exists a unique splitting stable under Frobenius:
	\begin{equation}\label{eq:unit root decomp}
	H^1_{dR}(\widehat{\A}_p / \widehat{\Mcal}) \cong \omega_{\widehat{\A}_p / \widehat{\Mcal}} \oplus U,
	\end{equation}
	where $U$ is an invertible $\Ofr \otimes \Ocal_{\widehat{\Mcal}}$-module, horizontal for the Gauß-Manin connection. This is known as the \emph{unit root decomposition}. In particular, for any $p$-adic ring $R$ and any morphism $\xi \colon \textnormal{Spec}(R) \to \Mcal$, taking the pullback of this isomorphism gives a splitting
	\begin{equation}\label{eq:unit root decomp affine}
	H^1_{dR}(\widehat{\A}_{\xi,p} / R) \cong \omega_{\widehat{\A}_{\xi,p}/R} \oplus \xi^*U,
	\end{equation}
	where $\A_\xi$ is the pullback of $\A$ along $\xi$. 

\end{theorem} 

Let $\alpha, \beta \in I_F^+$ and $\chi := \alpha + \underline{1} + \beta$, and take the sheaf 
$$\Gscr^{\alpha, \beta}_{\C_p} := \TSym^{\alpha + \underline{1}}(\omega_{\A/\Mcal}) \otimes \TSym^\beta(\Hscr_{\A^\vee})$$
on $\Mcal$. Restricting to the completion and using \eqref{eq:unit root decomp} gives maps
$$H^0(\Mcal, \Gscr^{\alpha,\beta}_{\C_p}) \to H^0(\widehat{\Mcal}, \Gscr^{\alpha,\beta}_{\C_p}|_{\widehat{\Mcal}}) \to H^0(\widehat{\Mcal}, \TSym^{\alpha + \underline{1}}(\omega_{\widehat{\A}_p/\widehat{\Mcal}}) \otimes \TSym^{\beta}(\omega_{\widehat{\A}_p/\widehat{\Mcal}})).$$
As mentioned in \Cref{TSym char 0}, and since we are working over $\C_p$, we may use the isomorphism $\TSym^{\alpha + \underline{1}}(\omega_{\widehat{\A}_p/\widehat{\Mcal}}) \otimes \TSym^{\beta}(\omega_{\widehat{\A}_p/\widehat{\Mcal}}) \cong \TSym^{\chi}(\omega_{\widehat{\A}_p/\widehat{\Mcal}})$ as the final map in our specialization.

\begin{definition}\label{p-adic specialization def}

	Let $\alpha, \beta \in I_F^+$. The \emph{$p$-adic specialization associated to $(\alpha, \beta)$} is the morphism
	$$\varphi_p^{\beta, \alpha} \colon H^0(\Mcal, \Gscr^{\alpha, \beta}_{\C_p}) \to H^0(\widehat{\Mcal}, \TSym^{\chi}(\omega_{\widehat{\A}_p / \widehat{\Mcal}}))$$
	defined as the composition of the above morphisms. We will drop $\alpha$ and $\beta$ from the notation when the characters are clear.

\end{definition}

\begin{remark}\label{p-adic isom using basis}

	Whenever we have a basis $\omega(\A)$ of $\omega_{\A/\Mcal}$, the element $\omega(\A)^{[\alpha + \underline{1}]} \otimes \omega(\A)^{[\beta]}$ induces an isomorphism $\omega(\alpha, \beta) \colon \Ocal_{\widehat{\Mcal}} \cong \TSym^{\chi}(\omega_{\widehat{\A}_p/\widehat{\Mcal}})$. We may use it to define the composition 
	$$\begin{tikzcd}[column sep = 1.25cm]
	H^0(\Mcal, \Gscr^{\alpha, \beta}_{\C_p}) \arrow{r}{\varphi^{\beta, \alpha}_p} & H^0(\widehat{\Mcal}, \TSym^{\chi}(\omega_{\widehat{\A}_p / \widehat{\Mcal}})) \arrow["\sim"', "\omega(\alpha{,} \beta )^{-1}"]{r} & V(\cfr, \Gamma(l), \C_p),
	\end{tikzcd}$$
	giving us $p$-adic Hilbert modular forms from the $p$-adic specialization. Still, note that this depends on the choice of basis, whereas $\varphi_{p}^{\beta, \alpha}$ is independent of this choice. 

\end{remark}

Consider a point $\xi_0 \colon \Spec(\C_p) \to \Mcal$ defining a test object $\B_0 / \C_p$. Then, if we consider its completion $\widehat{\B}_{0,p}$ along the special fiber of its zero section as well as the unit root decomposition \eqref{eq:unit root decomp affine}, we obtain the following version of the $p$-adic specialization:
$$\varphi_{p, \B_0}^{\beta, \alpha} \colon H^0(\Spec(\C_p), \xi_0^*\Gscr^{\alpha, \beta}_{\C_p}) \to H^0(\Spf(\C_p), \TSym^{\chi}(\omega_{\widehat{\B}_{0,p} / \C_p})).$$
As above, we write $\varphi_{p, \B_0}$ if $\alpha$ and $\beta$ are understood. As in the complex situation, we have certain functoriality properties, such as the following commutative diagram:
\begin{equation*}\label{eq:p-adic specialization after pullback to a point}
	\begin{tikzcd}[column sep = 1.25cm]
	H^0(\Mcal, \Gscr^{\alpha, \beta}_{\C_p}) \arrow{r}{\varphi_{p}^{\beta, \alpha}} \arrow[swap]{d}{\xi_0^*} & H^0(\widehat{\Mcal}, \TSym^\chi(\omega_{\widehat{\A}_p/\widehat{\Mcal}})) \arrow{d}{\xi^*_0} \\
	H^0(\Spec(\C_p), \xi_0^*\Gscr^{\alpha, \beta}_{\C_p}) \arrow[swap]{r}{\varphi_{p,\B_0}^{\beta, \alpha}} & H^0(\Spf(\C_p), \TSym^{\chi}(\omega_{\widehat{\B}_{0,p} / \C_p})).
	\end{tikzcd}
\end{equation*}

\subsubsection{Comparison with Katz's \textit{p}-adic Eisenstein series}

Directly from the results in the complex case, we obtain a comparison result between the above $p$-adic specialization and the $p$-adic Eisenstein series $\widehat{G}_{k, Hf}$. 

\begin{notation}\index[notation]{$Afr p$@$\Afr_p$}

	We fix the following pieces of data.
	
	\begin{itemize}
	
		\item Set $\Mcal := \Mcal(\cfr, \Gamma(l), \Gamma_{00}(p^\infty))_{\Afr_p}$, for $\Afr_p := \iota_p^{-1}(\Ofr_{\C_p})$, and $\A \to \Mcal$ the universal abelian scheme. We write $\Mcal_{\C_p}$, $\Mcal_\C$, and so on for any base change, and similarly for $\A$.
		
		\item Let $\bfr = l\Ofr$ and $\ffr = p^n\Ofr$ be coprime integral ideals. We fix a finite index subgroup $\Gamma \leq \Ofr^{\times, +}$ which fixes all of the points of $\A_{(n)}[lp^n]$, and consider a function $f \colon (\Ofr/l\Ofr)^2 \to \Afr_p$ such that $\sum_{t \in (\Ofr/l\Ofr)^2} f(t) = 0$. We may see it as an element 
		$$f \in \Ocal_{\Mcal}[\A_{(n)}[l] \setminus \{e_{(n)}(\Mcal)\}]^{0, \Gamma}$$ 
		by using the $\Gamma(l)$-level structure on $\A_{(n)}$.
	
		\item Fix a basis $\omega(\A)$ of $\omega_{\A/\Mcal}$, which induces another $\omega(\A_{\C_p})$ over $\C_p$ by base change. As we saw in \Cref{p-adic periods def}, we obtain some $p$-adic period $\Omega_p \in \Ofr \otimes_\Z \C_p$ relating this basis to $\omega(\A_{\C_p})_{can}$. Lastly, the Verschiebung map also gives a basis $\omega(\A_{\C_p, (n)}) := (V^n_\A)^*(\omega(\A_{\C_p}))$.
	
		\item As an argument for the Eisenstein series of Katz, we consider a locally constant function $H \colon (\Ofr \otimes_\Z \Z_p)^2 \to \Afr_p$ which is supported in the units and of parity $k$. In particular, we will assume that $H$ factors through $(\Ofr/p^n\Ofr)^2$.
		
	\end{itemize}

\end{notation}

\begin{theorem}\label{p-adic specialization is Katz Eisenstein series}

	With the setup as above, we have the following equality:
	$$\frac{1}{\Omega_p^{\underline{k}}}\widehat{G}_{k, Hf} = \frac{(-1)^{\frac{g(g-2k-1)}{2}}}{\sqrt{d_F}} \cdot \varphi_p((V^n_{\A_{\C_p}})^*EK^{\underline{0}, \underline{k}-\underline{1}}_{\Gamma, \A_{\C_p, (n)}}(f, P_{\A_{\C_p}} H))(\omega(\A_{\C_p, (n)})^{[\underline{k}]}),$$
	where $(V^n_{\A_{\C_p}})^*$ is the isomorphism induced by the Verschiebung over $\C_p$:
	$$H^0(\Mcal_{\C_p}, \TSym^{\underline{k}}(\omega_{\A_{\C_p, (n)}/\Mcal_{\C_p}})) \cong H^0(\Mcal_{\C_p}, \TSym^{\underline{k}}(\omega_{\A_{\C_p}/\Mcal_{\C_p}})).$$

\end{theorem}

\begin{proof}

	Since $H$ is locally constant, by \Cref{p-adic Eisenstein series def} it follows that $\widehat{G}_{k, Hf}$ is the image of $G_{k, Hf}$ under the map
	$$\begin{tikzcd}[column sep = small]
	M_{\underline{k}}(\cfr, \Gamma(l), \Gamma_{00}(p^\infty), \C_p) \arrow{r} & \bigoplus_{\chi} M_\chi(\cfr, \Gamma(l), \Gamma_{00}(p^\infty), \C_p) \arrow{r} & V(\cfr, \Gamma(l), \C_p),
	\end{tikzcd}$$
	where the second map is that of \Cref{map from classical to p-adic mod forms}. In fact, we claim that this composition is equal to 
	$$\begin{tikzcd}
	M_{\underline{k}}(\cfr, \Gamma(l), \Gamma_{00}(p^\infty), \C_p) \arrow{r}{\varphi_p} & H^0(\widehat{\Mcal}, \TSym^{\underline{k}}(\omega_{\widehat{\A}_p/\widehat{\Mcal}})) \arrow{r}{\sim} & V(\cfr, \Gamma(l), \C_p),
	\end{tikzcd}$$
	where the isomorphism is the map $\omega(\underline{k} - \underline{1}, \underline{0})^{-1}_{can}$ associated to the canonical basis $\omega(\A_{\C_p})_{can}$ and the pair $(\underline{k} - \underline{1}, \underline{0})$ as in \Cref{p-adic isom using basis}. This is clear if we write down the following diagram (compare with \cite[(2.6.11)]{Katz1978}):
	$$\begin{tikzcd}
	M_{\underline{k}}(\cfr, \Gamma(l), \Gamma_{00}(p^\infty), \C_p) \arrow{r}{\varphi_p} \arrow{dr} & H^0(\widehat{\Mcal}, \TSym^{\underline{k}}(\omega_{\widehat{\A}_p/\widehat{\Mcal}})) \arrow["\sim"', "\omega(\underline{k})_{can}^{-1}"]{d} \\
	 & V(\cfr, \Gamma(l), \C_p),
\end{tikzcd}$$
where the diagonal map is the restriction of the mentioned map from \Cref{map from classical to p-adic mod forms} to the summand corresponding to $\chi = \underline{k}$. Note that the maps $\omega(\underline{k} - \underline{1}, \underline{0})_{can}$ and $\omega(\underline{k})_{can}$ are equal, since they are induced by the bases $\omega(\A_{\C_p})_{can}^{[\underline{k}]} \otimes \omega(\A_{\C_p})_{can}^{[\underline{0}]}$ and $\omega(\A_{\C_p})_{can}^{[\underline{k}]}$, which are clearly the same.

	Writing it down explicitly, we have shown that $\widehat{G}_{k,Hf} = \varphi_p(G_{k, Hf})(\omega(\A_{\C_p})_{can}^{[\underline{k}]})$. By definition of the $p$-adic period $\Omega_p$, it follows that 
	$$\frac{1}{\Omega_p^{\underline{k}}} \widehat{G}_{k,Hf} = \varphi_p(G_{k, Hf})(\omega(\A_{\C_p})^{[\underline{k}]}).$$	
	We then want to compare this $p$-adic Hilbert modular form to
	$$\frac{(-1)^{\frac{g(g-2k-1)}{2}}}{\sqrt{d_F}} \cdot \varphi_p((V^n_{\A_{\C_p}})^*EK^{\underline{0}, \underline{k}-\underline{1}}_{\Gamma, \A_{\C_p, (n)}}(f, P_\A H))(\omega(\A_{\C_p, (n)})^{[\underline{k}]}).$$
	To do this, we note that it is enough to compare both sides before applying $\varphi_p$. In fact, we claim that both modular forms are in the image of the map
	$$H^0(\Mcal, \TSym^{\underline{k}}(\omega_{\A/\Mcal})) \to H^0(\Mcal_{\C_p}, \TSym^{\underline{k}}(\omega_{\A_{\C_p}/\Mcal_{\C_p}}))$$
	induced by $\iota_p$. This is clear for $G_{k, Hf}$ by the $q$-expansion principle, since both $H$ and $f$ take values in $\Afr_p$. On the other hand, for the Eisenstein-Kronecker classes this follows by the commutativity of the following diagram:
	$$\begin{tikzcd}
	H^0(\Mcal, \TSym^{\underline{k}}(\omega_{\A_{(n)}/\Mcal})) \arrow{r} \arrow[swap]{d}{(V^n_\A)^*} &  H^0(\Mcal_{\C_p}, \TSym^{\underline{k}}(\omega_{\A_{\C_p, (n)}/\Mcal_{\C_p}})) \arrow{d}{(V^n_{\A_{\C_p}})^*} \\
	H^0(\Mcal, \TSym^{\underline{k}}(\omega_{\A/\Mcal})) \arrow{r} & H^0(\Mcal_{\C_p}, \TSym^{\underline{k}}(\omega_{\A_{\C_p}/\Mcal_{\C_p}})),
	\end{tikzcd}$$
	since $EK^{\underline{0}, \underline{k}-\underline{1}}_{\Gamma, \A_{\C_p, (n)}}(f, P_{\A_{\C_p}} H)$ is the image of $EK^{\underline{0}, \underline{k}-\underline{1}}_{\Gamma, \A_{(n)}}(f, P_\A H)$ via the top horizontal map of this diagram.
	
	Our comparison has been reduced to showing that we have the following equality
	$$G_{k, Hf} = \frac{(-1)^{\frac{g(g-2k-1)}{2}}}{\sqrt{d_F}} \cdot (V^n_{\A})^*EK^{\underline{0}, \underline{k}-\underline{1}}_{\Gamma, \A_{(n)}}(f, P_\A H)$$
	inside $M_{\underline{k}}(\cfr, \Gamma(l), \Gamma_{00}(p^\infty), \Afr_p)$. Since the complex specialization for $(\underline{k}- \underline{1}, \underline{0})$ is the inclusion 
	$$M_{\underline{k}}(\cfr, \Gamma(l), \Gamma_{00}(p^\infty), \Afr_p) \subset M_{\underline{k}}(\cfr, \Gamma(l), \Gamma_{00}(p^\infty), \C),$$ 
	we may compare these modular forms over $\C$, but then \Cref{EK classes comparison full level case} shows that they agree.

\end{proof}

As in the complex case, if we drop the $\Gamma(l)$-level structure assumption, we still obtain a comparison result by using the function $\widetilde{f}_{[b_1]}$ we previously defined. 

\begin{proposition}\label{p-adic specialization is Katz no full}
	
	Let $\bfr_1 = b_1\Ofr$, $\bfr_2 = b_2\Ofr$, $\bfr = \bfr_1\bfr_2$, and $\ffr = p^n\Ofr$ be integral ideals as in \Cref{b f ideals notation}, and set $l = 1$. With the same notation as above, we have the equality
	$$\frac{1}{\Omega_p^{\underline{k}}}(b_1^gb_2^{kg} \widehat{G}_{k,[b_2]^*H} - b_1^g\widehat{G}_{k,H} + b_1^{kg}\widehat{G}_{k, [b_1]^*H} - (b_1b_2)^{kg}\widehat{G}_{k, [b_1b_2]^*H}) = $$
	$$ = \frac{(-1)^{\frac{g(g-2k-1)}{2}}}{\sqrt{d_F}} \cdot \varphi_{p}((V^n_{\A_{\C_p}})^*EK^{\underline{0}, \underline{k}-\underline{1}}_{\Gamma, \A_{\C_p,(n)}}(\widetilde{f}_{[b_1]}, P_{\A_{\C_p}}H))(\omega(\A_{\C_p, (n)})^{[\underline{k}]}),$$
	where $\widetilde{f}_{[b_1]}$ is the function of \Cref{f for comparison with Katz} defined for the abelian variety $\A_{(n)}/\Mcal$. 

\end{proposition}

\section{\textit{p}-adic measures and comparison}

In this final section, we use the construction of Kings-Sprang in \cite[Section 5]{Kings2025} to describe how to obtain a $p$-adic measure $\mu_{\Gamma, \A}(f,x)$ taking values in $V(\cfr, \Gamma(l), \C_p)$. Furthermore, by using \Cref{p-adic specialization is Katz Eisenstein series} together with an interpolation result from Kings-Sprang, we show how to relate this measure to the $p$-adic measures $\mu_{\Gamma, KE}(f)$ from \Cref{Katz Eisenstein measure - Katz def}.

Through the rest of this work, we fix the following data.

\begin{notation}\label{conditions for p-adic measure affine}
 
 	Let $R$ be a $p$-adic ring with a fixed embedding $F^{\text{Gal}} \subset K := \text{Frac}(R)$ (in particular, we assume that $R$ is an integral domain).
	 
	\begin{itemize}

		\item $(\B / R, m, \lambda, \alpha, \beta, \omega(\B))$ is a $(\Gamma(l), \Gamma_{00}(p^\infty))$-test object with a basis, where the polarization ideal $\cfr$ has integral inverse, and is such that $N\cfr$ is invertible in $R$. In particular, we have a canonical choice of basis $\omega(\B)_{can}$ as in \Cref{canonical basis affine def} (which a priori is different from $\omega(\B)$).
		
		\item $\Gamma \leq \Ofr^{\times}$ is a finite index subgroup, which naturally acts on $\B$ by the RM structure. 
	
		\item $\bfr, \ffr$ are integral ideals of $F$ which are coprime to $p$ and to each other (which in particular means that $[\ffr]^\vee$, $[\bfr]^\vee$ and $[\bfr]$ are étale isogenies over $R$), 
	
		\item $x$ and $f$ are as in \Cref{pairs f x of EK classes general} for the morphisms $\varphi = [\ffr]$ and $\delta = [\bfr]$. In particular, $\D = \B[\bfr] \setminus \{x(R)\}$, and $\Ucal_\D = \B \setminus \D$.
	
	\end{itemize}
	
\end{notation}

\subsection{Construction of the \textit{p}-adic Eisenstein measure}

Following the approach of Kings-Sprang, we show in this section how to construct a $p$-adic measure using the Eisenstein-Kronecker classes. 

\subsubsection{\textit{p}-adic theta functions} 

Let us go through the procedure described in \cite[Section 5.3]{Kings2025} to obtain a $p$-adic theta function associated to the test object $\B/R$. Consider the Eisenstein-Kronecker class associated to $f$ (cf.~\Cref{EK first def}):
$$EK_{\Gamma, \B}(f) \in H^{g-1}(\Ucal_\D, \Gamma; \widehat{\Po} \otimes \Omega^g_{\B / R}).$$
Let $\varphi \colon \B \to \B'$ be some $\Gamma$-equivariant isogeny with étale dual and $y \in \B(R)$ a $\varphi$-torsion section (note that we do not assume $y$ to be $\Gamma$-equivariant). If we write $y_p \colon \text{Spec}(R/pR) \to \B_{R/pR}$ for its special fiber, we can take the formal completion $\widehat{\B}_y$ associated to this closed immersion. Since this comes equipped with a map $\widehat{\B}_y \to \B$, we may restrict along it:
$$EK_{\Gamma, \B}(f)|_{\widehat{\B}_y} \in H^{g-1}(\widehat{\B}_y, \Gamma; (\widehat{\Po} \otimes \Omega^g_{\B / R})|_{\widehat{\B}_y}).$$
The formal scheme $\widehat{\B}_y$ is affine, which means that 
$$H^{g-1}(\widehat{\B}_y, \Gamma; (\widehat{\Po} \otimes \Omega^g_{\B / R})|_{\widehat{\B}_y}) \cong H^{g-1}(\Gamma, H^0(\widehat{\B}_y, (\widehat{\Po} \otimes \Omega^g_{\B/R})|_{\widehat{\B}_y})).$$
The translation map $T_y$ induces an isomorphism $\widehat{\B}_p \xrightarrow{\sim} \widehat{\B}_y$, so taking the pullback along this map gives another isomorphism
$$T_y^* \colon H^0(\widehat{\B}_y, (\widehat{\Po} \otimes \Omega^g_{\B/R})|_{\widehat{\B}_y}) \xlongrightarrow{\sim} H^0(\widehat{\B}_p, (T_y^*\widehat{\Po} \otimes \Omega^g_{\B/R})|_{\widehat{\B}_p}).$$
Following our Eisenstein-Kronecker class through these isomorphisms, and applying the isomorphism \eqref{eq:Poincare isom triv}, we obtain an element
$$\widehat{\varrho}_yT_y^*(EK_{\Gamma, \B}(f)|_{\widehat{B}_y}) \in H^{g-1}(\Gamma, H^0(\widehat{\B}_p, \text{pr}_{\widehat{\B}_p, *}\Ocal_{\widehat{\B}_p \times \widehat{\B}^\vee_p} \otimes \Omega^g_{\B/R}|_{\widehat{\B}_p})).$$
By the projection formula, we obtain an isomorphism 
$$\text{pr}_{\widehat{\B}_p, *}\Ocal_{\widehat{\B}_p \times \widehat{\B}^\vee_p} \otimes (\Omega^g_{\B/R})|_{\widehat{\B}_p} \cong \text{pr}_{\widehat{\B}_p,*}\left( \Ocal_{\widehat{\B}_p \times \widehat{\B}^\vee_p} \otimes \text{pr}_{\widehat{\B}_p}^*(\Omega^g_{\B/R}|_{\widehat{\B}_p})\right),$$
and thus
$$\widehat{\varrho}_yT_y^*(EK_{\Gamma, \B}(f)|_{\widehat{B}_y}) \in H^{g-1}(\Gamma, H^0(\widehat{\B}_p \times \widehat{\B}^\vee_p,  \Ocal_{\widehat{\B}_p \times \widehat{\B}^\vee_p} \otimes \text{pr}_{\widehat{\B}_p}^*(\Omega^g_{\B/R}|_{\widehat{\B}_p}))).$$
To simplify the group cohomology, we may just take the cap product with some generator of group homology in degree $g-1$.

\begin{lemma}[{\cite[Lemma 5.15]{Kings2025}}] \label{cap product p-adic}

	Let $\Gamma'\leq \Gamma$ be a torsion-free subgroup of finite index and $\xi \in H_{g-1}(\Gamma, \Z)$ be an element such that $\textnormal{res}(\xi) \in H_{g-1}(\Gamma', \Z)$ is a generator. Then, the cap-product with $\xi$ induces a canonical homomorphism
	$$H^{g-1}(\Gamma, H^0(\widehat{\B}_p \times \widehat{\B}^\vee_p, \Ocal_{\widehat{\B}_p \times \widehat{\B}^\vee_p} \otimes \textnormal{pr}_{\widehat{\B}_p}^*(\Omega^g_{\B/R}|_{\widehat{\B}_p}))) \to H^0(\widehat{\B}_p \times \widehat{\B}^\vee_p, \Ocal_{\widehat{\B}_p \times \widehat{\B}^\vee_p} \otimes \TSym^{\underline{1}}(\omega_{\widehat{\B}_p/R}))_\Gamma.$$
	Assume now that $\Gamma$ is contained in $\Ofr^{\times, +}$. The canonical basis $\omega(\B)_{can}$ induces a $\Gamma$-equivariant isomorphism $\TSym^{\underline{1}}(\omega_{\widehat{\B}_p/R}) \cong R$, and thus we can see the image of this cap product map inside
	$$H^0(\widehat{\B}_p \times \widehat{\B}^\vee_p, \Ocal_{\widehat{\B}_p \times \widehat{\B}^\vee_p})_\Gamma.$$

\end{lemma} 

Combining all of these steps, we have obtained our desired theta function whenever $\Gamma$ is composed of totally positive units.

\begin{definition}

	Let $\varphi \colon \B \to \B'$ be a $\Gamma$-equivariant isogeny with étale dual and $y$ a $\varphi$-torsion point, and assume that $\Gamma$ is contained in $\Ofr^{\times, +}$. The \emph{$p$-adic theta function} associated with $EK_{\Gamma, \B}(f)$ at $y$ is the image of $\widehat{\varrho}_yT_y^*(EK_{\Gamma, \B}(f)|_{\widehat{B}_y})$ under the map of \Cref{cap product p-adic}, which we denote by
	$$\vartheta_{\Gamma, \B}(f,y) \in H^0(\widehat{\B}_p \times \widehat{\B}^\vee_p, \Ocal_{\widehat{\B}_p \times \widehat{\B}^\vee_p})_\Gamma.$$

\end{definition}

It follows directly from a group cohomology computation that passing to a smaller subgroup multiplies the theta function by a scalar.

\begin{lemma}[{\cite[Lemma 5.17]{Kings2025}}]\label{theta functs in smaller subgp}

	Assume that $\Gamma$ is contained in $\Ofr^{\times, +}$ and that $\Gamma' \leq \Gamma$ is a finite index subgroup. Then, for any $\varphi$-torsion point $y$ as above, we have that
	$$\vartheta_{\Gamma', \B}(f,y) = [\Gamma : \Gamma']\vartheta_{\Gamma,\B}(f,y)$$
	as elements in $H^0(\widehat{\B}_p \times \widehat{\B}^\vee_p, \Ocal_{\widehat{\B}_p \times \widehat{\B}^\vee_p})_{\Gamma}$.

\end{lemma}

In fact, this can be used to extend our definition of theta function to those finite index subgroups $\Gamma$ which are not necessarily totally positive. We may do this by taking the theta function for a subgroup $\Gamma' \leq \Gamma$ of finite index which is totally positive and dividing by $[\Gamma : \Gamma']$, which is of course only possible whenever this index is invertible in $R$. If we choose $\Gamma' := \Gamma \cap \Ofr^{\times, +}$, then this index is always a power of 2, which will be invertible in $R$ if we set the following assumption.

\begin{convention}

	From now on we assume $p \neq 2$.

\end{convention}

We may then define the theta function for $\Gamma$ directly in terms of the one for $\Gamma \cap \Ofr^{\times, +}$.

\begin{definition}

	Let $\Gamma$ be an arbitrary finite index subgroup of $\Ofr^\times$, $\varphi \colon \B \to \B'$ a $\Gamma$-equivariant isogeny with étale dual, and $y$ a $\varphi$-torsion point. The \emph{$p$-adic theta function} associated with $EK_{\Gamma, \B}(f)$ at $y$ 
	$$\vartheta_{\Gamma, \B}(f,y) \in H^0(\widehat{\B}_p \times \widehat{\B}^\vee_p, \Ocal_{\widehat{\B}_p \times \widehat{\B}^\vee_p})_\Gamma$$
	is the image of $[\Gamma : \Gamma \cap \Ofr^{\times, +}]^{-1}\vartheta_{\Gamma \cap \Ofr^{\times, +}, \B}(f,y) $ under the map 
	\begin{equation}\label{eq:subgroup cohom inclusion}
		H^0(\widehat{\B}_p \times \widehat{\B}^\vee_p, \Ocal_{\widehat{\B}_p \times \widehat{\B}^\vee_p})_{\Gamma \cap \Ofr^{\times, +}} \to H^0(\widehat{\B}_p \times \widehat{\B}^\vee_p, \Ocal_{\widehat{\B}_p \times \widehat{\B}^\vee_p})_{\Gamma}
	\end{equation}
	given by the inclusion $\Gamma \cap \Ofr^{\times, +} \subset \Gamma$. 
	
\end{definition}

\begin{remark}

	If $p = 2$, the element $[\Gamma : \Gamma \cap \Ofr^{\times, +}]^{-1}\vartheta_{\Gamma \cap \Ofr^{\times, +}, \B}(f,y)$ is not contained a priori in $H^0(\widehat{\B}_p \times \widehat{\B}^\vee_p, \Ocal_{\widehat{\B}_p \times \widehat{\B}^\vee_p})_{\Gamma \cap \Ofr^{\times, +}}$, and one should replace \eqref{eq:subgroup cohom inclusion} by the map
	$$H^0(\widehat{\B}_p \times \widehat{\B}^\vee_p, \Ocal_{\widehat{\B}_p \times \widehat{\B}^\vee_p})_{\Gamma \cap \Ofr^{\times, +}} \otimes_R K \to H^0(\widehat{\B}_p \times \widehat{\B}^\vee_p, \Ocal_{\widehat{\B}_p \times \widehat{\B}^\vee_p})_{\Gamma} \otimes_R K.$$
 	One can then define $\vartheta_{\Gamma, \B}(f,y)$ as the image of $[\Gamma : \Gamma \cap \Ofr^{\times, +}]^{-1}\vartheta_{\Gamma \cap \Ofr^{\times, +}, \B}(f,y)$ under this map, and proceed through the following sections in the same manner.

\end{remark}

\subsubsection{Obtaining a measure from the theta function}

Now that we have constructed our $p$-adic theta function from our Eisenstein-Kronecker classes, we want to use the level structure on $\B$ to obtain a $p$-adic measure, which we will relate later to Katz's Eisenstein measure. 

Since $N\cfr$ is invertible in $R$, recall that, by the construction \eqref{eq:dual ab sch structure}, $\B^\vee/R$ is equipped with a $\Gamma_{00}(p^\infty)$-test object structure, and in particular with a canonical basis $\omega(\B^\vee)_{can}$.

\begin{definition}

	Let $\alpha, \beta \in I_F$ be characters. Recall the moment map for the formal group $\widehat{\B}_p \times \widehat{\B}^\vee_p$:
	$$\text{mom}_{\widehat{\B}_p \times \widehat{\B}^\vee_p} \colon \Ocal_{\widehat{\B}_p \times \widehat{\B}^\vee_p} \to \widehat{\TSym^\bullet}(\omega_{\widehat{\B}_p \times \widehat{\B}^\vee_p / R}).$$
	We define differential operators
	$$\partial(\B)^{\alpha} \partial(\B^\vee)^\beta \colon H^0(\widehat{\B}_p \times \widehat{\B}^\vee_p, \Ocal_{\widehat{\B}_p \times \widehat{\B}^\vee_p}) \to R$$
	given on any global section $h$ of $\Ocal_{\widehat{\B}_p \times \widehat{\B}^\vee_p}$ via the formula
	$$\text{mom}_{\widehat{\B}_p \times \widehat{\B}^\vee_p}(h) = \left( (\partial(\B)^\alpha \partial(\B^\vee)^\beta h) \cdot \omega(\B)_{can}^{[\alpha]} \omega(\B^\vee)_{can}^{[\beta]} \right)_{\alpha, \beta}.$$

\end{definition}

\begin{remark}

	By composing the moment map with the isomorphism \eqref{dual colie algebra iso} induced by the $\cfr$-polarization $\lambda$ of $\B$, which in particular sends $\omega(\B^\vee)_{can}$ to $\omega(\B)_{can}$ (recall the diagram \eqref{eq:relation canonical bases}), we can consider these operators to be acting on $\widehat{\B}_p \times \widehat{\B}_p$, and we then write
	$$\partial(\B)^{\alpha}\partial(\B^\vee)^\beta = \partial(\B)^{\alpha} \otimes \partial(\B)^{\beta}.$$

\end{remark}

Similarly, we also define the functions which we integrate to obtain our moments. Since we have an isomorphism $\Ofr \otimes_\Z \Z_p \cong \Z_p^g$ of $\Z_p$-modules, they can be defined as the powers of the obvious maps.

\begin{definition}\index[notation]{$t alpha s beta$@$\ts^\alpha \st^\beta$}

	Let $\alpha, \beta \in I_F^+$. Consider the continuous function (recall that we fixed an embedding $F^{\text{Gal}} \subset \text{Frac}(R)$)
	\begin{align*}
	\ts^\alpha \st^\beta \colon (\Ofr \otimes_\Z \Z_p) \times (\Ofr \otimes_\Z \Z_p) & \longrightarrow R \\
	(x \otimes r, y \otimes s) & \longmapsto \prod_{i = 1}^g (r \sigma_i(x))^{\alpha_i} \cdot \prod_{i = 1}^g (s \sigma_i(y))^{\beta_i}.
	\end{align*}
	We may also extend this definition to general $\alpha, \beta \in I_F$ by using this formula restricted to the space $(\Ofr \otimes_\Z \Z_p)^\times \times (\Ofr \otimes_\Z \Z_p)^\times$ and then extending by zero.

\end{definition}

We next want to describe the global sections of $\widehat{\B}_p \times \widehat{\B}^\vee_p$ as a space of $R$-valued measures. This will be a consequence of the $\Gamma_{00}(p^\infty)$-level structure which we imposed on $\B$, which will make this formal group into a formal torus. Still, we want our isomorphism to also respect the action by $\Gamma$, so we now describe the following $\Gamma$-action on the space of measures $\text{Meas}((\Ofr \otimes_\Z \Z_p)^2, R)$. Note that as in the CM case (see \cite[Proposition 5.6]{Kings2025}), we have an isomorphism
$$(\Ofr \otimes_\Z \Z_p)^2 \cong T_p\widehat{\B}_p^t \times T_p(\widehat{\B}_p^\vee)^t.$$
Indeed, the formal completions can be computed as
$$\widehat{\B}_p^t \cong \varinjlim_n C_n^t(\overline{R}), \quad \text{and} \quad (\widehat{\B}_p^\vee)^t \cong \varinjlim_n (C_n^\vee)^t(\overline{R}),$$
with $\overline{R}$ an algebraic closure of $R$, and $C_n$, $C_n^\vee$ the canonical subgroups of \Cref{quot by can subgps general ring def}. In particular, this means that by our \Cref{signs of dual} the $\Gamma$-action on the first term is given by $\gamma^{-1}$ for all $\gamma \in \Gamma$, since it is given by multiplication by $\gamma$ in $C_n$. Similarly, in $(C_n^\vee)^t$ we apply this convention twice, so the $\Gamma$-action is the standard one. Therefore, from this geometric interpretation, we can see that for any $\gamma \in \Gamma$ and any pair $(x,y) \in (\Ofr \otimes_\Z \Z_p)^2$, 
$$\gamma \cdot (x,y) = (\gamma^{-1}x, \gamma y),$$
which agrees exactly with \Cref{convention Gamma action}. We thus will see $\text{Meas}((\Ofr \otimes_\Z \Z_p)^2,R)$ as a $\Gamma$-module with this induced action by pushforward: for any measure $\mu$ and any continuous function $H \colon (\Ofr \otimes_\Z \Z_p)^2 \to R$, we write
$$\int_{(\Ofr \otimes_\Z \Z_p)^2}Hd(\gamma \cdot \mu) := \int_{(\Ofr \otimes_\Z \Z_p)^2}(\gamma \cdot H) d\mu := \int_{(\Ofr \otimes_\Z \Z_p)^2}H(\gamma^{-1} \cdot (-), \gamma \cdot (-))d\mu.$$

\begin{proposition}\label{global sections are measures}

	For any $r \geq 0$, there is a $\Gamma$-equivariant isomorphism of rings
	$$H^0(\widehat{\B}_p \times \widehat{\B}^\vee_p, \Ocal_{\widehat{\B}_p \times \widehat{\B}^\vee_p}) \xlongrightarrow{\sim} \textnormal{Meas}((\Ofr \otimes_\Z \Z_p)^2, R),$$
	where the image $\mu_h$ of some global section $h$ is uniquely determined by the integration formulas
	$$\int_{(\Ofr \otimes_\Z \Z_p)^2} \ts^\alpha \st^\beta d\mu_h = \partial(\B)^\alpha \partial(\B^\vee)^\beta h = (\partial(\B)^{\alpha} \otimes \partial(\B)^{\beta})h$$
	for all $\alpha, \beta \in I_F^+$. Moreover, this construction commutes with arbitrary extensions of scalars $R \to R'$ of $p$-adic rings. 

\end{proposition}

\begin{proof}

	This is simply the result \cite[Proposition 5.20]{Kings2025}, which follows from \cite[Theorem 1]{Katz1981a}. As in the proof of the former, the key observation is the fact that \Cref{mu level structure datum} implies that
	$$\widehat{\B}_p \cong \widehat{\G}_m \otimes_{\Z_p} (\Ofr \otimes_\Z \Z_p),$$
	and similarly for $\widehat{\B}^\vee_p$. 

\end{proof}

\begin{remark}

	Although we have been working up to this point mostly with the formal completions $\widehat{\B}_p$ and not $\widehat{\B}$ (which were the completions along the ``whole'' zero section), by the isomorphism \eqref{eq:both formal compls agree} we know that the ring of global sections appearing in this result is the same (sans topology) in both cases.

\end{remark}

If we take the coinvariants of the isomorphism in \Cref{global sections are measures}, the image of the theta function $\vartheta_{\Gamma, \B}(f,y)$ will only determine a measure up to an equivalence class. Instead of choosing a representative, we may instead use the following result, by which we will write these equivalence classes as measures on a profinite quotient.

\begin{lemma}\label{coinv of measures are measures in quot}

	Let $G$ be a profinite group, and $H$ a normal closed subgroup which is finitely topologically generated. Then, there exists a canonical isomorphism
	$$\textnormal{Meas}(G,R)_H \xrightarrow{\sim} \textnormal{Meas}(G/H,R),$$
	which commutes with continuous maps $R \to S$, where $S$ is also a $p$-adic ring.

\end{lemma}

\begin{proof}

	First, note that we have isomorphisms
	\begin{align*}
	\text{Meas}(G, R) & \cong R \llbracket G \rrbracket = \varprojlim_U R[G / U], \text{ and}\\
	\text{Meas}(G/H, R) & \cong R \llbracket G/H \rrbracket = \varprojlim_U R[(G/U) / (H / H \cap U)],
	\end{align*}
	where in both limits, $U$ runs over all open normal subgroups of $G$. For simplicity, we write $G_U := G / U$, and $H_U := H / H \cap U$, which are finite (and discrete) groups for any such $U$. We thus show first that there is an isomorphism $\text{Meas}(G_U, R)_{H_U} \xrightarrow{\sim} \text{Meas}(G_U / H_U, R)$. Consider the obvious morphism of rings 
	$$\varphi_U : R[G_U] \to R[G_U / H_U], \quad \sum_i r_i \cdot g_i \mapsto \sum_i r_i \cdot g_iH_U. $$
	It is immediate that this map is surjective. We then want to show that its kernel is equal to the ideal $I_U$ generated by the elements $h-1$ for all $h \in H_U$. It is clear that $I_U \subset \ker(\varphi_U)$, since $\varphi_U(h-1) = 0$ for all $h \in H_U$. Assume then that $\sum_{i = 1}^n r_i \cdot g_i H_U = 0$ for some $r_i \in R, g_i \in G_U$. Grouping together the $g_i$ defining the same equivalence class in $H_U$, we may assume as well that $g_i H_U = g_j H_U$ for all $i, j$, which means that $\sum_{i = 1}^n r_i = 0$. In fact, for all $j > 1$, there exists an element $h_j \in H_U$ such that $g_j = g_1 h_j$. Setting $h_1 := 1$, we can then write (in $R[G_U]$)
	$$\sum_{i = 1}^n r_i \cdot g_i = \sum_{i = 1}^n r_i \cdot g_1 h_i = g_1 \cdot \left( \sum_{i = 1}^n r_i \cdot h_i \right),$$
	where for the last equality we see $R[G_U]$ as a right $R[H_U]$-module. Since the element $\sum_{i =1}^n r_i \cdot h_i$ is clearly in $I_U$, it follows that $\ker(\varphi_U) = I_U$. We have thus obtained a short exact sequence
	\begin{equation}\label{kernel is coinv, fin case}
	\begin{tikzcd}
	0 \arrow{r} & I_U \arrow{r} & R[G_U] \arrow{r}{\varphi_U} & R[G_U / H_U] \arrow{r} & 0.
	\end{tikzcd}
	\end{equation}
	Taking the inverse limit $\varprojlim_U$, we obtain now a left-exact sequence of $R$-modules
	$$\begin{tikzcd}
	0 \arrow{r} & \varprojlim_U I_U \arrow{r} & R \llbracket G \rrbracket \arrow{r}{\varphi} & R \llbracket G/H \rrbracket,
	\end{tikzcd}$$
	where $\varphi := \varprojlim_U \varphi_U$ is the obvious map. Note that for any inclusion $U \subset V$ of open normal subgroups of $G$, the transition map $I_V \to I_U$ is clearly surjective. Therefore, the system $(I_U)_U$ is a Mittag-Leffler system, which means that this is in fact a short exact sequence.
	
	To finish our proof, we just need to show that $I := \varprojlim_U I_U$ is equal to the ideal $J$ generated by the elements $h - 1$ for all $h \in H$. We clearly have equalities $f_U(J) = I_U$ for all $U$, where the $f_U : R \llbracket G \rrbracket \to R[G/U]$ are the restriction maps. This gives an isomorphism
	 $$I \xrightarrow{\sim} \varprojlim_U f_U(J) \cong \overline{J}.$$
	Write $\gamma_1, \ldots, \gamma_r$ for the topological generators of $H$. By construction, the ideal $(\gamma_1 - 1, \ldots, \gamma_r - 1)$ is the image of the following map:
	 $$\prod_{i = 1}^r R\llbracket G \rrbracket \to R \llbracket G \rrbracket, \quad (z_1, \ldots, z_r) \mapsto \sum_i z_i (\gamma_i - 1).$$
	 Since this map is continuous (as $R \llbracket G \rrbracket$ is a topological ring), the image of the compact set $\prod_i R\llbracket G \rrbracket$ inside the Hausdorff space $R \llbracket G \rrbracket$ is closed. Since every element in $J$ is a limit of elements in $(\gamma_1 - 1, \ldots, \gamma_r - 1)$, it follows that $J = \overline{J} = (\gamma_1 - 1, \ldots, \gamma_r - 1)$.
	 
	Lastly, the fact that this isomorphism commutes with extension of scalars comes from the obvious functoriality of the short exact sequences \eqref{kernel is coinv, fin case} with respect to $R$.	

\end{proof}

\begin{remark}

	Explicitly, the map $\varphi \colon R \llbracket G \rrbracket \to R \llbracket G/H \rrbracket$ appearing on the proof can be described on the level of measures as the pushforward $p_* \colon \text{Meas}(G, R) \to \text{Meas}(G/H, R)$, where $p \colon G \to G/H$ is the projection map. 

\end{remark}

We may thus fix the profinite groups over which we want to define our final measure.

\begin{notation}

	Set $G := (\Ofr \otimes_\Z \Z_p)^\times \times (\Ofr \otimes_\Z \Z_p)^\times$ and $ H(\Gamma)$ as the closure of $\Gamma$ in $G$, where we see $\Gamma$ as a subgroup of $G$ via the morphism $\gamma \mapsto (\gamma^{-1}, \gamma)$.

\end{notation}

With this in mind, \Cref{coinv of measures are measures in quot} gives an isomorphism $\text{Meas}(G, R)_{\overline{\Gamma}} \cong \text{Meas}(G/H(\Gamma), R)$. Projecting to the smaller quotient, we have a composition
\begin{equation}\label{coinv of measures comp}
\begin{tikzcd}
\text{Meas}(G,R)_{\Gamma} \arrow{r} & \text{Meas}(G,R)_{\overline{\Gamma}} \arrow{r}{\sim} & \textnormal{Meas}(G/H(\Gamma), R).
\end{tikzcd}
\end{equation}
Combining all of these maps, we obtain our measure.

\begin{definition}\label{p-adic Eisenstein measure}\index[notation]{$mu univ$@$\mu_{\Gamma, \B(f,x)}$}

	Recall that $p \neq 2$ and consider again the pair $(f,x)$ from \Cref{conditions for p-adic measure affine}. We write $\widetilde{\mu}_{\Gamma, \B}(f,x)$ for the image of $\vartheta_{\Gamma, \B}(f,x)$ under the composition
	$$\begin{tikzcd}[column sep = small]
	H^0(\widehat{\B}_p \times \widehat{\B}^\vee_p, \Ocal_{\widehat{\B}_p \times \widehat{\B}^\vee_p})_\Gamma \arrow{r}{\sim} & \text{Meas}((\Ofr \otimes_\Z \Z_p)^2, R)_\Gamma \arrow{r} & \text{Meas}(G, R)_\Gamma \arrow{r} & \text{Meas}(G/H(\Gamma), R),
	\end{tikzcd}$$
	where the maps are (from left to right): the isomorphism of \Cref{global sections are measures}, the restriction to the units $G$, and the composition \eqref{coinv of measures comp}. Then, the \emph{$p$-adic Eisenstein measure} associated to the tuple $(\B / R, \omega(\B), x)$ and to $f$ is the measure
	$$\mu_{\Gamma, \B}(f,x) \in \textnormal{Meas}(G / H(\Gamma), R)$$
	defined as follows, depending on the group $\Gamma$.
	
	\begin{enumerate}[(1)]
	
		\item Assume that $\Gamma$ is contained in $\Ofr^{\times, +}$. Then, the aforementioned measure is given as 
		$$\mu_{\Gamma, \B}(f,x) := \ts^{-\underline{1}} \cdot \widetilde{\mu}_{\Gamma, \B}(f,x)$$
		(see again \Cref{maps of measures def}). Note that the function $\ts^{-\underline{1}}$ is $\Gamma$-equivariant, and thus defined on $G/H(\Gamma)$.
	
		\item For arbitrary $\Gamma$, it is defined as the image of $[\Gamma : \Gamma \cap \Ofr^{\times, +}]^{-1}\mu_{\Gamma \cap \Ofr^{\times, +}, \B}(f,x)$ under the pushforward map
		$$\text{pr}_* \colon \text{Meas}(G/H(\Gamma \cap \Ofr^{\times, +}), R) \to \text{Meas}(G/H(\Gamma), R)$$
		induced by the projection $\text{pr} \colon G/H(\Gamma \cap \Ofr^{\times, +}) \twoheadrightarrow G/H(\Gamma)$.
	
	\end{enumerate}
	
\end{definition}

Since we have followed the same steps as Kings-Sprang, we see that we have generalized their $p$-adic measure in the CM case.

\begin{lemma}\label{both measures agree affine case}

	Assume that $R = \Ofr_{\C_p}$, that $L$ is a CM extension of $F$ together with a CM type which is $p$-ordinary, and that $\B$ has CM by $L$ (compatibly with the RM by $F$). Then, $\widetilde{\mu}_{\Gamma, \B}(f,x)$ is equal to the $p$-adic Eisenstein measure constructed by Kings-Sprang in \cite[Definition 5.21]{Kings2025}.

\end{lemma}

\begin{remark}

	As we will see later, the reason why we work with the ``twisted'' measure $\mu_{\Gamma, \B}(f,x)$ and not with $\widetilde{\mu}_{\Gamma, \B}(f,x)$ is because the one which will generalize the measure $\mu_{\Gamma, KE}$ constructed by Katz is the former and not the latter.

\end{remark}

We also obtain a result determining how this $p$-adic measure changes when we go from $\Gamma$ to a smaller, finite index subgroup when we are in the subgroup of totally positive units.

\begin{lemma}\label{measure in small subgps}

	Let $\Gamma' \leq \Gamma$ be a finite index subgroup, $(f,x)$ as above, and assume that $\Gamma \leq \Ofr^{\times, +}$. Then, we have that
	$$\textnormal{pr}_{\Gamma', *}\mu_{\Gamma', \B}(f,x) = [\Gamma : \Gamma']\mu_{\Gamma, \B}(f,x),$$
	with $\textnormal{pr}_{\Gamma'} \colon G/H(\Gamma') \twoheadrightarrow G/H(\Gamma)$ the obvious projection map. 

\end{lemma}

\begin{proof}

	This is a direct consequence of \Cref{theta functs in smaller subgp}.

\end{proof}

\subsection{\textit{p}-adic interpolation}

As computed by Kings-Sprang in \cite{Kings2025}, the $p$-adic measures $\mu_{\Gamma, \B}(f,x)$ interpolate the Eisenstein-Kronecker classes. We now use this result in order to show that these measures coincide with the Katz Eisenstein measures $\mu_{\Gamma, KE}(f)$.

\subsubsection{Interpolation and Eisenstein-Kronecker classes}

We start by setting some notation.

\begin{notation}

	Write 
	$$B:= \B \times_R \text{Spec}(K), \quad B^\vee := \B^\vee \times_R \text{Spec}(K)$$ 
	for the generic fibers of $\B$ and $\B^\vee$. Set as well $\omega(B)$, $\omega(B^\vee)$ to be the images of the differential forms $\omega(\B)$ and $\omega(\B^\vee)$ under the pullback maps $B \to \B$ and $B^\vee \to \B^\vee$.

\end{notation}

Recall from \Cref{quot by can subgps general ring def} the abelian schemes $\B_{(n)}/R$, $\B_{(n)}^\vee/R$ and the étale isogenies
$$V^n_\B \colon \B_{(n)} \to \B, \quad V^n_{\B^\vee} \colon \B_{(n)}^\vee \to \B^\vee.$$
Taking the pullback of $\D$ and of $f$ under $V^n_\B$, we can write (using that $\bfr$ is coprime to $p$):
$$\D' = \B_{(n)}[\bfr] \subset \B_{(n)}, \quad f' \in R[\D']^{0, \Gamma}.$$
As before, we can take the generic fibers of these abelian schemes to obtain abelian varieties $B_{(n)}/K$ and $B_{(n)}^\vee / K$, as well as isogenies $V^n_B$, $V^n_{B^\vee}$. 

Next, we will use a result of Kings-Sprang to show that the value of this $p$-adic measure at $\Gamma$-equivariant functions of the form $\ts^{\alpha + \underline{1}}\st^\beta H$, with $H$ some locally constant function supported on $G$, gives back a formula relating these moments to the $p$-adic specialization of the Eisenstein-Kronecker classes and to the partial Fourier transform of $H$ from \Cref{partial Fourier alg def}.

\begin{proposition}\label{measure at loc const functs general}

	Let $H$ be a locally constant function $(\Ofr \otimes_\Z \Z_p)^2 \to R$ which is supported on the units $G$ and constant modulo $p^n$. Assume that $\Gamma$ fixes the points of $B_{(n)}[p^n]$, and write $\omega(B_{(n)}) := (V^n_B)^*(\omega(B))$. Then, for any $\alpha, \beta \in I_F^+$ such that $\ts^{\alpha + \underline{1}}\st^\beta H$ is $\Gamma$-equivariant, we have an equality
	$$\frac{1}{\Omega_p^{\alpha + \underline{1} + \beta}} \int_{G/H(\Gamma)} \ts^{\alpha + \underline{1}} \st^\beta H d\mu_{\Gamma, \B}(f,x) = $$
	$$ = \sum_{s' \in B_{(n)}[p^n]} P_BH(s') \varphi_{p, B_{(n)}}(EK^{\beta, \alpha}_{\Gamma, B_{(n)}}(f', x' + s'))(\omega(B_{(n)})^{[\alpha + \underline{1}]} \otimes \omega(B_{(n)})^{[\beta]}),$$
	where $x'$ is a lift of $x$ via the Verschiebung map.

\end{proposition} 

\begin{proof}

	This follows by repeating the same proof as \cite[Theorem 5.23, Corollary 5.25]{Kings2025}, as the arguments still hold with the weaker assumption that $\B$ has real and not complex multiplication. Note as well that the formula in the mentioned work is for
	$$\int_{G/H(\Gamma)} \ts^\alpha \st^\beta H d\widetilde{\mu}_{\Gamma, \B}(f,x)$$
	whenever $\ts^{\alpha}\st^\beta H$ is $\Gamma$-equivariant. If we place ourselves in the case where $\Gamma$ is contained in $\Ofr^{\times, +}$, then this holds, and moreover $\mu_{\Gamma, \B}(f,x) = \ts^{-\underline{1}} \cdot \widetilde{\mu}_{\Gamma, \B}(f,x)$, which combined gives the desired equality. This then implies the result for arbitrary $\Gamma$ by definition of the $\mu_{\Gamma, \B}(f,x)$ together with \cite[Corollary 2.29]{Kings2025}.

\end{proof}

\begin{remark}

	As it was observed in \cite[p.85]{Kings2025}, this equality shows in particular that the right-hand side is defined integrally, although a priori we may only give a rational construction on the generic fiber $B_{(n)}$. Indeed, the isogeny $[p^n]$ does not have an étale dual over $R$, but it does become étale over $K$.

\end{remark}

Lastly, we can combine this explicit expression with \Cref{p-adic specialization is Katz Eisenstein series} to compare the value of our measure at a locally constant function with the $p$-adic Eisenstein series $\widehat{G}_{k, Hf}$. Note that the following result also holds for $\Gamma$ which don't fix the $p^n$-torsion subgroup of these abelian schemes.

\begin{corollary}\label{measure is Katz Eisenstein series affine}

	Let $R = \Ofr_{\C_p}$. Consider a locally constant function $H \colon (\Ofr \otimes_\Z \Z_p)^2 \to \Afr_p$ which is supported in $G$ and is of parity $k$. Fix as well a function $f \colon (\Ofr/l\Ofr)^2 \to \Afr_p$, and assume that $\Gamma$ fixes $\B_{(n)}[l]$. We then have an equality
	$$\frac{(-1)^{\frac{g(g - 2k - 1)}{2}}}{\sqrt{d_F}}\int_{G/H(\Gamma)} \ts^{\underline{k}}H d\mu_{\Gamma, \B}(f,e) =  \widehat{G}_{k,Hf, \Gamma}(B).$$

\end{corollary} 

\begin{proof}

	Let $\Gamma' := \Gamma \cap \Ofr^{\times, +}$ and $\Gamma'_n \leq \Gamma'$ be a finite index subgroup which fixes $B_{(n)}[p^n]$. We claim that it is enough to show this equality for $\Gamma'_n$. Indeed, using our definition of $\mu_{\Gamma, \B}(f, e)$ together with \Cref{measure in small subgps}, we see that
	$$\int_{G/H(\Gamma)} \ts^{\underline{k}}H d\mu_{\Gamma, \B}(f,e) = \frac{1}{[\Gamma : \Gamma'_n]} \int_{G/H(\Gamma'_n)} \ts^{\underline{k}}H d\mu_{\Gamma'_n, \B}(f,e).$$
	On the other hand, \Cref{Katz Eis series in smaller subgps} shows that $\widehat{G}_{k,Hf, \Gamma} = [\Gamma : \Gamma'_n]^{-1}\widehat{G}_{k,Hf, \Gamma'_n}$, which proves our claim.
	
	Then, the formula of \Cref{measure at loc const functs general} shows that
	$$\frac{1}{\Omega_p^{\underline{k}}} \int_{G/H(\Gamma'_n)} \ts^{\underline{k}} H d\mu_{\Gamma'_n, \B}(f,e) = \sum_{s' \in B_{(n)}[p^n]} P_BH(s') \varphi_{p, B_{(n)}}(EK^{\underline{0}, \underline{k} - \underline{1}}_{\Gamma'_n, B_{(n)}}(f, s'))(\omega(B_{(n)})^{[\underline{k}]}).$$
	Since the pair $(f,e)$ can be obtained via pullback from the universal abelian scheme, our desired equality thus follows by evaluating the expression of \Cref{p-adic specialization is Katz Eisenstein series} at $B$. Note that we need to pass to the finite index subgroup $\Gamma'_n$, since if there are elements of norm $-1$ in $\Gamma$, then $\underline{k}$ may not be of critical type for $\Gamma$, case in which we cannot define the Eisenstein-Kronecker classes for the pair $(\alpha, \beta) = (\underline{k} - \underline{1}, \underline{0})$. 

\end{proof}

If we assume that $l = 1$, so that we don't have any $\Gamma(l)$-level structure, we again obtain a comparison result by using the function $\widetilde{f}_{\B, [b_1]} \in R[\D]^{0,\Gamma}$ which we constructed before. 

\begin{corollary}\label{measure is Katz Eisenstein series no full}

	Let $R = \Ofr_{\C_p}$, and let $H \colon (\Ofr \otimes_\Z \Z_p)^2 \to \Afr_p$ be as above. For any pair of integers $b_1$, $b_2$ defining ideals $\bfr_i := b_i\Ofr$ as in \Cref{b f ideals notation} which are coprime to $p$ we have an equality
	$$\frac{(-1)^{\frac{g(g - 2k - 1)}{2}}}{\sqrt{d_F}}\int_{G/H(\Gamma)} \ts^{\underline{k}}H d\mu_{\Gamma, \B}(\widetilde{f}_{\B, [b_1]},e) = $$
	$$ = (b_1^gb_2^{kg} \widehat{G}_{k,[b_2]^*H, \Gamma} - b_1^g\widehat{G}_{k,H, \Gamma} + b_1^{kg}\widehat{G}_{k, [b_1]^*H, \Gamma} - (b_1b_2)^{kg}\widehat{G}_{k, [b_1b_2]^*H, \Gamma})(B),$$
	with $\widetilde{f}_{\B, [b_1]}\in \Ofr_{\C_p}[\D]^{0,\Gamma}$ the function from \Cref{f for comparison with Katz}.

\end{corollary}

\begin{proof}

	The proof goes in the same way as \Cref{measure is Katz Eisenstein series affine} by instead using \Cref{p-adic specialization is Katz no full}. We remark that one needs to check that the pullback of $\widetilde{f}_{\B, [b_1]}$ under the maps
	$$\begin{tikzcd}
	B_{(n)} \arrow{r}{V^n_B} & B \arrow{r} & \B.
	\end{tikzcd}$$
	is equal to $\widetilde{f}_{B_{(n)}, [b_1]}$, but this follows from its construction.

\end{proof}

\subsubsection{The universal measure and its comparison to Katz's measure}

Since we now want to extend the measures $\mu_{\Gamma, \B}(f,x)$ to the universal case, we return to our moduli setup. 

\begin{notation}

	Let $R_0$ be a $p$-adic ring which is a $\Ofr_{F^{\text{Gal}}}[d_F^{-1}]$-algebra and such that $\cfr$ is coprime to $R_0$. Write $\Mcal := \Mcal(\cfr, \Gamma(l), \Gamma_{00}(p^\infty))_{R_0}$ for the Hilbert moduli scheme over $R_0$ of $(\Gamma(l), \Gamma_{00}(p^\infty))$-level, and $\A \to \Mcal$ for the universal abelian scheme. Recall our notation for the formal completions $\widehat{\Mcal}$, $\widehat{\A}_p$ from \Cref{formal compl Hilbert mod scheme notation} and \Cref{formal compl special fiber notation}.

\end{notation}

For any profinite group $G'$, consider the following sheaf of $\Ocal_{\widehat{\Mcal}}$-modules:
$$\text{Meas}(G', \Ocal_{\widehat{\Mcal}}) := \varprojlim_{U} (\Ocal_{\widehat{\Mcal}} \otimes_{R_0} R_0[G'/U]),$$
where $U$ goes over the open normal subgroups of $G'$. Set $G' = G/H(\Gamma)$. Then, for any point $\Spf(R) \to \widehat{\Mcal}$ defining a $(\Gamma(l), \Gamma_{00}(p^\infty))$-test object $\B/R$, we can see the measure $\mu_{\Gamma, \B}(f,x)$ as an element of $\text{Meas}(G/H(\Gamma), \Ocal_{\widehat{\Mcal}})(\Spf(R))$. This allows us to glue these measures to obtain a measure taking values in the ring
$$V := V(\cfr, \Gamma(l), R_0) = H^0(\widehat{\Mcal}, \Ocal_{\widehat{\Mcal}})$$
of $p$-adic $\cfr$-Hilbert modular forms over $R_0$.

\begin{definition}

	Let $(f,x)$ be a pair associated to $\A$ as in \Cref{pairs f x of EK classes general} for the morphisms $\varphi = [\ffr]$ and $\delta = [\bfr]$, and assume that $\Gamma$ acts trivially on $\Ofr/l\Ofr$. The $p$-adic measure
	$$\mu_{\Gamma, \A}(f,x) \in \text{Meas}(G/H(\Gamma), \Ocal_{\widehat{\Mcal}})(\widehat{\Mcal}) = \text{Meas}(G/H(\Gamma), V)$$
	obtained by gluing the measures $\mu_{\Gamma, \B}(f_\B, x_\B)$ for some finite cover by formal affines of $\widehat{\Mcal}$ is the \emph{universal Eisenstein series associated to $(f,x)$}.

\end{definition}

\begin{remark}

	Note that by construction, this $p$-adic measure is uniquely determined by the following property: for any $(\Gamma(l), \Gamma_{00}(p^\infty))$-test object $\B / R_0$, the map
	$$\textnormal{Meas}(G / H(\Gamma), V) \to \textnormal{Meas}(G / H(\Gamma), R_0)$$ 
	given by evaluation at $\B$ sends $\mu_{\Gamma, \A}(f, x)$ to $\mu_{\Gamma, \B}(f_\B, x_\B)$.

\end{remark}

To finish, we will show how our measure recovers the measures $\mu_{KE}(f)$ that we defined in \Cref{Katz Eisenstein measure - Katz def}, by using the explicit description from \Cref{measure is Katz Eisenstein series affine}. 

\begin{theorem}\label{measures agree}

	Let $R_0 = \Ofr_{\C_p}$, and $f \colon (\Ofr/l\Ofr)^2 \to \Afr_p$ be a function such that $\sum_{t \in (\Ofr/l\Ofr)^2} f(t) = 0$. Moreover, assume that $\Gamma$ fixes $\Ofr/l\Ofr$. Then, we have an equality
	$$\frac{(-1)^{\frac{g(g+1)}{2}}}{\sqrt{d_F}} \cdot \mu_{\Gamma, \A}(f, e) = (\textnormal{inv} \times \textnormal{id})_*\mu_{\Gamma, KE}(f),$$
	where $\textnormal{inv} \colon (\Ofr \otimes_\Z \Z_p)^\times \to (\Ofr \otimes_\Z \Z_p)^\times$ is the inversion map $x \mapsto x^{-1}$.

\end{theorem}

\begin{proof}
	
	For simplicity, write $\mu_\Gamma := \mu_{\Gamma, \A}(f,e)$ and $\mu'_{\Gamma} := (\textnormal{inv} \times \textnormal{id})_*\mu_{\Gamma, KE}$. Note that $\text{inv} \times \text{id}$ induces an isomorphism $G/\overline{\Gamma} \xrightarrow{\sim} G / H(\Gamma)$, and thus $\mu'_\Gamma$ is a measure in $\text{Meas}(G/H(\Gamma), V)$. Moreover, for any finite index subgroup $\Gamma' \leq \Gamma$, we have equalities
	$$\text{pr}_{\Gamma', *}\mu_{\Gamma'} = [\Gamma : \Gamma']\mu_\Gamma, \quad \text{pr}_{\Gamma', *}\mu'_{\Gamma'} = [\Gamma : \Gamma']\mu'_\Gamma,$$
	with $\text{pr}_\Gamma \colon G/H(\Gamma') \twoheadrightarrow G/H(\Gamma)$ the projection map. We may thus assume that $\Gamma$ is contained in $\Ofr^{\times, +}$ for the rest of this proof, and we omit it from the notation in these measures, writing now $\mu$ and $\mu'$.
	
	In order to show that the two measures are equal, we will compare them at all functions $\ts^{\underline{1}}H$, where $H$ is of the following form. 
	
	\begin{itemize}
	
		\item[($\star$)] Let $H \colon G \to \Ofr_{\C_p}$ be a $\Gamma$-equivariant and locally constant function whose image is contained in $\Afr_p$. 
		
	\end{itemize}
	Therefore, the product $\ts^{\underline{1}}H$ defines a continuous, $\Gamma$-equivariant function $G \to \Ofr_{\C_p}$.
	
	Since $V$ is an $\Ofr_{\C_p}$-algebra, we may use the isomorphism of \Cref{measure on intermediate algebra} to write
	$$\text{Meas}(G/H(\Gamma), V) \cong \Hom_{\Ofr_{\C_p}}(C(G/H(\Gamma), \Ofr_{\C_p}), V).$$
	We may thus characterize our measures by their values at continuous functions $G/H(\Gamma) \to \Ofr_{\C_p}$. We then show that, for any function $H$ as in ($\star$), 
	$$\int_{G/H(\Gamma)} \ts^{\underline{1}}H d\mu \quad \text{and} \quad \int_{G/H(\Gamma)} \ts^{\underline{1}}H d\mu'$$ 
	are equal up to the determined constants and pullbacks. Since these two integrals are valued in the ring of $p$-adic Hilbert modular forms $V$, it is enough to compare the value of both of them at every $(\Gamma(l), \Gamma_{00}(p^\infty))$-test object $\B$ defined over $\Ofr_{\C_p}$.
	
	By construction, for the measure $\mu$ we may just compute directly the value of the measure $\mu_{\Gamma, \B}(f_\B, e_\B)$ at $\ts^{\underline{1}}H$, and thus using \Cref{measure is Katz Eisenstein series affine} we find that
	$$\frac{(-1)^{\frac{g(g+1)}{2}}}{\sqrt{d_F}}\left( \int_{G / H(\Gamma)} \ts^{\underline{1}}H d\mu \right)(\B) = \widehat{G}_{1,Hf, \Gamma}(B).$$
	On the other hand, by \Cref{Katz Eisenstein measure - Katz def}, we have that for any continuous $\Gamma$-equivariant function $H' \colon G \to \Ofr_{\C_p}$,
	$$\int_{G/\overline{\Gamma}} H'd \mu_{\Gamma, KE}(f) = \widehat{G}_{1, \ts^{-\underline{1}}(\text{inv} \times \text{id})^*H'f, \Gamma}.$$
	Therefore, by definition of the pushforward $(\text{inv} \times \text{id})_*$, we can see that
	$$\left( \int_{G/H(\Gamma)} \ts^{\underline{1}}H d\mu' \right)(\B)= \widehat{G}_{1, Hf, \Gamma}(\B).$$
	
	This shows our desired equality at all functions of the form $\ts^{\underline{1}}H$, with $H$ as in ($\star$). We now claim that these functions are dense in $C(G/H(\Gamma), \Ofr_{\C_p})$, which would show the equality between $\mu$ and $\mu'$. First, we clearly have that the map given by multiplication by $\ts^{\underline{1}}$,
	$$\ts^{\underline{1}} \cdot (-) \colon C(G/H(\Gamma), \Ofr_{\C_p}) \to C(G/H(\Gamma), \Ofr_{\C_p}),$$
	is a homeomorphism. Then, since the locally constant functions are dense in the topological space $C(G/H(\Gamma), \Ofr_{\C_p})$, the set of functions 
	$$\left\{ \ts^{\underline{1}} \cdot H_0 \mid H_0 \colon G/H(\Gamma) \to \Ofr_{\C_p}\ \text{locally constant} \right\}$$ 
	is also dense. Lastly, since $\Afr_p \subset \Ofr_{\C_p}$ is dense, we can also approximate continuous functions in $\Ofr_{\C_p}$ by those with values in $\Afr_p$, which finishes our proof.
	
\end{proof}

Despite the fact that this result gives back some of the generalized measures $\mu_{\Gamma, KE}(f)$, this does not directly recover an equality between our universal measure $\mu_{\Gamma, \A}(f,e)$ and the measure $\mu_{\Gamma, KE}$ of \cite{Katz1978}. Indeed, for this we would have to take $f$ to be the characteristic function of $([0], [0])$, which does not satisfy the trace zero hypothesis. Instead, to make this comparison we will set $l = 1$ and use \Cref{measure is Katz Eisenstein series no full}.

Recall that for any integer $m$ which was coprime to $p$, we had a map
$$[m] \colon (\Ofr \otimes_\Z \Z_p)^2 \to (\Ofr \otimes_\Z \Z_p)^2, \quad (x,y) \mapsto (mx, y/m).$$
Pullback by this map induced the functions $[m]^*H$, as well as a pushforward map of measures (see \Cref{maps of measures def}):
$$[m]_* \colon \text{Meas}(G/H(\Gamma), R) \to \text{Meas}(G/H(\Gamma),R),$$
which is precisely given by $\int_{G/H(\Gamma)} H d([m]_*\mu) := \int_{G/H(\Gamma)} ([m]^*H)d\mu$. 

\begin{theorem}\label{measures agree no full}

	Let $R_0 = \Ofr_{\C_p}$. Fix principal ideals $\bfr_1 = b_1\Ofr, \bfr_2 = b_2\Ofr$ with $b_1, b_2 \in \Z$ as in \Cref{b f ideals notation}. If we choose $\widetilde{f}_{[b_1]} = \widetilde{f}_{\B,[b_1]}$ as in \Cref{f for comparison with Katz} for the ideals $\bfr_1$, $\bfr_2$, we have an equality
	$$\frac{(-1)^{\frac{g(g+1)}{2}}}{\sqrt{d_F}} \cdot \mu_{\Gamma, \A}(\widetilde{f}_{[b_1]},e) = b_1^g(1 - [b_1]_*)(b_2^g[b_2]_* - 1)(\textnormal{inv} \times \textnormal{id})_*\mu_{\Gamma, KE}.$$

\end{theorem}

\printindex[notation]

\bibliographystyle{alpha}
\bibliography{Thesis} 

\end{document}